\documentclass[11pt]{article}
\usepackage[utf8]{inputenc}

\usepackage{amsmath,amsthm,amssymb}
\usepackage{graphicx}
\usepackage{wrapfig}
\usepackage{enumitem}
\usepackage{tikz-cd}
\usepackage{esvect}
\usepackage{bm}
\usepackage{tocloft}
\usepackage{lipsum} 
\usepackage{mathrsfs}
\usepackage{float}
\usepackage{graphicx}
\usepackage{commath}
\usepackage{mathtools}
\usepackage{ulem}
\usepackage[nottoc]{tocbibind} 
\usepackage{csquotes}
\usepackage{xurl}

\usepackage{microtype}
\usepackage{caption}
\usepackage{subcaption}

\usetikzlibrary{backgrounds, matrix,fit,matrix,decorations.pathreplacing, calc, positioning}

\makeatletter
\newcommand\setItemnumber[1]{\setcounter{enum\romannumeral\@enumdepth}{\numexpr#1-1\relax}}
\makeatother

\newcommand{\R}{\mathbb{R}}
\newcommand{\C}{\mathbb{C}}

\newcommand{\e}{\varepsilon}

\newcommand{\beq}{\begin{equation}}
\newcommand{\eeq}{\end{equation}}
\newcommand{\bdef}{\begin{definition}}
\newcommand{\eedef}{\end{definition}}
\newcommand{\bremark}{\begin{remark}}
\newcommand{\eremark}{\end{remark}}

\newtheorem{theorem}{Theorem}[section]
\newtheorem{lemma}[theorem]{Lemma}
\newtheorem{proposition}[theorem]{Proposition}

\newtheorem{corollary}[theorem]{Corollary}
\newtheorem{definition}[theorem]{Definition}
\newtheorem{claim}[theorem]{Claim}
\newtheorem{conjecture}[theorem]{Conjecture}

\theoremstyle{definition}
\newtheorem{example}[theorem]{Example}
\newtheorem{remark}[theorem]{Remark}

\makeatletter
\newtheorem*{rep@theorem}{\rep@title}
\newcommand{\newreptheorem}[2]{%
\newenvironment{rep#1}[1]{%
 \def\rep@title{#2 \ref{##1}}%
 \begin{rep@theorem}}%
 {\end{rep@theorem}}}
\makeatother

\newreptheorem{theorem}{Theorem}
\newreptheorem{lemma}{Lemma}
\newreptheorem{proposition}{Proposition}

\PassOptionsToPackage{hyphens}{url}\usepackage{hyperref} 
\hypersetup{
  colorlinks=true,
  citecolor=blue,
  linkcolor=blue, 
  urlcolor=cyan
}

\def\?{{\bf ****???????****}}
\def\M{{\mathcal M}}

\def\lb{\label}

\def\MA{Monge--Amp\`ere }

\newcommand{\defeq}{\vcentcolon=}

\long\def\frame#1#2#3#4{\hbox{\vbox{\hrule height#1pt
 \hbox{\vrule width#1pt\kern #2pt
 \vbox{\kern #2pt
 \vbox{\hsize #3\noindent #4}
\kern#2pt}
 \kern#2pt\vrule width #1pt}
 \hrule height0pt depth#1pt}}}

\def\calC{\mathcal{C}}

\def\?{{\bf ****???????****}}
\def\M{{\mathcal M}}
\def\q{\quad}
\def\qq{\qquad}
\def\eps{\varepsilon}
\def\ra{\rightarrow}
\def\Kppolar{K^{\circ,p}}

\usepackage{xcolor}
\usepackage{xpatch}
\colorlet{partnumbercolour}{blue}
\makeatletter
\AtBeginDocument{
\xpatchcmd{\@part}{\partname\nobreakspace\thepart}{\textcolor{partnumbercolour}{\thepart}}{}{}
}
\makeatother

\title{$L^p$-polarity, Mahler volumes, and the isotropic constant}
\author{Bo Berndtsson, Vlassis Mastrantonis, Yanir A. Rubinstein }

\begin{document}
\let\i\undefined
\newcommand{\i}{\sqrt{-1}}
\def\LpK{K^{\circ,p}}
\def\LpS{S^{\circ,p}}
\def\Cov{\mathrm{Cov}}
\numberwithin{equation}{section}

\maketitle
\date

\begin{abstract}
This article introduces $L^p$ versions of the support function of a convex body $K$
and associates to these canonical $L^p$-polar bodies $K^{\circ, p}$
and Mahler volumes $\mathcal{M}_p(K)$. Classical polarity is then seen
as $L^\infty$-polarity. This one-parameter 
generalization of polarity leads to a generalization of the Mahler conjectures,
with a subtle advantage over the original conjecture: conjectural uniqueness of extremizers
for each $p\in(0,\infty)$.
We settle the upper bound
by demonstrating the existence and uniqueness of an $L^p$-Santal\'o point and an 
$L^p$-Santal\'o inequality for symmetric convex bodies. The proof uses Ball's
Brunn--Minkowski inequality for harmonic means, the classical Brunn--Minkowski inequality, symmetrization, and a systematic study of the $\mathcal{M}_p$ functionals.
Using our results on the $L^p$-Santal\'o point and a new observation motivated by
complex geometry, we show how Bourgain's slicing conjecture can be reduced to
lower bounds on the $L^p$-Mahler volume coupled with a certain conjectural convexity property of the logarithm of the \MA measure of the $L^p$-support function.
We derive a suboptimal version of this convexity using Kobayashi's theorem on the Ricci curvature of Bergman metrics to illustrate this approach to slicing.
Finally, we explain how Nazarov's complex analytic approach to the classical Mahler conjecture
is instead precisely an approach to the $L^1$-Mahler conjecture. 
\end{abstract}

\tableofcontents

\section{Introduction}
\lb{IntroSec}

The polar $K^\circ$ and the support function $h_K$ 
of a convex body $K$ are fundamental objects in Functional and Convex Analysis.
The Mahler and Bourgain Conjectures have motivated an enormous amount of 
research in those fields over the past 85 years. One of the goals of this article
is to point out that $K^\circ$ and $h_K$ are $L^\infty$-versions of a more
general one-parameter family of objects
$$
K^{\circ, p}
$$
and
$$
h_{p,K},
$$
introduce the associated one-parameter generalization of the Mahler volume
$\M_p$
and Conjectures,
and establish some of their fundamental properties.
As we explain in detail and back up with explicit computations, minimizers
should be unique (see Figure~\ref{FigureDiamondCube} and the discussion surrounding it). This is a subtle, but perhaps crucial, advantage, as compared to Mahler's original conjecture. To quote Tao \cite{TaoBlog} (see also  B\l{}ocki \cite[p. 90]{blocki2}),
\begin{displayquote}
\textit{``In my opinion, the main reason why this conjecture is so difficult is that unlike the upper bound, in which there is essentially only one extremiser up to affine transformations (namely the ball), there are many distinct extremisers for the lower bound..."}
\end{displayquote}

As an application of the theory of $L^p$-polarity, we 
develop a connection between these new objects ($L^p$-support
functions and $L^p$-Mahler volumes)
and Bourgain's slicing conjecture, e.g., making contact with Kobayashi's 
theorem on the Ricci curvature of Bergman metrics. 
Finally, we explain how Nazarov's and B\l{}ocki's work on a complex analytic approach to the classical Mahler conjecture fits in, being
precisely an approach to the $L^1$-Mahler conjecture. 

Our approach is loosely motivated by complex geometry,
but the article in its entirety can be read with no knowledge of complex methods.
As is probably
clear from the text, the authors are novices in the study of the Mahler and Bourgain conjectures and apologize for any omission in accrediting results properly. 
The motivation for this article lies not so much in the particular results as in showing the link between complex geometry and this beautiful area. 
It should also be stressed
that the list of references is far from complete. We have tried to make the text
accessible to both convex and complex analysts and so perhaps included a bit more
background than usual.

\subsection{Motivation from Bergman kernels}
\lb{BergmanSubSec}

Denote by 
\beq
\lb{polarEq}
K^\circ\defeq \{y\in \R^n: \langle x,y\rangle\leq 1 \text{ for all } x\in K\}
\eeq 
the polar body associated to a convex body $K\subset \R^n$.
A key step in Nazarov's
complex-analytic approach
 to the Bourgain--Milman inequality \cite[Theorem 1]{bourgain-milman} is 
 a bound on the {\it Mahler volume} 
\begin{equation}
\lb{MahlerEq}
    \M(K)\defeq n! |K| |K^\circ|, 
\end{equation}
of a symmetric (i.e., $-K=K$)
convex body $K$ 
from below by 
a multiple of the Bergman kernel
$\mathcal{K}_{T_K}(z,w)$ of the tube domain 
$T_K\defeq \R^n+ \i K$ 
over $K$, evaluated on the diagonal at the origin \cite[p. 338]{nazarov}. 
This was generalized by Hultgren \cite[Lemma 11]{hultgren} and two of us \cite[Proposition 6]{MR} to any convex body $K$: 
\begin{equation}\label{0.1}
    \pi^n|K|^2 \mathcal{K}_{T_K}(\i b(K), \i b(K))\leq \M(K-b(K)),
\end{equation}
where
\begin{equation*}
    b(K)\defeq \int_K x\frac{\dif x}{|K|}
\end{equation*}
is 
the barycenter of $K$.

This article, however, is not about Bergman kernels (though 
we come back to Bergman kernels in \S\ref{NazarovSubSec} and \S\ref{complexgeom_subsec}). Nonetheless, the {\it $L^p$-Mahler volumes} introduced below are partly motivated by (\ref{0.1}). In order to prove (\ref{0.1}) one uses Jensen's inequality together with an explicit formula for the Bergman kernels of tube domains evaluated on the diagonal, due to Rothaus \cite[Theorem 2.6]{rothaus}, Korányi \cite[Theorem 2]{koranyi}, and Hsin \cite[(1.2)]{hsin}, that as observed recently 
can be expressed as \cite[Remark 36]{MR} 
\begin{equation}\label{diagonal_berg}
    \mathcal{K}_{T_K}(0, 0)= \frac{1}{(4\pi)^n} \int_{\R^n} e^{-h_{1,K}(y)} \frac{\dif y}{|K|}, 
\end{equation}
where, following \cite[Definition 13]{MR}, we denote by
\begin{equation}
\lb{h1KEq}
    h_{1,K}(y)\defeq \log\int_{K} e^{\langle x,y\rangle}\frac{\dif x}{|K|}, 
\end{equation}
the logarithmic Laplace transform of the convex indicator function $\bm{1}_K^\infty$ ($\bm{1}_K^\infty$ is $0$ on $K$ and $\infty$ otherwise). 
Therefore, the left-hand side of (\ref{0.1}) becomes $\pi^n |K|\int_{\R^n} e^{-h_{1,K-b(K)}(y)}\dif y$, bearing a curious resemblance to the standard formula for the Mahler volume
\eqref{MahlerEq},
\begin{equation}\label{0.2}
    \M(K)= |K|\int_{\R^n} e^{-h_K(y)}\dif y
\end{equation}
(see \eqref{KcircVolume} below),
where 
\begin{equation}
\lb{hinftyKEq}
h_K(y):=\sup_{x\in K}\langle x,y\rangle
\end{equation}
is the (classical) support function of $K$.

\subsection{\texorpdfstring{$L^p$-support function, -polarity, and -Mahler volume}{Lᵖ-support function, -polarity, and -Mahler volume}}

Motivated by the preceding
discussion and \cite[Remark 36]{MR}, we introduce the {\it $L^p$-support function} of a compact body $K\subset \R^n$ for all $p>0$,
\begin{equation}\label{hpdef}
    h_{p,K}(y)\defeq \log\left(\int_{K} e^{p\langle x,y\rangle}\frac{\dif x}{|K|}\right)^{\frac1p}, \quad y\in \R^n, 
\end{equation}
unifying and interpolating between \eqref{h1KEq} and \eqref{hinftyKEq}
(notice that $h_{\infty,K}:=\lim_{p\ra\infty}h_{p,K}=h_K$ by Corollary \ref{hpLimit}). 
These are convex functions in $y$, monotone increasing in $p$, and take the Cartesian product of bodies to the sum of the respective $L^p$-support functions (Lemma \ref{list}). Less obviously, they also enjoy a convexity property in $p$ (Lemma \ref{conv1}), and a ``concavity" property in $K$ (Lemma \ref{conv2}).

Generalizing (\ref{0.2}), we introduce the {\it $L^p$-Mahler volume},
\begin{equation}\label{Mpdef}
    \M_p(K)\defeq |K|\int_{\R^n} e^{-h_{p,K}(y)}\dif y.
\end{equation}
The functional $\M_p$ shares many (but not all) of the properties of $\M=\M_\infty$ (by Corollary \ref{hpLimit}), e.g., invariance under the action of $GL(n,\R)$ (Lemma \ref{Mp_aff_inv}), tensoriality (Remark \ref{tensor}), existence and uniqueness of a Santal\'o point (Proposition \ref{santalo_point_thm}), and a Santal\'o inequality for symmetric bodies (Theorem \ref{santalo_symmetric}).

It is natural to ask whether there is an analogue of \eqref{MahlerEq} for $\M_p$,
i.e., is there a canonically associated body to $K$ for which $\M_p$ can be 
expressed as the volume of a product body in $\R^{2n}$?
We answer this affirmatively. To that end, we introduce the following:

\bdef
\label{LpKdef}
Let $K\subset\R^n$. Define the {\rm $L^p$-polar body of $K$} by 
\begin{equation}
\lb{KpcircEq}
K^{\circ, p}\defeq \bigg\{y\in\R^n: \int_0^\infty r^{n-1} e^{-h_{p,K}(ry)}\dif r\ge(n-1)!\bigg\}. 
\end{equation}
\eedef

Our first result answers the aforementioned question. 
\begin{theorem}
\label{LpKwelldefined}
    Let $p\in(0,\infty]$. For a convex body $K\subset \R^n$,
    $K^{\circ,p}$ is convex, closed, has non-empty interior, and
    \begin{equation}\label{MpKeq}
        \M_p(K)= n! |K||K^{\circ,p}|. 
    \end{equation}
    It is compact (bounded) if and only if $0\in \mathrm{int}\, K$. For $K$ symmetric, $K^{\circ,p}$ is symmetric.
\end{theorem}

Theorem \ref{LpKwelldefined} justifies the notation
\begin{equation}\label{opNorm}
    \|y\|_{K^{\circ,p}}\defeq \left( \frac{1}{(n-1)!}\int_0^\infty r^{n-1} e^{-h_{p,K}(ry)}\dif r\right)^{-\frac1n},
\end{equation}
(the power serves to homogenize)
and $K^{\circ,p}=\{y\in\R^n: \|y\|_{K^{\circ, p}}\leq 1\}.$
For $p=\infty$ one recovers the usual polar body, i.e., $K^{\circ,\infty}= K^{\circ}$ (Lemma \ref{LpKpcont}). The case $p=0$ is treated in \S\ref{casep0SubSec}. 
Figure~\ref{mainFigure} illustrates some explicit examples.

{\allowdisplaybreaks
\begin{figure}[H]
  \begin{subfigure}[b]{0.49\textwidth}
  \centering
    \includegraphics[scale=.49]{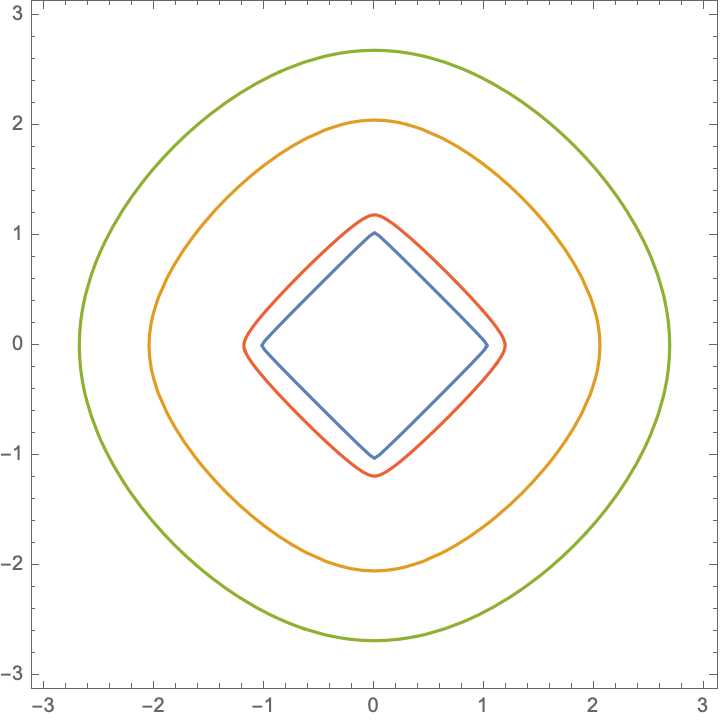}
    \caption{\small $(B_\infty^2)^{\circ,p}$}
    \label{fig:f1}
  \end{subfigure}
    \begin{subfigure}[b]{0.5\textwidth}
    \centering
    \includegraphics[scale=.49]{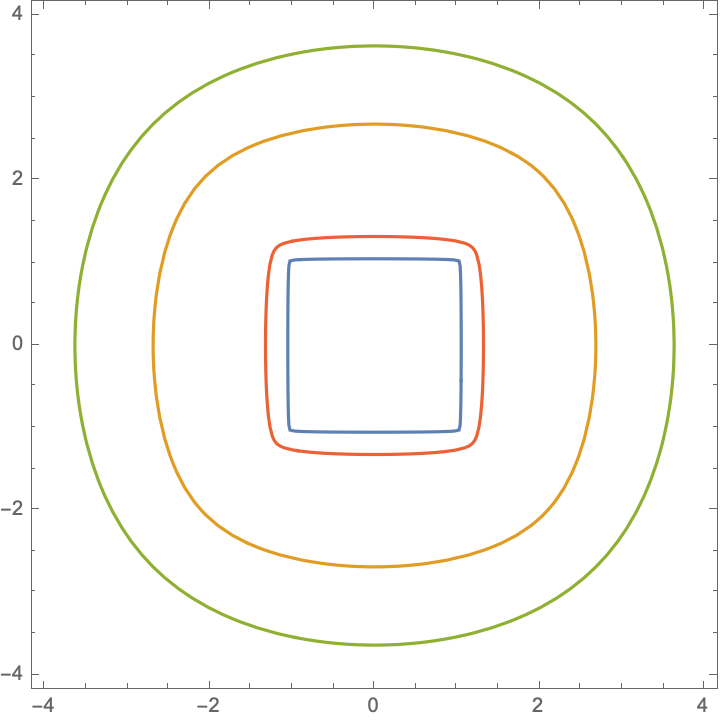}
    \caption{\small $(B_1^2)^{\circ,p}$}
    \label{fig:f2}
  \end{subfigure}
\end{figure}
\begin{figure}[ht]\ContinuedFloat
\centering
      \begin{subfigure}[b]{0.5\textwidth}
  \centering
    \includegraphics[scale=.5]{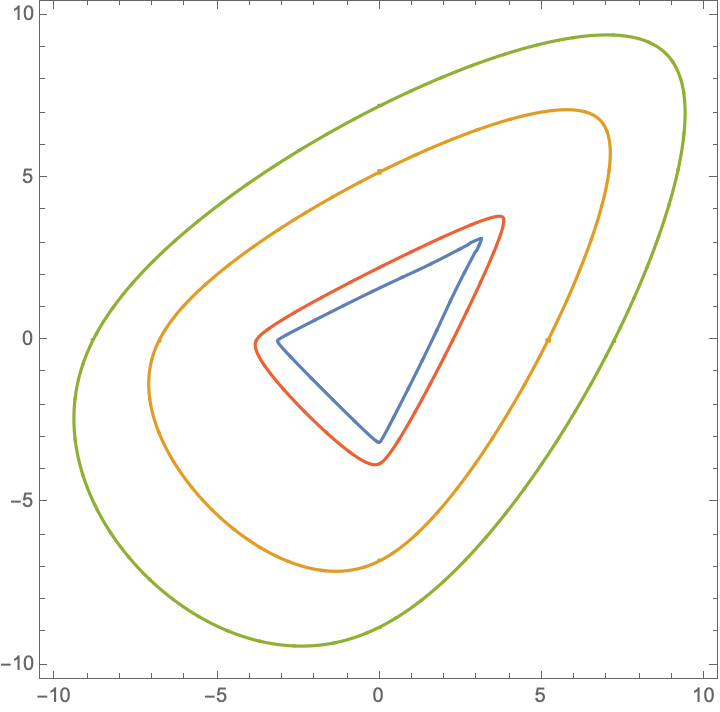}
    \caption{\small $(\Delta_2- b(\Delta_2))^{\circ,p}$}
    \label{fig:f4}
  \end{subfigure}
  \caption{\small The $L^p$-polars of (a) the square $B_\infty^2:=[-1,1]^2$, (b) the diamond
  $B_1^2:=(B_\infty^2)^\circ$, (c) the 2-simplex centered at the origin, 
  for $p=1/2$ (green), $p=1$ (orange), $p= 10$ (red) and $p=100$ (blue).}
  \label{mainFigure}
\end{figure}
}

As $p$ approaches 0, the $L^p$-polars of all three of the bodies pictured in Figure \ref{mainFigure} increase to $\R^2$. In fact, for any convex body $K\subset \R^n$, $K^{\circ,p}$ increases to $\{y: \langle y, b(K)\rangle\leq 1\}$ as $p\to 0$ (Proposition \ref{Kcirc0prop}), so we define $K^{\circ,0}$ to be exactly that (Definition \ref{Kcirc0def}). In particular, $K^{\circ,0}$ is either $\R^2$ or a half-space depending on whether or not $b(K)$ vanishes. By Example \ref{simplexpPolar}, we may plot a few of the $L^p$-polars of the standard simplex on the plane \eqref{SimplexDef}. Note that $\Delta_2^{\circ,0}$ is a half-space since $b(\Delta_2)\neq 0$.

\begin{figure}[H]
    \centering
    \includegraphics[scale=.5]{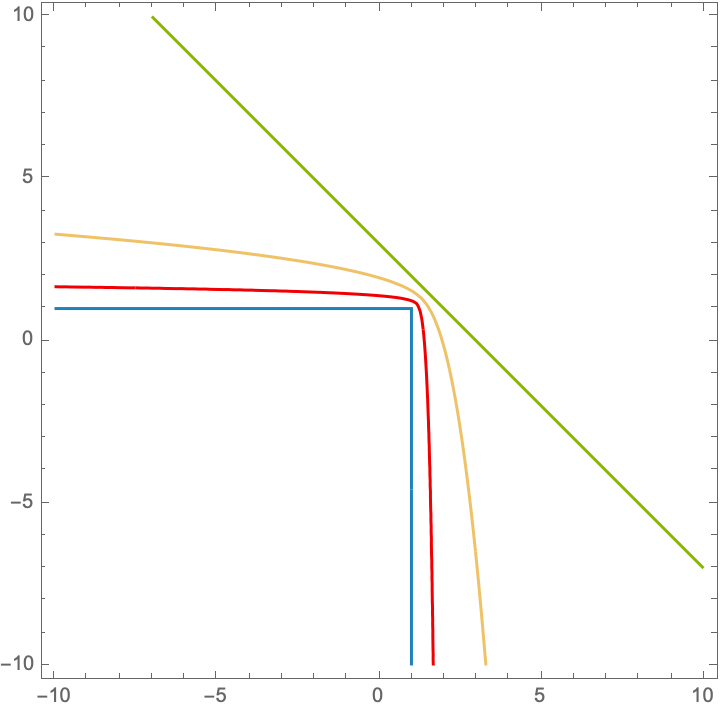}
    \caption{\small The boundary of $(\Delta_2)^{\circ,p}$ for $p=0$ (green), $p=2$ (orange), $p=10$ (red), and $p=\infty$ (blue).}
    \label{fig:f3}
\end{figure}

The proof of Theorem \ref{LpKwelldefined} has several parts.
To obtain \eqref{MpKeq} we rely on a result of Ball (Theorem \ref{KBall_ineq})
that implies that \eqref{opNorm} has all the properties of a norm, except that it is, in general, only {\it positively} 1-homogeneous, i.e., $\|\lambda y\|_{K^{\circ,p}}= \lambda \|y\|_{K^{\circ,p}}$ for $\lambda>0$. If $K$ is symmetric then $\|\cdot\|_{K^{\circ,p}}$ is 
fully 1-homogeneous, i.e., a norm (then $K^{\circ,p}$ is also symmetric).
For completeness, we include a detailed and self-contained proof of Ball's result
in Appendix \ref{appendixA}.
In particular, $\|\cdot\|_{K^{\circ,p}}$ is convex and so is $K^{\circ,p}$. 
Equality \eqref{MpKeq} follows from a standard formula relating the volume of a convex body to the surface integral of $\|\cdot\|_{K^{\circ,p}}^{-n}$ over the unit sphere
(see \eqref{normVolume}). 
Non-emptiness of the interior follows from $K^\circ\subset K^{\circ,p}$
(Lemma \ref{LpKpcont}). This inclusion also implies that $K^{\circ,p}$ is unbounded when $0\notin\mathrm{int}\,K$. 
The converse is slightly more subtle: when $0\in\mathrm{int}\,K$ one has a small cube $[-\eps,\eps]^n\subset K$.
For classical polarity this would be the end of the argument;
yet unlike classical polarity, $L^p$-polarity does not invert inclusions,
so we cannot simply argue that $K^{\circ,p}\subset ([-\eps,\eps]^n)^{\circ,p}$.
Instead, we use the existence of a small cube inside of $K$ to obtain a lower bound on $h_{p,K}$ in terms of $h_{p, [-\e,\e]^n}$ (see \eqref{hpKrEq}) which then induces an upper bound on $K^{\circ,p}$ by a multiple of $([-\e,\e]^n)^{\circ,p}$. The latter can be shown to be bounded (Claim \ref{cubeppolarBounded}), from which the boundedness of $K^{\circ,p}$ follows
by using yet another key estimate (Lemma \ref{hpKineq}).

\subsection{\texorpdfstring{$L^p$-Mahler conjectures
and uniqueness of minimizers}{Lᵖ-Mahler conjectures and uniqueness of minimizers}}
\lb{LpMahlerConjSubSec}

For $q>0$, denote by 
\begin{equation}
\lb{BqnEq}
    B_q^n\defeq \{x\in \R^n: |x_1|^q+ \ldots+ |x_n|^q\leq 1\}
\end{equation}
the (closed) $n$-dimensional $q$-ball, and denote by 
\begin{equation}\label{SimplexDef}
    \Delta_n\defeq \{x\in [0,\infty)^n: x_1+\ldots+ x_n\leq 1\} 
\end{equation}
the standard simplex in $\R^n$.
We propose a 1-parameter generalization of Mahler's conjectures. Mahler's original conjectures amount to setting $p=\infty$
in the following statements \cite[p. 96]{mahler},\cite{mahler2}.
\begin{conjecture}\label{pMahlerSym}
    Let $p\in(0,\infty]$. For a symmetric convex body $K\subset \R^n$,
    $$
    \M_p([-1,1]^n)\leq \M_p(K)\leq \M_p(B_2^n).
    $$ 
\end{conjecture}
\begin{conjecture}\label{pMahler}
    Let $p\in (0,\infty]$. For a convex body $K\subset \R^n$, 
    $$
    \inf_{x\in\Delta_n}\M_p(\Delta_n-x) \le \M_p(K).
    $$
\end{conjecture}

By Proposition \ref{santalo_point_thm} below the infimum in Conjecture \ref{pMahler} is attained by a unique point.

By the Bourgain--Milman inequality \cite[Corollary 6.1]{bourgain-milman}, there is $c>0$ independent of dimension so that $\M(K)\geq c^n$ for all convex bodies $K\subset \R^n$. By Lemma \ref{mKmpKineq} below, this induces a lower bound on $\M_p$ for all $p$ with the constant only depending on $p$.
The best known constant for $\M$ in dimensions $n\geq 4$ with $K$ symmetric is  $c=\pi$ \cite[Corollary 1.6]{kuperberg2}, \cite[Theorem 2.1]{berndtsson1}. The sharp bound $c=4$ is due to Mahler in dimension $n=2$ \cite[(2)]{mahler2}, and Iriyeh--Shibata in dimension $n=3$ \cite[Theorem 1.1]{iriyeh-shibata} (cf. Fradelizi et al. \cite{fradelizi-et_al}). For general $K$, the best known constant is $c=2$ for $n=3$ and $c=\pi/2$
for $n\geq 4$ by the symmetric bound and a symmetrization trick (see, for example, \cite[Corollary 55]{MR}). In dimension $n=2$ the sharp bound is due to Mahler \cite[(1)]{mahler2}.
One may also formulate other versions of Mahler's original conjecture, e.g., to zonoids \cite{reisner} or unconditional bodies \cite[\S 4]{saint-raymond} and generalize these to all $p$, but in this article we focus on Conjectures \ref{pMahlerSym} and \ref{pMahler}. In the special case $p=1$, using \eqref{diagonal_berg} one can show that the lower
bound of
Conjecture \ref{pMahlerSym} is equivalent to a conjecture of B\l{}ocki \cite[p. 56]{blocki},
while Conjecture \ref{pMahler} reduces to a conjecture of two of us 
\cite[Conjecture 10]{MR}, both stated in terms of Bergman kernels
of tube domains.

Conjectures \ref{pMahlerSym} and \ref{pMahler} for all $p\in(0,\infty)$
imply Mahler's conjectures, as we show in Lemma \ref{mKmpKineq}.
On the surface, the former look harder to deal with.
However, there is a subtle, perhaps crucial, advantage in the ``regularized"
version of the symmetric Mahler conjecture (Conjecture \ref{pMahlerSym} for $p\in(0,\infty)$) 
compared to the classical version ($p=\infty$) of that conjecture.
This has to do with the non-uniqueness of minimizers in the classical 
symmetric Mahler conjecture which has been pointed out by experts
\cite[\S 1.3]{tao}, \cite{TaoBlog} (see, in particular, the comments)
as one of the main obstacles to tackling it (see also the quote by Tao in \S\ref{IntroSec}). 
Let us elaborate on that.

Indeed, tensoriality of $\M=\M_\infty$ 
together with its invariance under classical polarity
leads to the conjectured non-uniqueness
of symmetric minimizers, referred to as Hanner polytopes
(non-uniqueness here is in the strong sense:  after taking the quotient by $GL(n,\R)$, i.e., there are minimizing bodies that are in different $GL(n,\R)$-orbits). Hanner polytopes are symmetric convex polytopes that are defined inductively: $[-1,1]$ is the unique Hanner polytope in dimension $n=1$. In higher dimensions, a Hanner polytope is given either as the Cartesian product of two lower-dimensional Hanner polytopes, or as the polar of such \cite[Theorems 3.1--3.2, 7.1]{hanner}, \cite[Corollary 7.4]{hansen-lima}. For example, in dimension $n=3$ there are precisely two non-$GL(n,\R)$ equivalent Hanner polytopes: the cube $[-1,1]^3$, as the product of lower-dimensional Hanner polytopes, and its polar $B_1^3$ \cite[pp. 86--87]{hanner}.

By contrast, our $L^p$-polarity operation  
\eqref{KpcircEq}
is no longer a duality, i.e., 
$(K^{\circ, p})^{\circ,p}\neq K$ in general. 
In fact, the $L^p$-polar always has a smooth boundary
for $p\in(0,\infty)$, and hence $L^p$-polarity is never a duality operation among polytopes.
By \eqref{MpKeq}
this means $\M_p$ is not invariant under $L^p$-polarity.
We conjecture that for all $p\in(0,\infty)$, up to the action of $GL(n,\R)$, $\M_p$ is uniquely minimized by the cube among symmetric convex bodies, and by the simplex, appropriately repositioned, among general convex bodies.
If true, this would give some motivation for studying $\M_p$ and show
that the original Mahler conjecture has (for the better and for the worse) additional invariance absent from our $L^p$-Mahler conjectures. Figure \ref{FigureDiamondCube} illustrates this 
symmetry-breaking property of $\M_p$ in $n=3$:

\begin{figure}[H]
    \centering
    \includegraphics[scale=.6]{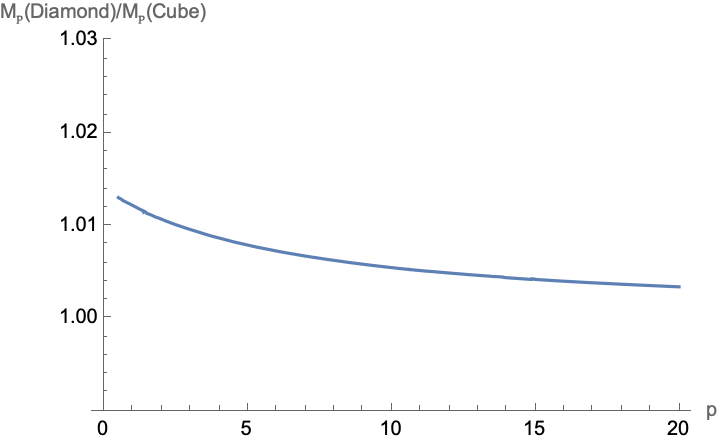}
    \caption{\small $\M_p(B_1^3)/\M_p(B_\infty^3)$ for $p\in[1/2,20]$.}
\lb{FigureDiamondCube}
\end{figure}

We emphasize that the above discussion pertains to the symmetric case, since in the non-symmetric case, the simplex, appropriately repositioned, is already
conjectured to be
the unique (up to $GL(n,\R)$) minimizer for the classical non-symmetric Mahler conjecture \cite{TaoBlog}. That is, $\M$ should be minimized by $\Delta_n-b(\Delta_n)$, where $b(\Delta_n)$ coincides with the Santal\'o point of $\Delta_n$. Note that $(\Delta_n- b(\Delta_n))^{\circ}$ is a $GL(n,\R)$ image of $\Delta_n-b(\Delta_n)$, so polarity does not produce a non-$GL(n,\R)$ equivalent minimizer in this case. 
The conjectured uniqueness of the minimizer in the non-symmetric case (regardless of $p$) is perhaps related to the fact that
$\Delta_n$ cannot be expressed as a product of polytopes of lower dimension.

\subsection{\texorpdfstring{$L^p$-Santal\'o theorem}{Lᵖ-Santaló theorem}}

For a function $f: \R^n\to \R\cup\{\infty\}$, denote by 
\begin{equation*}
    V(f)\defeq \int_{\R^n} e^{-f(x)}\dif x, 
    \quad 
    \text{ and } 
    \quad 
    b(f)\defeq \frac{1}{V(f)}\int_{\R^n} x e^{-f(x)} \dif x
\end{equation*}
its volume and barycenter respectively. 
This terminology is motivated by 
$
V(h_K)= n!|K^\circ|
$
(see \eqref{KcircVolume}),
and 
$
b(h_K)= (n+1) b(K^\circ)
$
(see \eqref{bhK}). 
By Theorem \ref{LpKwelldefined}, $V(h_{p,K})=n!|\Kppolar|$.
However, lacking homogeneity, it is not clear how $b(h_{p,K})$ can be directly related to $b(K^{\circ,p})$ (\S \ref{S3}).
Our next result 
generalizes the Santal\'o point.

\begin{proposition}\label{santalo_point_thm}
    Let $p\in(0,\infty]$. For a convex body $K\subset \R^n$ there exists a unique $x_{p,K}\in\R^n$ with
    \begin{equation*}
        \M_p(K-x_{p,K})= \inf_{x\in \R^n}\M_p(K-x), 
    \end{equation*}
    which is also the unique point such that $b(h_{p,K-x_{p,K}})=0$.
    Moreover, $x_{p,K}\in\mathrm{int}\,K$.
\end{proposition}

Part of the proof of Proposition \ref{santalo_point_thm} is almost identical to Santal\'o's proof of the existence and uniqueness of Santal\'o points \cite[\S 2]{santalo}. The idea is to show that the function $x\mapsto \M_p(K-x)$ is $\infty$ for $x\notin\mathrm{int}\,K$ (Lemma \ref{finiteness_bodies}), and smooth and strictly convex for $x\in \mathrm{int}\,K$ (Lemma \ref{strict_convexity}). This forces the existence of a unique minimum. The main difference is that we study $\int_{\R^n} e^{-h_{p,K}(y)}\dif y$ under translations of $K$, while Santal\'o studied the surface integral $\int_{\partial B_2^n}h_{K}(u)^{-n} \dif u$ \cite[(1.1)]{santalo}.

One of our main results is a generalization of Santal\'o's theorem, verifying the upper bound in Conjecture \ref{pMahlerSym}:

\begin{theorem}\label{santalo_symmetric}
Let $p\in(0,\infty]$. For a symmetric convex body $K\subset \R^n$, $\M_p(K)\leq \M_p(B_2^n)$.
\end{theorem}
In particular, by taking $p\to \infty$, one recovers Santal\'o's inequality $\M(K)\leq \M(B_2^n)$ \cite[(1.3)]{santalo}
(though, of course, for this purpose alone there are
direct, easier, proofs, e.g., \cite[Theorem 14]{saint-raymond}).

The $L^p$-polar $K^{\circ, p}$ \eqref{KpcircEq} is central to the proof of Theorem \ref{santalo_symmetric}. One idea behind the proof is standard: for $u\in\partial B_2^n$, the Steiner symmetrization with respect to a hyperplane through the origin 
$$
u^\perp\defeq \{x\in\R^n: \langle x,u\rangle=0\},
$$ 
increases the volume of the $L^p$-polar $|(\sigma_u K)^{\circ,p}|\geq |K^{\circ,p}|$ (Proposition \ref{steiner_ppolar}). 
Yet proving this seems non-standard and rather non-trivial. We achieve it 
by proving the following estimate comparing the slices of $K$ and those of $\sigma_u K$ over $u^\perp$, 
\begin{equation}\label{slice_comp}
    \frac12\left(\LpK\cap (u^\perp+ tu)\right)+ \frac12\left( \LpK\cap (u^\perp- tu)\right)\subset (\sigma_uK)^{\circ,p}\cap (u^\perp+ tu), 
    \q \hbox{for all $t\in \R$, }
\end{equation}
and then using the (classical) Brunn--Minkowski inequality.
To obtain (\ref{slice_comp}) we use Ball's Brunn--Minkowski inequality for harmonic means (Theorem \ref{KBall_ineq}), together with the convexity of $x\mapsto \log\left(\frac1t \sinh(t)\right)$ (Claim \ref{steiner_claim3}). 

\begin{remark}
\lb{LZRem}
Theorem \ref{santalo_symmetric} is different from the 
Lutwak--Zhang $L^p$ Santal\'o inequalities who introduced the 
symmetrized $L^p$-centroid body $\Gamma_pK$ with support function  given by \begin{equation*}
    h_{\Gamma_pK}(y) \defeq \left( \frac{1}{c_{n,p}}\int_K |\langle x,y\rangle|^p \frac{\dif x}{|K|}\right)^{\frac1p}
\end{equation*}
(where $c_{n,p}$ is a constant that depends on $n$ and 
$p$ determined by $\Gamma_p B_2^n= B_2^n$), for which they proved 
$    |K| |(\Gamma_pK)^\circ|\leq |B_2^n|^2$
\cite{lutwak-zhang}.
The Lutwak--Zhang construction is restricted to symmetric bodies since
$\Gamma_pK$ 
is always symmetric (regardless of whether $K$ is),
and the large $p$ limit does not recover
the polar body but rather the reflection body:
$\lim_{p\to \infty}\Gamma_p K= K\cup (-K)$
(since $\lim_{p\to \infty}h_{\Gamma_pK}(y)= \sup_{x\in K}|\langle x,y\rangle|$).
Subsequently, Ludwig and Haberl--Schuster extended
this to non-symmetric bodies \cite[p. 4195]{ludwig},
\cite[\S3]{haberl-schuster} 
introducing the $L^p$-centroid body $M_pK^+$ whose support
function is
\begin{equation*}
    h_{M_pK^+}(y)\defeq \left(C_{n,p} (n+p)\int_{K} \max\{\langle x,y\rangle,0\}^p \dif x\right)^{\frac1p}.
\end{equation*}
Note that as $p\to\infty$, $K^{\circ,p}\to K^\circ$ (Lemma \ref{LpKpcont}) while
$M_pK^+\to K$  \cite[p. 9]{haberl-schuster}.
Yet for fixed $p$, it is not apparent to us if there is a precise relation between  $M_pK^+$ and our $K^{\circ,p}$ (though the polar of former are `isomorphic' to the latter---see
Remark \ref{isomorphicRemark}). They seem to be distinct. For example, $\Gamma_2K$ is the Legendre ellipsoid of the convex body, thus bounding $|K||(\Gamma_2K)^\circ|$ from below by a bound of the form $c^n$, where $c$ is a constant independent of dimension, would imply  Bourgain's Conjecture \ref{SlicingConjecture}  \cite[p. 14]{lutwak-zhang}. On the contrary, by Lemma \ref{mKmpKineq} below, the Bourgain--Milman inequality implies bounds of this type for $\M_p$ for all $p>0$.
It would be interesting to investigate relations between these constructions and ours, as well as relations to the level-sets of the logarithmic Laplace transform (see Remark \ref{isomorphicRemark}), e.g., as in the works of
Klartag--Milman and Lata\l{}a--Wojtaszczyk \cite{klartag-Emilman,latala-wojta}.
\end{remark}

\subsection{Relation to the isotropic constant and Bourgain's slicing conjecture}

The $L^p$-support function 
$h_{p,K}$ is related to the covariance matrix of a convex body
(Lemma~\ref{MA_isotropic}),
\begin{equation}
\lb{CovEq}
    \mathrm{Cov}_{ij}(K)\defeq \int_{K}x_ix_j\frac{\dif x}{|K|}- \int_{K}x_i\frac{\dif x}{|K|}\int_{K}x_j\frac{\dif x}{|K|}, 
\end{equation}
via the identity 
\beq
\lb{hesshpkCovKEq}
\nabla^2 h_{p,K}(0)= p\mathrm{Cov}(K).
\eeq
This turns out to have an interesting connection to the slicing problem.
Denote 
\begin{equation}
\lb{CDefEq}
    \mathcal{C}(K)\defeq \frac{|K|^2}{\det\mathrm{Cov}(K)}.
\end{equation}
Note
\begin{equation}
\lb{CLKEq}
\mathcal{C}(K)=
\frac{1}{L_K^{2n}},
\end{equation}
where $L_K$ is the isotropic constant \cite[Definition 2.3.11]{brazitikos_etal}. Bourgain conjectured the following \cite[Remark p. 1470]{bourgain2} \cite[(1.9)]{bourgain}. 
\begin{conjecture}\label{SlicingConjecture}
   There exists a constant $c>0$ independent of dimension such that $\calC(K)\geq c^n$, for all $n\in\mathbb{N}
   $ and all convex bodies $K\subset \R^n$.  
\end{conjecture}

Let  $B>0$. We introduce the following convexity hypothesis:
\begin{equation}\label{B}\tag{$\ast_B$}
     u_{B,K}\defeq 
    \log\det\nabla^2 h_{1,K}+ B h_{1,K} \quad \text{ is convex}.
\end{equation}
Note here that $h_{1,K}$ is twice differentiable (Lemma \ref{strict_convexity}).
We restrict to $p=1$ since property \eqref{B} is equivalent to a similar convexity property on $h_{p,K}$ (see Remark \ref{h1Ksuffices}).

\begin{theorem}\label{bern_iso_prop}
    Let $K\subset \R^n$ be a convex body for which \eqref{B} holds for some $B>0$. Then: 
    
    \noindent 
    (i) There is $x_K\in \mathrm{int}\,K$ with
    \begin{equation*}
        \calC(K)\geq \frac{\M_{1/B}(K- x_K)^2}{\M_{1/(2B)}(K-x_K)} 
        \Big(\frac{2}{e B}\Big)^n. 
    \end{equation*}

    \noindent 
    (ii) There is $x_K\in \mathrm{int}\,K$ with
    \begin{equation*}
        \calC(K)\geq \frac{\M(K-x_K)}{e^{2n}B^n}\geq \Big(\frac{\pi}{2 e^{2}B}\Big)^n. 
    \end{equation*}

    \noindent 
    (iii) If $K$ is symmetric, 
    \begin{equation*}
        \calC(K)\geq \frac{\M(K)}{e^n B^n}\geq \Big(\frac{\pi}{e B}\Big)^n.
    \end{equation*}
\end{theorem}

Theorem \ref{bern_iso_prop} has
the following consequence for Bourgain's slicing conjecture. 
\begin{corollary}\label{slicingCorollary}
    If there is a constant $B>0$ independent of dimension such that \eqref{B} holds for all convex bodies in all dimensions, then Conjecture \ref{SlicingConjecture} holds.
\end{corollary}

In this direction,  we have the following partial progress:

\begin{theorem}\label{bern_iso_theorem}
$(\ast_{n+1})$ holds for all convex bodies $K\subset \R^n$.
\end{theorem}

As an immediate corollary of Theorems \ref{bern_iso_prop} and \ref{bern_iso_theorem} we recover the so-called `folklore' bound 
on the isotropic constant 
due to Milman--Pajor \cite[p. 96]{milman-pajor}.
\begin{corollary}\label{suboptimal_bound}
    For a convex body $K\subset \R^n$, 
    $\displaystyle\calC(K)\geq \bigg(\frac{\pi}{2e^2n}\bigg)^n$. 
\end{corollary}
Corollary \ref{suboptimal_bound} is equivalent to an upper bound on the isotropic constant, 
\beq
\lb{MPInEq}
L_K\leq C \sqrt{n},
\eeq
for $C= e\sqrt{2/\pi}$ and hence is far from optimal: 
by Milman--Pajor \eqref{MPInEq} holds with $C= 2\pi e$,
by Bourgain $L_K\leq C n^{1/4}\log n$, 
\cite[Theorem 1.6]{bourgain},  by
Klartag $L_K\leq C n^{1/4}$ \cite[Corollary 1.2]{klartag}, 
while very recently 
Chen obtained
$L_K\leq C_1 e^{C_2 \sqrt{\log(n)}\sqrt{\log\log(3n)}}$
(in particular, $L_K\leq Cn^\e$ for all $\e>0$) \cite{chen},
\cite[(1)]{klartag-lehec}, and on these foundations several authors 
improved this to $L_K\leq C (\log(n))^q$ for various 
values of $q$ \cite{klartag-lehec,JLV,klartag2}.
Conjecture \ref{SlicingConjecture} remains open.

The proof of Theorem \ref{bern_iso_prop} 
starts with the observation \eqref{hesshpkCovKEq}.
The convexity assumption \eqref{B} allows for the application of Jensen's inequality with respect to any probability measure $\mu$. 
Because of \eqref{hesshpkCovKEq} this will only be useful if
$\mu$ is centered at the origin, i.e., 
\begin{equation*}
    b(\mu)\defeq \int_{\R^n} y\dif\mu(y)=0\in\R^n. 
\end{equation*}
We use the family of log-concave measures given by the $1/p$-support functions,
\begin{equation}\label{nuMeasure}
    \nu_{p,K}\defeq 
    \frac{e^{-h_{1/p, K}(y)}\dif y}{\int_{\R^n} e^{-h_{1/p, K}(y)}\dif y}= \frac{e^{-ph_{1,K}(y)}\dif y}{\int_{\R^n} e^{-ph_{1,K}(y)}\dif y},
\end{equation}
and optimize over $p$ (the equality in \eqref{nuMeasure} follows from Lemma \ref{list} (i) below).
Proposition \ref{santalo_point_thm} is crucial here, since it assures that $K$ may be translated to a position for which $b(\nu_{p,K})=0$ (Corollary \ref{santalo_point_corollary}).
After applying Jensen's inequality for the measures $\nu_{p,K}$, it remains to bound $\int_{\R^n}\log\det\nabla^2 h_{1,K}\dif\nu_{p,K}(y)$ and $\int_{\R^n} h_{1,K}(y)\dif\nu_{p,K}(y)$; this is done in Lemmas \ref{berniso_lem2} and \ref{berniso_lem1} respectively. The $L^p$-Mahler volumes $\M_p$ figure quite prominently throughout the proofs.

The proof of
Theorem \ref{bern_iso_theorem} is based upon an explicit computation 
\begin{equation}\label{hpkHessianEq}
\log\det\nabla^2 h_{p,K}(y) = -p(n+1) h_{p,K}(y)+ \log\psi(y),
\end{equation}
$\psi$ being the determinant of a positive-definite matrix.
This relies on writing $\det\nabla^2 h_{p,K}$ as the determinant of the $(n+1)\times(n+1)$ Gram matrix $M$ of the first moments of the measure 
\begin{equation*}
    \frac{e^{p\langle x,y\rangle}}{e^{ph_{p,K}(y)}} \frac{\bm{1}_K^\infty(x)\dif x}{|K|}. 
\end{equation*}
Each entry of $M$ then involves an $e^{-ph_{p,K}(y)}$ term, thus $\det\nabla^2 h_{p,K}= \det M= e^{-(n+1) p h_{p,K}}\psi$, for a positive $\psi>0$. Taking logarithm gives \eqref{hpkHessianEq}. For the remaining terms, $\psi$, being the sum of products of $n+1$ integrals over $K$, can be written as an integral over $K^{n+1}$, 
\begin{equation*}
    \psi(y)= C\int_{K^{n+1}} |\Delta(z)|^2 e^{p\langle z, (y,\ldots, y)\rangle}\dif z,  
\end{equation*}
from which the convexity of $\log\psi$ can be deduced (Lemma \ref{logconvLemma}), and hence the claim of Theorem \ref{bern_iso_theorem}.  

Finally, we generalize Theorem \ref{bern_iso_theorem} to 
the setting of a general probability measure---this is formulated in
Theorem \ref{MeasureProp}. In this generality, we show that
the constant $B=n+1$ is actually optimal. In \S\ref{complexgeom_subsec} we give a completely different proof of both theorems using, surprisingly, Kobayashi's
theorem on the Ricci curvature of Bergman metrics, coming back full circle
to the point of departure of this article in \S\ref{BergmanSubSec}: Bergman kernels.

\subsection{Perspective on the work on Nazarov and B\l{}ocki}
\lb{NazarovSubSec}

Having presented $L^p$-polarity, it is perhaps worthwhile to
revisit our original motivation for developing this theory: the work of
Nazarov \cite{nazarov} and B\l{}ocki \cite{blocki,blocki2}. 

Nazarov applied the theory 
of Bergman kernels of tube domains to tackle the symmetric Mahler conjecture.
The constant he obtained $c=\pi^3/16$ 
in the inequality $\M(K)\geq c^n$ for symmetric convex bodies $K\subset \R^n$
was sub-optimal compared to the conjectured
value of $c=4$ (see \S\ref{LpMahlerConjSubSec})
 but the possibility remained open
that perhaps a better choice of holomorphic $L^2$ function
and weight function in H\"ormander's $\bar\partial$-technique
would allow to tackle the Mahler conjectures,
or that perhaps, as Nazarov suggested \cite[p. 337]{nazarov},
\begin{displayquote}
\textit{``...in order to get the
Mahler conjecture itself on this way, one would have to work directly with the
Paley--Wiener space by either finding a good analogue of the H\"ormander theorem
allowing to control the Paley--Wiener norm of the solution, or by finding some novel
way to construct decaying analytic functions of several variables."}
\end{displayquote}

Nazarov's approach was subsequently revisited 
by B\l{}ocki \cite{blocki,blocki2}, Hultgren \cite{hultgren}, and ourselves \cite{berndtsson2,MR}.
It became plausible after the work of
B\l{}ocki \cite[p. 96]{blocki2} that Nazarov's approach might not
yield Mahler's conjectures.
In view of the results in the present article (e.g., Lemma \ref{mKmpKineq}) it is now clear why
this is so, and exactly
how Nazarov's approach fits in our story: it is an approach to the 
case $p=1$ of Conjectures \ref{pMahlerSym}--\ref{pMahler}. 
It is a beautiful coincidence that $L^1$-Mahler volumes can be expressed
in terms of Bergman kernels (see \S \ref{BergmanSubSec} and \cite[(42)]{MR}),
\begin{equation}
\lb{M1KEq}
    \M_1(K-b(K))= (4\pi)^n |K|^2\mathcal{K}_{T_K}(\i b(K),\i b(K));
\end{equation}
but even if one had a complete understanding
of the variation of such kernels among tube domains, 
solving the classical Mahler conjectures would still require bridging
the gap between $L^1$ and $L^\infty$.

Finally, we touch upon an observation encountered by B\l{}ocki
 \cite[p. 96]{blocki2}:
\begin{displayquote}
\textit{``This shows (although only numerically) that the Bergman kernel for tube domains does not behave well under taking duals."}
\end{displayquote}
Indeed, the theory of Bergman kernels of tube domains corresponds to $\M_1$ and
$L^1$-polarity and the lack of homogeneity of $h_{1,K}$
leads to incompatability with $L^\infty$-polarity, i.e., 
with classical polarity/duality.

\bigskip
\noindent 
\textbf{Organization.} In \S\ref{S11} basic properties of $h_{p,K}$ are laid out, namely the convexity of $h_{p,K}$ (Lemma \ref{hpKconvex}), its behaviour under affine transformations of $K$, Cartesian products, and its monotonicity with respect to $p$ (Lemma \ref{list}). Convexity properties of $h_{p, K}$ with respect to $p$ or $K$ are studied in \S \ref{S12}. In \S \ref{S13}, an upper bound to the support function $h_p$ in terms of $h_{p,K}$ for bodies with barycenter at the origin $b(K)=0$ is given. \S \ref{S14} is dedicated to the explicit computation of $h_{p,[-1,1]^n}$ for the cube. 
In \S \ref{S21} the $L^p$-polar $K^{\circ, p}$ is introduced, for which $\M_p(K)=n! |K||K^{\circ,p}|$, $K^{\circ}\subset K^{\circ,p}$ and $\cap_{p>0}K^{\circ, p}= K^\circ$.
Inequalities relating $\M$ to $\M_p$ are established in \S \ref{S22}, and \S \ref{S23} is dedicated to computing $\M_p([-1,1]^n)$. In \S \ref{hpDiamondSection}, the $L^p$-support of the diamond $B_1^n$ is explicitly computed in all dimensions and for all $p$ (Lemma \ref{diamondpSupport}).
Section \ref{S3} establishes the existence and uniqueness of Santal\'o points for $\M_p$ (Proposition \ref{santalo_point_thm}), and in Section \ref{S4} we prove a Santal\'o inequality for $\M_p$ for symmetric convex bodies, showing that the $2$-ball $B_2^n$ is the maximizer (Theorem \ref{santalo_symmetric}). In Section \ref{S5}, we study the isotropic constant and the relations between $h_{p,K}$, $\M_p$, and Bourgain's conjecture. In particular, we prove
Theorem \ref{bern_iso_prop}, Theorem \ref{bern_iso_theorem}, and 
its generalization, Theorem \ref{MeasureProp}. 
We conclude by giving an alternative proof of the latter using Bergman kernel methods and Kobayashi's theorem.
In Appendix \ref{appendixA}, we verify that $K^{\circ,p}$ is a convex body by proving Proposition \ref{norm_phi}, and provide a detailled proof of Ball's Brunn--Minkowski inequality for the harmonic mean (Theorem \ref{KBall_ineq}).

\bigskip
\noindent 
\textbf{Acknowledgments.}
Research supported by NSF grants 
DMS-1906370,2204347 and a Swedish Research
Council conference grant ``Convex and Complex:
Perspectives on Positivity
in Geometry."

\section{\texorpdfstring{$L^p$ support functions}{Lᵖ support functions}}\label{S1}

In this section we lay out basic properties for $h_{p,K}$. In \S \ref{S11} we show convexity of $y\mapsto h_{p,K}(y)$ (Lemma \ref{hpKconvex}) and list several properties in Lemma \ref{list}, e.g., how $h_{p,K}$ transforms under affine transformations of $K$ or with respect to Cartesian products. In \S \ref{S12} we study convexity properties of $h_{p,K}$ in terms of convex combinations of $p$ (Lemma \ref{conv1}) or $K$ (Lemma \ref{conv2}). An upper bound for the support function $h_K$ by $h_{p,K}$ for bodies with barycenter at the origin $b(K)=0$ is given in \S \ref{S13}. Finally, in \S \ref{S14} we carry out explicit computations for the cube. 

\subsection{Basic properties of \texorpdfstring{$h_{p,K}$}{Basic properties of hₚ,ₖ}}\label{S11}

The functions $h_{p,K}$ defined by (\ref{hpdef}) are convex, even if the underlying body $K$ is only compact. 
\begin{lemma}\label{hpKconvex}
    Let $p\in(0,\infty)$. For a compact body $K\subset \R^n$, $h_{p,K}(y)$ is a convex function of $y$. 
\end{lemma}
\begin{proof}
    Let $y,z\in \R^n$ and $\lambda\in (0,1)$. By H\"older's inequality, 
    \begin{equation*}
        \begin{aligned}
            h_{p,K}((1-\lambda)y+ \lambda z)&= \frac1p \log\left( \int_{K} e^{p\langle x, (1-\lambda)y+ \lambda z\rangle}\frac{\dif x}{|K|}\right)\\
            &= \frac1p \log\left( \int_K \left(e^{p\langle x,y\rangle}\right)^{1-\lambda} \left(e^{p\langle x,z\rangle}\right)^\lambda\frac{\dif x}{|K|}\right)\\
            &\leq \frac1p \log\left[\left( \int_K e^{p\langle x,y\rangle}\frac{\dif x}{|K|}\right)^{1-\lambda} \left(\int_K e^{p\langle x,z\rangle}\frac{\dif x}{|K|} \right)^{\lambda}\right]\\
            &= \frac{1-\lambda}{p} \log\left( \int_K e^{p\langle x,y\rangle}\frac{\dif x}{|K|}\right) +\frac{\lambda}{p}\log\left(\int_K e^{p\langle x,z\rangle}\frac{\dif x}{|K|} \right)\\
            &= (1-\lambda)h_{p,K}(y)+ \lambda h_{p,K}(z). 
        \end{aligned}
    \end{equation*}
\end{proof}

Next, a list of properties of $L^p$-support functions that will be useful throughout. 
\begin{lemma}\label{list}
Let $0<p<q< \infty$. For compact bodies $K\subset \R^n, L\subset \R^m$, and $A\in GL(n,\R), a\in \R^n$:

\noindent
(i) $h_{p,K}(y)= \frac1p h_{1,K}(py)$. 

\noindent
(ii) $h_{p, K-a}(y)= h_{p,K}(y)-\langle a,y\rangle$. 

\noindent
(iii) $h_{p, AK}(y)= h_{p, K}(A^T y)$. 

\noindent 
(iv) $h_{p, K\times L}(y,z)= h_{p, K}(y)+ h_{p, L}(z), y\in \R^n, z\in \R^m$.

\noindent
(v) $h_{p, K}\leq h_{q, K}\leq h_{K}$.

\end{lemma}
\begin{proof}
    \noindent 
    (i) By definition, $h_{p, K}(y)=\frac1p\log \int_K e^{p\langle x,y\rangle}\frac{\dif x}{|K|}= \frac1p \log\int_K e^{\langle x,py\rangle}\frac{\dif x}{|K|}= \frac1p h_{1,K}(py)$.

    \noindent
    (ii) Changing variables $x= u-a$, for $x\in K-a$, $u\in K$, and $\dif x= \dif u$,
    \begin{equation*}
    \begin{aligned}
        h_{p,K-a}(y)&= \frac1p\log\left( \int_{K-a} e^{p\langle x,y\rangle}\frac{\dif x}{|K-a|}\right)= \frac1p\log\left( \int_K e^{p\langle u-a, y\rangle}\frac{\dif u}{|K|}\right)\\
        &= \frac1p \log\left( \int_K e^{p\langle u,y\rangle}\frac{\dif u}{|K|} e^{-p\langle a,y\rangle}\right)= h_{p,K}(y)-\langle a,y\rangle. 
    \end{aligned}
    \end{equation*}

    \noindent 
    (iii) For $x= Au, \dif x= |\det A| \dif u$, 
    \begin{equation*}
    \begin{aligned}
        h_{p, AK}(y)&= \frac1p\log\left( \int_{AK} e^{p\langle x,y\rangle}\frac{\dif x}{|AK|}\right)= \frac1p \log\left( \int_K e^{p\langle Au, y\rangle}\frac{|\det A|\dif u}{|\det A| |K|}\right)\\
        &= \frac1p \log\left( \int_K e^{p\langle u, A^T y\rangle} \frac{\dif u}{|K|}\right)= h_{p, K}(A^T u).
    \end{aligned}
    \end{equation*}

    \noindent 
    (iv) By Tonelli's theorem \cite[\S 2.37]{folland},\cite[Claim 22]{MR}, 
    \begin{equation*}
    \begin{aligned}
        h_{p, K\times L}(y,z)&= \frac1p\log\left(\int_{K\times L} e^{p\langle (x,u), (y,z)\rangle}\frac{\dif x\dif u}{|K\times L|}\right)\\
        &= \frac1p\log\left( \int_{K\times L} e^{p\langle x,y\rangle} e^{p\langle z,u\rangle}\frac{\dif x\dif u}{|K||L|}\right) \\
        &= \frac1p \log\left[\left( \int_K e^{p\langle x,y\rangle}\frac{\dif x}{|K|}\right)\left(\int_L e^{p\langle z,u\rangle}\frac{\dif u}{|L|}\right)\right]\\
        &= \frac1p \log\left( \int_K e^{p\langle x,y\rangle}\frac{\dif x}{|K|}\right)+\frac1p\log\left( \int_L e^{p\langle z,u\rangle}\frac{\dif u}{|L|}\right) 
        = h_{p, K}(y)+ h_{p,L}(z).
    \end{aligned}
    \end{equation*}

    \noindent 
    (v) By \eqref{hinftyKEq},
    \begin{equation*}
        h_{q, K}(y)\defeq \frac1q\log\left( \int_K e^{q\langle x,y\rangle}\frac{\dif x}{|K|}\right)\leq \frac1q \log\left( \int_K e^{q h_K(y)}\frac{\dif x}{|K|}\right)= \frac1q \log e^{qh_{K}(y)}= h_K(y). 
    \end{equation*}
    By H\"older's inequality (note $q/p>1$), 
    \begin{equation*}
    \begin{aligned}
        h_{p,K}(y)&= \frac1p \log\left( \int_K e^{p\langle x,y\rangle} \frac{\dif x}{|K|}\right)\leq \frac1p \log\left[ \left( \int_K e^{\frac{q}{p} p\langle x,y\rangle}\frac{\dif x}{|K|}\right)^{\frac{p}{q}} \left( \int_K \frac{\dif x}{|K|}\right)^{1-\frac{p}{q}}\right] \\
        &= \frac1p \frac{p}{q} \log\left( \int_K e^{q\langle x,y\rangle}\frac{\dif x}{|K|}\right)= h_{q, K}(y).
    \end{aligned}
    \end{equation*}
\end{proof}

\begin{remark}\label{tensor}
One may wonder why we have a factor of $n!$ in \eqref{MahlerEq}
and \eqref{MpKeq}.
The first reason is that then one has \eqref{0.2} and \eqref{Mpdef}.
The second, more important, reason is that then $\M_p$ is {\it tensorial}.
Indeed, by Lemma \ref{list} (iv) and (\ref{Mpdef}), 
\begin{equation*}
\begin{aligned}
    \M_p(K\times L)&\defeq |K\times L|\int_{\R^{n}\times \R^m} e^{-h_{p,K\times L}(y,z)}\dif y\dif z\\
    &= |K||L|\int_{\R^n\times \R^m} e^{-h_{p,K}(y)} e^{-h_{p,L}(z)}\dif y\dif z\\
    &= \M_p(K)\M_p(L). 
\end{aligned}
\end{equation*}
\end{remark}

\subsection{Additional convexity properties}\label{S12}
Lemma \ref{hpKconvex} states that $y\mapsto h_{p,K}(y)$ is convex regardless of the convexity of $K$. Regarding $p$ and $K$ as the variables, we show two more 
properties: Lemma \ref{conv1} describes convexity in $p$, and
Lemma \ref{conv2} shows an asymptotic (in $p$) concavity in $K$.
These two lemmas are not used elsewhere in the article
and we state them for their independent interest. 

\begin{lemma}\label{conv1}
    Let $p,q\in (0,\infty)$. For a convex body $K\subset \R^n$ and $\lambda\in (0,1)$, 
    \begin{equation*}
        h_{(1-\lambda)p+\lambda q,K}\leq \frac{(1-\lambda)p}{(1-\lambda)p+\lambda q} h_{p, K}+ \frac{\lambda q}{(1-\lambda)p+ \lambda q}h_{q, K}.
    \end{equation*}
\end{lemma}
\begin{proof}
    By H\"older's inequality, 
    \begin{equation*}
        \begin{aligned}
            h_{(1-\lambda)p+\lambda q, K}(y)&= \frac{1}{(1-\lambda)p+\lambda q} \log\left( \int_K e^{((1-\lambda)p+\lambda q) \langle x,y\rangle} \frac{\dif x}{|K|}\right) \\
            &= \frac{1}{(1-\lambda)p+\lambda q} \log\left( \int_K e^{(1-\lambda)p\langle x,y\rangle} e^{\lambda q\langle x,y\rangle}\frac{\dif x}{|K|}\right)\\
            &\leq \frac{1}{(1-\lambda)p+\lambda q} \log\left[\left( \int_K e^{p\langle x,y\rangle}\frac{\dif x}{|K|}\right)^{1-\lambda} \left(\int_K e^{q\langle x,y\rangle}\frac{\dif x}{|K|}\right)^{\lambda}\right]\\
            &= \frac{(1-\lambda)p}{(1-\lambda)p+\lambda q}\frac1p\log\left(\int_K e^{p\langle x,y\rangle}\frac{\dif x}{|K|}\right)+ \frac{\lambda q}{(1-\lambda)p+ \lambda q} \frac1q \log\left(\int_K e^{q\langle x,y\rangle}\frac{\dif x}{|K|}\right)\\
            &= \frac{(1-\lambda)p}{(1-\lambda)p+\lambda q} h_{p,K}(y)+ \frac{\lambda q}{(1-\lambda)p+ \lambda q} h_{q, K}(y).
        \end{aligned}
    \end{equation*}
\end{proof}

\begin{lemma}\label{conv2}
    Let $p\in(0,\infty)$. For convex bodies $K, L\subset \R^n$ and $\lambda\in (0,1)$, 
    \begin{equation*}
        h_{p, (1-\lambda)K+ \lambda L}\geq (1-\lambda)h_{p, K}+ \lambda h_{p, L}-\frac1p \log\left( \frac{|(1-\lambda)K+\lambda L|}{|K|^{1-\lambda} |L|^\lambda}\right). 
    \end{equation*}
\end{lemma}
\begin{proof}
    Fix $y\in \R^n$. Note that 
    $$\mathbf{1}_{(1-\lambda)K+ \lambda L}((1-\lambda)x+\lambda z) e^{p\langle (1-\lambda)x+ \lambda z, y\rangle}\geq \left( \mathbf{1}_K(x) e^{p\langle x,y\rangle}\right)^{1-\lambda} \left(\mathbf{1}_L(z) e^{p\langle z,y\rangle}\right)^\lambda
    $$ 
    for all $x,z\in \R^n$. Therefore, by Pr\'ekopa--Leindler inequality \cite[Theorem 3]{prekopa}, 
    \begin{equation*}
        \int_{(1-\lambda)K+ \lambda L} e^{p\langle x, y\rangle}\dif x\geq \left(\int_K e^{p\langle x,y\rangle}\dif x\right)^{1-\lambda}\left(\int_L e^{p\langle x,y\rangle}\dif x\right)^\lambda. 
    \end{equation*}
    As a result, 
    \begin{equation*}
        \begin{aligned}
            h_{p, (1-\lambda)K+\lambda L}(y)&= \frac1p\log\left(\int_{(1-\lambda)K+\lambda L} e^{p\langle x,y\rangle}\frac{\dif x}{|(1-\lambda)K+\lambda L|}\right)\\
            &\geq \frac1p \log\left[\left(\int_{K}e^{p\langle x,y\rangle}\dif x\right)^{1-\lambda} \left( \int_L e^{p\langle x,y\rangle} \dif x\right)^\lambda \frac{1}{|(1-\lambda)K +\lambda L|}\right]\\
            &= \frac1p \log\left[\left(\int_{K}e^{p\langle x,y\rangle}\frac{\dif x}{|K|}\right)^{1-\lambda} \left( \int_L e^{p\langle x,y\rangle} \frac{\dif x}{|L|}\right)^\lambda \frac{|K|^{1-\lambda}|L|^\lambda}{|(1-\lambda)K +\lambda L|}\right] \\
            &= (1-\lambda) h_{p,K}(y)+ \lambda h_{p,K}(y)-\frac1p \log\left( \frac{|(1-\lambda)K +\lambda L|}{|K|^{1-\lambda}|L|^\lambda}\right),
        \end{aligned}
    \end{equation*}
    as claimed.
\end{proof}

\subsection{A reverse inequality}\label{S13}

By Lemma \ref{list} (v), 
\begin{equation*}
    h_{p,K}\leq h_K
\end{equation*}
regardless of the position of $K$.
A reverse inequality holds when the barycenter is at the origin:

\begin{lemma}\label{hpKineq}
Let $p\in(0,\infty)$. For a convex body $K\subset \R^n$ with $b(K)=0$, and $\lambda \in (0,1)$, 
\begin{equation*}
    h_K(y)\leq h_{p,K}(y/\lambda)- \frac{n}{p}\log(1-\lambda).
\end{equation*}
\end{lemma}
\begin{proof}
Let $x\in K$, $y\in \R^n$ and $\lambda\in(0,1)$. The aim is to use Jensen's inequality to get an upper bound on $\langle x,y\rangle$. Since $b(K)=0$,
\begin{equation}\label{xdotyeq}
\begin{aligned}
\langle x,y\rangle= \langle \lambda x, \frac{y}{\lambda}\rangle= \langle \lambda x+ (1-\lambda)b(K), \frac{y}{\lambda}\rangle&= \left\langle \lambda x+ (1-\lambda)\int_K u\frac{\dif u}{|K|}, \frac{y}{\lambda}\right\rangle\\
&= \int_{K}\left\langle \lambda x+ (1-\lambda)u, \frac{y}{\lambda}\right\rangle \frac{\dif u}{|K|}. 
\end{aligned}
\end{equation}
By convexity, $(1-\lambda)x+ \lambda u$ lies in $K$ as $x,u \in K$.
Therefore, by (\ref{xdotyeq}), Jensen's inequality, and the change of variables $v= \lambda x+ (1-\lambda)u$,
\begin{equation*}
    \begin{aligned}
    \langle x,y\rangle= \log e^{\langle x,y\rangle}&\leq \log \left(\int_{K}e^{\langle \lambda x+ (1-\lambda)u, \frac{y}{\lambda}\rangle}\frac{\dif u}{|K|}\right)\\
    &= \log\left(\int_{\lambda x+ (1-\lambda)K} e^{\langle v, \frac{y}{\lambda}\rangle}\frac{(1-\lambda)^{-n}\dif v}{|K|} \right)\\
    &= \log\left(\frac{1}{(1-\lambda)^n}\int_{\lambda x+ (1-\lambda)K} e^{p\langle v, \frac{y}{p\lambda}\rangle} \frac{\dif v}{|K|} \right)\\
    &\leq \log\left(\frac{1}{(1-\lambda)^n}\int_{K} e^{p\langle v, \frac{y}{p\lambda}\rangle} \frac{\dif v}{|K|} \right)\\
    &= p h_{p,K}(\frac{y}{p\lambda})- n\log(1-\lambda). 
    \end{aligned}
\end{equation*}
A supremum over $x\in K$ gives $h_K(y)\leq p h_{K,p}(\frac{y}{p\lambda})- n\log(1-\lambda)$. By a change of variable,
$h_K(py)\leq p h_{K,p}(\frac{y}{\lambda})- n\log(1-\lambda)$. The lemma now follows
from homogeneity of $h_K$.
\end{proof}

\begin{corollary}\label{hpLimit}
    Let $q\in (0,\infty]$. For a convex body $K\subset \R^n$, 
    \begin{equation*}
        \lim_{p\to q}h_{p,K}(y)= h_{q,K}(y).
    \end{equation*}
\end{corollary}
\begin{proof}
First, let $q\in (0,\infty)$. Since $K$ is bounded, there exists $M>0$ with $|x|\leq M$ for all $x\in K$. In particular, $e^{p\langle x,y\rangle}\leq e^{2qM|y|}$ for all $x\in K$ and $p\leq 2q$. By dominated convergence \cite[\S 2.24]{folland}, 
\begin{equation*}
    \lim_{p\to q} \int_K e^{p\langle x,y\rangle}\frac{\dif x}{|K|}= \int_{K}e^{q\langle x,y\rangle}\frac{\dif x}{|K|}. 
\end{equation*}
Therefore,
\begin{equation*}
    \lim_{p\to q} h_{p,K}(y)= \lim_{p\to q}\left(\frac{1}{p} \log\int_K e^{p\langle x,y\rangle}\frac{\dif x}{|K|}\right)= \frac1q \log\int_K e^{q\langle x,y\rangle}\frac{\dif x}{|K|}= h_{q,K}(y).
\end{equation*}

Next, consider $q=\infty$,
by Lemma \ref{list} (v), $h_{p,K}(y)$ is monotone increasing in $p$ with $h_{p,K}(y)\leq h_K(y)$, thus the limit exists with 
$\lim_{p\to\infty} h_{p,K}(y)\leq h_K(y)$, equivalently,
$\lim_{p\to\infty} [h_{p,K}(y)-\langle y, b(K)\rangle]
\leq h_K(y)-\langle y, b(K)\rangle$. By
 Lemma \ref{list} (ii), this is 
$$\lim_{p\to\infty} h_{p,K-b(K)}(y)
\leq h_{K-b(K)}(y).$$ 
On the other hand, as $b(K-b(K))=0$, 
Lemma \ref{hpKineq} applies, 
    \begin{equation*}
        h_{K-b(K)}(y)
        =
        h_{K-b(K)}(\lambda y)/\lambda 
        \leq h_{p,{K-b(K)}}(y)/\lambda-\frac{n}{\lambda p} \log(1-\lambda),
    \end{equation*}
    where we used the homogeneity of $h_K$ (here $\lambda$ 
    can be taken as any fixed value in $(0,1)$). 
    Letting first $p\to \infty$ and then $\lambda\ra 1$, 
    $$h_{K-b(K)}(y)\le
    \lim_{p\to\infty} h_{p,K-b(K)}(y).$$
    In conclusion, $h_{K-b(K)}(y)=
    \lim_{p\to\infty} h_{p,K-b(K)}(y)$ and using 
    Lemma \ref{list} (ii) again we obtain
$h_{K}(y)=
    \lim_{p\to\infty} h_{p,K}(y)$.
\end{proof}

\subsection{The cube}\label{S14}
We explicitly compute the $L^p$-support functions and $L^p$-Mahler volumes of the cube $[-1,1]^n$. This will be useful in proving Lemma \ref{finiteness_bodies} later. 
\begin{lemma}\label{cube_hp}
    For $p\in(0,\infty)$,
    \begin{equation*}
        h_{p, [-1,1]^n}(y)= \frac1p\sum_{i=1}^n \log\left(\frac{\sinh(py_i)}{py_i}\right), \quad y\in \R^n.
    \end{equation*}
\end{lemma}

\begin{proof}
By Claim \ref{cube-hp-claim} below, 
\begin{equation*}
    \begin{aligned}
        h_{p, [-1,1]^n}(y)&= \frac1p\log\left( \int_{[-1,1]^n}e^{p\langle x,y\rangle}\frac{\dif x}{|[-1,1]^n|}\right)\\
        &= \frac1p \log\left( 2^n \prod_{i=1}^n \frac{\sinh(py_i)}{py_i}\frac{1}{2^n}\right)\\
        &= \frac1p\sum_{i=1}^n \log\left(\frac{\sinh(py_i)}{py_i}\right).
    \end{aligned}
\end{equation*}
\end{proof}

\begin{claim}\label{cube-hp-claim}
    For $y\in\R^n$, 
    \begin{equation*}
        \int_{[-1,1]^n} e^{\langle x,y\rangle}\dif x= 2^n \prod_{i=1}^n \frac{\sinh(y_i)}{y_i}.
    \end{equation*}
\end{claim}
\begin{proof}
This may be expressed as the product of integrals $\int_{[-1,1]^n} e^{\langle x,y\rangle}$$= \prod_{i=1}^n \int_{-1}^1 e^{x_iy_i}\dif x_i$,
because $e^{\langle x,y\rangle}= e^{x_1y_1} \ldots e^{x_ny_n}$ and $[-1,1]^n$ is the product of $n$ copies of $[-1,1]$.
It is therefore enough to take $n=1$ and $y\in\R$. 
Suppose first that $y\neq 0$. Then,
\begin{equation*}
    \int_{-1}^1 e^{xy}\dif x= \left[\frac{e^{xy}}{y}\right]_{x=-1}^1= \frac{e^{y}- e^{-y}}{y}= \frac{2\sinh(y)}{y}.
\end{equation*}
 For $y=0$, $\int_{-1}^1 e^{x\cdot 0}\dif x= 2$. By L'H\^{o}pital's rule also
 \begin{equation*}
     \lim_{y\to 0}\frac{2\sinh(y)}{y}= \lim_{y\to 0}\frac{e^{y}- e^{-y}}{y}= \lim_{y\to 0}\frac{e^y+ e^{-y}}{1}=2, 
 \end{equation*}
verifying the formula for all $y$.
\end{proof}

\section{\texorpdfstring{$L^p$-polarity and $L^p$-Mahler volumes}{Lᵖ-polarity and 
Lᵖ-Mahler volumes}}\label{S2}

In \S \ref{S21}, we motivate the definition of the $L^p$-polar body $K^{\circ,p}$ (Definition \ref{LpKdef}) and prove Theorem \ref{LpKwelldefined}. 
In \S \ref{continuityMpSection},
we establish the continuity of $K^{\circ,p}$ in $p$ (Lemma \ref{LpKpcont}), and show that for $p$ converging to 0, $K^{\circ,p}$ converges either to $\R^n$ or a half-space (Proposition \ref{Kcirc0prop}). In \S \ref{S22} we generalize (\ref{0.1})
to a lower bound of $\M$ in terms of $\M_p$, for all $p>0$, for bodies with $b(K)=0$ (Lemma \ref{mKmpKeq}). In \S \ref{S23} and \S \ref{hpDiamondSection} calculations for $\M_p([-1,1]^n)$ and $h_{p,B_1^n}$ are carried out and used to numerically approximate $\M_p(B_1^3)$, providing evidence that $\M_p([-1,1]^3)<\M_p(B_1^3)$ when $p<\infty$ (Figure \ref{FigureDiamondCube}).

\subsection{\texorpdfstring{The $L^p$-polar body}{The Lᵖ-polar body}}\label{S21}
\subsubsection{Motivating the definition}
The support function of a convex body is convex and $1$-homogeneous and hence its sublevel set 
\begin{equation*}
    K^\circ\defeq \{y\in \R^n: h_{K}(y)\leq 1\}
\end{equation*}
defines a convex body such that $\M(K)=\M_\infty(K)=
|K|\int_{\R^n} e^{-h_K(y)}\dif y
=n! |K||K^\circ|$. 
This is special for the case $p=\infty$. To see why,
first recall the definition (\ref{Mpdef}), $\M_p(K)\defeq |K|\int_{\R^n}e^{-h_{p,K}}$. 
Yet despite the suggestive notation, for $p\in(0,\infty)$, $h_{p,K}$ is not the support function of a convex body since it is not $1$-homogeneous. On the other hand, by Lemma \ref{hpKconvex}, $h_{p,K}$ is convex and hence the sublevel set $\{h_{p,K}\leq 1\}\defeq \{y\in\R^n: h_{p,K}(y)\leq 1\}$ is a convex body. Nonetheless, the volume of $\{h_{p,K}\leq 1\}$ is not related to $\M_p(K)$
since despite having
\begin{equation*}
    \int_{\R^n} e^{-h_{p,K}(x)}\dif x= \int_{-\infty}^\infty e^{-t}|\{h_{p,K}\leq t\}|\dif t,
\end{equation*}
without $1$-homogeneity it is not clear how $\{h_{p,K}\leq t\}$ relates to $\{h_{p,K}\leq 1\}$ for all $t$.

In order to properly define the ``$L^p$-polar" body we replace
$h_{p,K}$ by a 1-homogeneous 
cousin. An equivalent way of defining a convex body $L$ is via its `norm':
\begin{equation}\label{ConvexBodyNorm}
    \|x\|_L\defeq \inf\{t>0: x\in tL\}. 
\end{equation} 
This is a norm only when $L$ is symmetric, but it is always positively 1-homogeneous and sub-additive with $L= \{x\in \R^n: \|x\|_L\leq 1\}$ \cite[Theorem 4.3]{gruber}. Given such a `norm', the volume of $L$ can
be expressed as an integral over the sphere, 
\begin{equation}\label{normVolume}
\begin{aligned}
    |L|&= \int_{\{x\in \R^n: \|x\|_L\leq 1\}} \dif x\\
    &= \int_{\{(r,u)\in [0,\infty)\times \partial B_2^n: \|ru\|_{L}\leq 1\}}r^{n-1}\dif r\dif u \\
    &= \int_{\partial B_2^n}\int_{r=0}^{1/\|u\|_L} r^{n-1}\dif r\dif u\\
    &= \frac1n \int_{\partial B_2^n}\frac{\dif u}{\|u\|_L^n}. 
\end{aligned}
\end{equation}
Looking at \eqref{normVolume} one may be able to recover the `norm' of a convex body by writing its volume as an integral over $\partial B_2^n$. Our aim is to define a convex body $K^{\circ,p}$ with volume $|K^{\circ,p}|= \frac{1}{n!} \int e^{-h_{p,K}}$. Starting from the volume we guess its norm: we need to write $\int e^{-h_{p,K}}$ as an integral on the sphere matching \eqref{normVolume},
\begin{equation}\label{LpKvolume}
    \begin{aligned}
        |K^{\circ,p}|&= \frac{1}{n!} \int_{\R^n} e^{-h_{p,K}(y)}\dif y\\&= \frac{1}{n!} \int_{\partial B_2^n}\int_0^\infty e^{-h_{p,K}(ru)} r^{n-1}\dif r\dif u\\
        &= \frac1n \int_{\partial B_2^n} \frac{\dif u}{\left[\left(\frac{1}{(n-1)!}\int_0^\infty r^{n-1} e^{-h_{p,K}(ru)}\dif r\right)^{-\frac1n}\right]^n}.
    \end{aligned}
\end{equation}
This justifies the definition of $\|\cdot\|_{K^{\circ,p}}$ via \eqref{opNorm} and $K^{\circ,p}$ as the convex body associated to that `norm' (Definition \ref{LpKdef}).

\subsubsection{Proof of
Theorem \ref{LpKwelldefined}}
In this subsection 
we conclude the proof
of Theorem \ref{LpKwelldefined}.
We start with two lemmas.

\begin{lemma}\label{nonemptyInterior}
Let $0<p<q$ and recall \eqref{opNorm}. For a compact body $K$, $\|\cdot\|_{K^{\circ,p}}\leq \|\cdot\|_{K^{\circ,q}}\leq h_K(\cdot)$. In particular,
    $K^\circ\subset K^{\circ,q}\subset K^{\circ, p}$. 
\end{lemma}
Note the support function of a compact body coincides with the `norm' of the polar body \eqref{ConvexBodyNorm}, 
\begin{equation}\label{PolarSupportNorm}
    h_K(\cdot)=\|\cdot\|_{K^\circ}, 
\end{equation}
since $y\in K^\circ$ if and only if $h_K(y)\leq 1$ \cite[p. 56]{gruber}. Also, for convex bodies \cite[Corollary 13.1.1]{rockafellar}, 
\begin{equation}\label{inclusions}
    L\subset K \quad \text{ if and only if } \quad \|\cdot\|_K\leq \|\cdot\|_L \quad \text{ if and only if } \quad  h_{K^\circ}(\cdot)\leq h_{L^\circ}(\cdot).
\end{equation}

\begin{lemma}\label{LpKbounded}
Let $p\in (0,\infty]$. 
    For a convex body $K\subset \R^n$, $K^{\circ,p}$ is bounded (compact) if and only if $0\in\mathrm{int}\,K$.  
\end{lemma}
In particular, since $K^{\circ}$ has non-empty interior \cite[Corollary 14.6.1]{rockafellar}, Lemma \ref{nonemptyInterior} shows that $K^{\circ,p}$ is non-empty and has non-empty interior. 

Before proving Lemmas \ref{nonemptyInterior} and \ref{LpKbounded}, let us recall  an integral formula regarding $1$-homogeneous functions that will be useful throughout. 

\begin{claim}\label{1homogeneous}
    Let $k\in\mathbb{N}$. For a 1-homogeneous function $f: \R^n\to \R$ and $x\in \R^n$ with $f(x)\neq 0$, 
    \begin{equation*}
        \int_0^\infty r^{k-1} e^{-f(rx)}\dif r= \frac{(k-1)!}{f(x)^k}. 
    \end{equation*}
\end{claim}
\begin{proof}
    By homogeneity of $f$, $f(rx)= rf(x)$ for all $r>0$. Setting $\rho= r f(x)$, 
    \begin{equation*}
        \begin{aligned}
            \int_0^\infty r^{k-1} e^{-f(rx)}\dif x&= \int_0^\infty r^{k-1} e^{-r f(x)}\dif r\\
            &= \int_0^\infty \frac{\rho^{k-1}}{f(x)^{k-1}} e^{-\rho}\frac{\dif \rho}{f(x)}\\
            &= \frac{1}{f(x)^k}\int_{0}^\infty \rho^{k-1} e^{-\rho}\dif \rho 
            = \frac{(k-1)!}{f(x)^k},
        \end{aligned}
    \end{equation*}
    as claimed.
\end{proof}

\begin{proof}[Proof of Lemma \ref{nonemptyInterior}.]
    Let $p\leq q$.
    By Lemma \ref{list} (v),  $h_{p,K}\leq h_{q,K}$. 
    Thus by \eqref{opNorm},
    \begin{equation}\label{pLq}
        \begin{aligned}
            \|x\|_{K^{\circ,p}}&= \left( \frac{1}{(n-1)!}\int_0^\infty r^{n-1} e^{-h_{p,K}(rx)}\dif r\right)^{-\frac1n}\\
            &\leq \left( \frac{1}{(n-1)!}\int_0^\infty r^{n-1} e^{-h_{q,K}(rx)}\dif r\right)^{-\frac1n}= \|x\|_{K^{\circ,q}}. 
        \end{aligned}
    \end{equation}
    So, for $x\in K^{\circ,q}$, $\|x\|_{K^{\circ,p}}\leq \|x\|_{K^{\circ,q}}\leq 1$, thus $x\in K^{\circ,p}$. In addition, by homogeneity of $h_K$, Claim \ref{1homogeneous} gives
    \begin{equation}\label{hpK}
    \begin{aligned}
        \frac{1}{(n-1)!}\int_0^\infty r^{n-1} e^{-h_K(rx)}\dif r&=
      \frac{1}{h_K(x)^n}. 
    \end{aligned}
    \end{equation}
    Since by Lemma \ref{list} (v) $h_{p,K}\leq h_K$, then by \eqref{PolarSupportNorm}, \eqref{hpK} and a computation similar to \eqref{pLq}, $\|x\|_{K^{\circ,p}}\leq h_K(x)= \|x\|_{K^{\circ}}$, so 
     $K^{\circ}\subset K^{\circ,p}$ by \eqref{inclusions}.  
\end{proof}

For the proof of Lemma \ref{LpKbounded}, it is useful to know that the $L^p$-polars of $[-1,1]^n$ are bounded.
\begin{claim}\label{cubeppolarBounded}
    For $p\in (0,\infty]$, $([-1,1]^n)^{\circ,p}$ is bounded. 
\end{claim}
\begin{proof}
    Since $b([-1,1]^n)=0$, by Lemma \ref{hpKineq} with $\lambda=\frac12$, $h_{[-1,1]^n}(\frac{ry}{2})\leq h_{p,[-1,1]^n}(ry)+ \frac{n}{p} \log 2$ for all $y\in\R^n$ and $r>0$. Thus by \eqref{opNorm},
    {\allowdisplaybreaks
    \begin{align*}
            \|y\|_{([-1,1]^n)^{\circ,p}}&= \left( \frac{1}{(n-1)!}\int_0^\infty r^{n-1} e^{-h_{p,[-1,1]^n}(ry)}\dif r\right)^{-\frac1n} \\
            &\geq \left( \frac{1}{(n-1)!}\int_0^\infty r^{n-1} e^{-h_{[-1,1]^n}(r\frac{y}{2})} e^{\frac{n}{p}\log 2}\dif r \right)^{-\frac1n}\\
            &= e^{-\frac{\log 2}{p}} \left( \frac{1}{h_{[-1,1]^n}(\frac{y}{2})^n}\right)^{-\frac1n}
            = \frac{e^{-\frac{\log 2}{p}}}{2} h_{[-1,1]^n}(y)
            =  \frac{e^{-\frac{\log 2}{p}}}{2} \|y\|_{([-1,1]^n)^\circ}, 
    \end{align*}
    }
    by Claim \ref{1homogeneous}, the homogeneity of $h_{[-1,1]^n}$, and \eqref{PolarSupportNorm}. By \eqref{inclusions},
    \begin{equation*}
        ([-1,1]^n)^{\circ,p}\subset 2 e^{\frac{\log 2}{p}} ([-1,1]^n)^\circ= 2 e^{\frac{\log 2}{p}} B_1^n, 
    \end{equation*}
    which is bounded. 
\end{proof}

\begin{proof}[Proof of Lemma \ref{LpKbounded}]
    Assume $0\in\mathrm{int}\,K$ and let $r>0$ be such that $[-r,r]^n\subset K$. Then,
    \begin{equation}\label{hpKrEq}
    \begin{aligned}
        h_{p,K}(y)&= \frac1p\log\left(\int_{K}e^{p\langle x,y\rangle}\frac{\dif x}{|K|}\right)\\
        &\geq \frac1p\log\left(\int_{[-r,r]^n}e^{p\langle x,y\rangle}\frac{\dif x}{|K|}\right)\\
        &=\frac1p \log\left(\int_{[-r,r]^n} e^{p\langle x,y\rangle}\frac{\dif x}{|[-r,r]^n|} \frac{|[-r,r]^n|}{|K|}\right)\\
        &= h_{p, [-r,r]^n}(y)+ \frac1p\log\frac{(2r)^n}{|K|}.
    \end{aligned}
    \end{equation}
    Using this and \eqref{opNorm},
    \begin{equation*}
        \|y\|_{K^{\circ,p}}\geq \left(\frac{1}{(n-1)!}\int_{0}^\infty r^{n-1} e^{-h_{{p,[-r,r]^n}}(\rho y)-\frac1p\log\frac{(2r)^n}{|K|}} \dif \rho\right)^{-\frac1n}= \frac{(2r)^{\frac1p}}{|K|^{\frac{1}{np}}}\| y\|_{([-r,r]^n)^{\circ,p}}.
    \end{equation*}
    Thus, by \eqref{inclusions}, $K^{\circ,p}\subset \frac{|K|^{\frac{1}{np}}}{(2r)^{\frac1p}} ([-r,r]^{n})^{\circ,p}$ which is
    bounded by Claim \ref{cubeppolarBounded}.

    For the converse, we claim that if $0\notin\mathrm{int}\,K$ then $K^{\circ,p}$ is unbounded. By Lemma \ref{nonemptyInterior}, $K^{\circ}\subset K^{\circ,p}$ so it is enough to show $K^\circ$ is unbounded. This is classical \cite[Corollary 14.5.1]{rockafellar}.
\end{proof}

\begin{proof}[Proof of Theorem \ref{LpKwelldefined}]
    By Proposition \ref{norm_phi}, $\|\cdot\|_{K^{\circ,p}}$ is positively 1-homogeneous and sub-additive. The non-emptiness of the interior of $K^{\circ,p}$ follows from Lemma \ref{nonemptyInterior} since $K^\circ$ has non-empty interior. It is also closed and convex as the sublevel set of a continuous, convex function. Convexity of $\|\cdot\|_{K^{\circ,p}}$ follows from its 1-homogeneity and sub-additivity: for $x,y\in \R^n$ and $\lambda \in[0,1]$, 
    \begin{equation*}
        \|(1-\lambda)x+ \lambda y\|_{K^{\circ,p}}\leq \|(1-\lambda)x\|_{K^{\circ,p}}+ \|\lambda y\|_{K^{\circ,p}}= (1-\lambda)\|x\|_{K^{\circ,p}}+ \lambda\|y\|_{K^{\circ,p}}. 
    \end{equation*}
    If $K$ is symmetric, i.e., $-K= K$, then 
    $$
    h_{p,K}(-y)= \frac1p\log\int_K e^{p\langle x,-y\rangle}\frac{\dif x}{|K|}= \frac1p\log\int_{-K} e^{p\langle -z,-y\rangle}\frac{\dif z}{|K|}= \frac1p\log\int_{K}e^{p\langle z,y\rangle}\frac{\dif z}{|K|}= h_{p,K}(y).
    $$
    Therefore, $\|-x\|_{K^{\circ,p}}= \left( \int_0^\infty r^{n-1} e^{-h_{p,K}(-rx)}\dif r\right)^{-\frac1n}= \left( \int_0^\infty r^{n-1} e^{-h_{p,K}(rx)}\dif r\right)^{-\frac1n}= \|x\|_{K^{\circ,p}}$, making $\|\cdot\|_{K^{\circ,p}}$ a norm, and $K^{\circ,p}$ symmetric. Finally, \eqref{MpKeq} follows from \eqref{LpKvolume} and the definition of $\|\cdot\|_{\LpK}$. 
\end{proof}

\begin{remark}
    Ball showed that for a convex function $\phi: \R^n\to \R\cup\{\infty\}$
    and $q\ge1$, setting
$    \|y\|_{\phi,q}\defeq \left(\int_{0}^\infty r^{q-1}e^{-\phi(ry)}\dif r \right)^{-\frac1q}$ defines a positively 1-homogeneous, sub-additive function that is also a norm when $\phi$ is even \cite[Theorem 5]{ball2}. Then, 
$    \{y\in \R^n: \|y\|_{\phi,q}\leq 1\}$
defines a convex body 
(for even $\phi$  \cite[Theorem 5]{ball2},
for general $\phi$ \cite[Theorem 2.2]{klartag}). 
In this notation, \eqref{opNorm} reads $\|y\|_{K^{\circ,p}}^n= (n-1)! \|y\|_{h_{p,K},n}^n$.
For a statement and proof of Ball's theorem see Proposition \ref{norm_phi} below.
\end{remark}

\subsection{\texorpdfstring{Continuity of $\M_p$ and limiting cases}{Continuity of Mₚ and limiting cases}}\label{continuityMpSection}

First, we translate (pointwise) convergence of $L^p$-support functions to convergence of the norms of the $L^p$-polars.

\begin{lemma}\label{LpKpcont}
Let $0<p<q\leq \infty$.
For a compact body $K\subset \R^n$, $\LpK\subset K^{\circ,q}$ and
\begin{equation*}
    \lim_{p\to q}\|x\|_{\LpK}= \|x\|_{K^{\circ, q}}. 
\end{equation*}
In particular, $\bigcap_{0<p<q} \LpK= K^{\circ,q}$. 
\end{lemma}
\begin{proof}
By Corollary \ref{hpLimit}, $h_{p,K}$ increases to $h_{q,K}$ as $p$ increases to $q$. Therefore, one may use the monotone convergence theorem \cite[2.14]{folland} to take the limit under the integral in the definition of $h_{p,K}$,
\begin{equation*}
\begin{aligned}
    \lim_{p\to \infty} \|x\|_{\LpK}&=  \lim_{p\to \infty}\left( \frac{1}{(n-1)!} \int_{r=0}^\infty r^{n-1} e^{-h_{p,K}(rx)}\dif r\right)^{-1/n}\\
    &=  \left( \frac{1}{(n-1)!} \int_{r=0}^\infty \lim_{p\to q} \Big( r^{n-1} e^{-h_{p,K}(rx)}\Big) \dif r\right)^{-1/n}\\
    &=  \left( \frac{1}{(n-1)!} \int_{r=0}^\infty r^{n-1} e^{-h_{q,K}(rx)}\dif r\right)^{-1/n}
    = \|x\|_{K^{\circ,q}}.
\end{aligned}
\end{equation*}
\end{proof}

\subsubsection{\texorpdfstring{The cases $p=0$ and $p=\infty$}{The cases p=0 nad p=∞}}
\lb{casep0SubSec}
By Lemma \ref{LpKpcont}, as $p\to \infty$, $K^{\circ,p}$ converges to the polar body $K^\circ$ \eqref{polarEq}. 
In this subsection we focus on the other extreme case $p=0$ and show that in the limit $p\to 0$, $K^{\circ,p}$ converges either to $\R^n$ or to a half-space, depending on whether $b(K)=0$ or not. 
\begin{proposition}\label{Kcirc0prop}
    For a compact body $K\subset \R^n$, 
    \begin{equation*}
        \lim_{p\to 0}\|y\|_{K^{\circ,p}}= \langle y, b(K)\rangle, 
    \end{equation*}
    and
    \begin{equation*}
        \overline{\bigcup_{p>0} K^{\circ,p}}= \{y\in \R^n: \langle y, b(K)\rangle\leq 1\}. 
    \end{equation*}
\end{proposition}

Proposition \ref{Kcirc0prop} and Lemma \ref{LpKpcont} imply the following inclusion for all $K^{\circ,p}$.
\begin{corollary}\label{Kcirc0cor}
    For a compact body $K\subset \R^n$, $K^{\circ,p}\subset \{y\in\R^n: \langle y, b(K)\rangle\leq 1\}$ for all $p\in(0,\infty]$. 
\end{corollary}
The statement of Corollary \ref{Kcirc0cor} is trivial when $p=\infty$ because $b(K)\in K$, thus, by definition of the polar, $\langle y, b(K)\rangle\leq 1$ for all $y\in K^{\circ}$. The proof of Proposition \ref{Kcirc0prop} follows from the fact that 
the $L^p$-support functions converge to a linear function as $p\to 0$. 

\begin{lemma}\label{Kcirc0lemma}
    For a compact body $K\subset \R^n$ and $y\in \R^n$, 
    \begin{equation*}
        \lim_{p\to 0} h_{p,K}(y)= \langle y, b(K)\rangle. 
    \end{equation*}
\end{lemma}
\begin{proof}
    Expanding the exponential, 
    \begin{equation*}
    \begin{aligned}
        h_{p,K}(y)&= \frac1p\log\int_K e^{p\langle x,y\rangle}\frac{\dif x}{|K|} \\
        &=\frac1p \log\left(\int_K 1+ p\langle x,y\rangle+ O(p^2)\frac{\dif x}{|K|} \right) \\
        &= \frac1p \log\left(1+ p\langle y, b(K)\rangle+ O(p^2)\right). 
    \end{aligned}
    \end{equation*}
By L'H\^{o}pital's rule, 
\begin{equation*}
    \lim_{p\to 0}h_{p,K}(y)= \lim_{p\to 0}\frac{\log(1+ p\langle y, b(K)\rangle+ O(p^2))}{p}= \lim_{p\to 0}\frac{\langle y, b(K)\rangle+ O(p)}{1+ p\langle y, b(K)\rangle+ O(p^2)}= \langle y, b(K)\rangle. 
\end{equation*}
\textit{Alternative proof:} 
\begin{equation*}
    \begin{aligned}
        \lim_{p\to 0} h_{p,K}(y)&= \lim_{p\to 0}\frac1p \log\int_K e^{p\langle x,y\rangle}\frac{\dif x}{|K|} \\
        &= \lim_{p\to 0} \frac{\int_K \langle x,y\rangle e^{p\langle x,y\rangle}\frac{\dif x}{|K|}}{\int_K e^{p\langle x,y\rangle}\frac{\dif x}{|K|}} 
        = \int_K \langle x,y\rangle \frac{\dif x}{|K|}
        = \langle y, b(K)\rangle,
    \end{aligned}
\end{equation*}
again by L'H\^{o}pital's rule.
\end{proof}

\begin{proof}[Proof of Proposition \ref{Kcirc0prop}.]
For $y\in \R^n$ with $\langle y, b(K)\rangle\neq 0$, 
by the monotone convergence theorem \cite[\S 2.14]{folland} and Lemma \ref{Kcirc0lemma},
\begin{equation*}
    \begin{aligned}
        \lim_{p\to 0}\|y\|_{K^{\circ,p}}&= \lim_{p\to 0} \left(\frac{1}{(n-1)!}\int_{0}^\infty r^{n-1} e^{-h_{p,K}(ry)} \dif r\right)^{-\frac1n}\\
        &=\left( \frac{1}{(n-1)!}\int_0^\infty r^{n-1} e^{-r\langle y, b(K)\rangle}\dif r\right)^{-\frac1n} \\
        &= \langle y, b(K)\rangle, 
    \end{aligned}
\end{equation*}
where Claim \ref{1homogeneous} was used on the 1-homogeneous $y\mapsto \langle y, b(K)\rangle$ . If $\langle y, b(K)\rangle=0$, similarly,
\begin{equation*}
    \lim_{p\to 0} \|y\|_{K^{\circ,p}}= \left(\frac{1}{(n-1)!} \int_0^\infty r^{n-1}\dif r\right)^{-\frac1n}= 0= \langle y,b(K)\rangle. 
\end{equation*}
\end{proof}

Proposition \ref{Kcirc0prop} motivates the following definition. 
\begin{definition}\label{Kcirc0def}
    For a  compact body $K\subset \R^n$, let
    \begin{equation*}
        K^{\circ,0}\defeq \{y\in \R^n: \langle y, b(K)\rangle\leq 1\}.
    \end{equation*}
\end{definition}

For a set $A\subset \R^n$, denote by 
\begin{equation*}
    \mathrm{co}\, A
\end{equation*}
its \textit{convex hull} defined as the smallest convex set in $\R^n$
containing $A$. 

\begin{example}\label{simplexpPolar}
    The polar body of the standard 2-dimensional simplex $\Delta_2$ is given by the intersection of two half-spaces
    \begin{equation*}
        \Delta_2^\circ= \{(x,y)\in \R^2: x\leq 1 \text{ and } y\leq 1\}.
    \end{equation*}
    That is because $\Delta_2= \mathrm{co}\{(0,0), (1,0), (0,1)\}$, thus $(x,y)\in \Delta^\circ$ if and only if $x= \langle (x,y), (1,0)\rangle\leq 1$ and $y=\langle (x,y), (0,1)\rangle\leq 1$. In addition, $|\Delta_2|= 1/2$, thus the $x$ coordinate of the barycenter of $\Delta_2$, 
    \begin{equation*}
        \frac{1}{|\Delta_2|}\int_{\Delta_2}x\dif x\dif y= 2\int_{x=0}^1\int_{y=0}^{1-x} x\dif y \dif x= 2\int_0^1 x(1-x)\dif x= \frac13.
    \end{equation*}
    Similarly, $\frac{1}{|\Delta_2|}\int_{\Delta_2} y\dif y= \frac13$, and hence $b(\Delta_2)= (1/3,1/3)$. As a result, 
    \begin{equation*}
        \Delta^{\circ,0}= \{(x,y)\in \R^2: x+y\leq 3\}.
    \end{equation*}
    By Lemma \ref{LpKpcont}, 
     $   \{x\leq 1\}\cap \{y\leq 1\}\subset (\Delta_2)^{\circ,p}\subset \{x+y\leq3\},$ 
    for all $p\geq 0$. By direct calculation,
    \begin{equation*}
        h_{p,\Delta_2}(x,y)= \frac1p\log\left( \frac{\frac{e^{px}-1}{px}- \frac{e^{py}-1}{py}}{p\frac{x-y}{2}}\right), 
    \end{equation*}
    from which we get Figure \ref{fig:f3} in \S\ref{IntroSec}. By Lemma \ref{list} (ii), 
    \begin{equation*}
        h_{p, \Delta_2-b(\Delta_2)}(y)= h_{p,\Delta_2}- \langle (x,y), b(\Delta_2)\rangle= \frac1p\log\left( \frac{\frac{e^{px}-1}{px}- \frac{e^{py}-1}{py}}{p\frac{x-y}{2}}\right)- \frac{x}{3}-\frac{y}{3},  
    \end{equation*}
    leading to Figure \ref{mainFigure} (c) in \S\ref{IntroSec}.
\end{example}

\subsection{\texorpdfstring{Inequalities between $\M_p$ and $\M$}{Inequalities between Mₚ and M}}\label{S22}

By Lemma \ref{list} (v), $h_{p,K}\leq h_K$ for all $p$, thus
\begin{equation}\label{MlessM_p}
    \M(K)\leq \M_p(K). 
\end{equation}
In view of Lemma \ref{hpKineq}, a reverse inequality holds under the extra assumption of $b(K)=0$. 
\begin{lemma}\label{mKmpKineq}
    Let $p\in(0,\infty)$. For a convex body $K\subset \R^n$ with $b(K)=0$, 
    \begin{equation*}
        \left( \frac{p}{(1+p)^{1+\frac1p}}\right)^n \M_p(K)
        \leq \M(K). 
    \end{equation*}
    Hence, $\lim_{p\to \infty}\M_p(K)= \M(K)$. 
\end{lemma}

\bremark
Lemma \ref{mKmpKineq} generalizes the Bergman kernel
inequality (\ref{0.1}) (recall \eqref{M1KEq}).
\eremark

\begin{proof}
    Assume $b(K)=0$. Lemma \ref{hpKineq} applies to give, 
    \begin{equation}\label{mKmpKeq}
        \begin{aligned}
            \M(K)&= |K|\int_{\R^n} e^{-h_K(y)}\dif y\\
            &\geq (1-\lambda)^{\frac{n}{p}}|K|\int_{\R^n} e^{-h_{p,K}(y/\lambda)}\dif y \\
            &= (1-\lambda)^{\frac{n}{p}} \lambda^n |K|\int_{\R^n} e^{-h_{p,K}(y)}\dif y
            = \Big( (1-\lambda)^{\frac{1}{p}}\lambda\Big)^n \M_p(K).
        \end{aligned}
    \end{equation}
    It remains to maximize $f(\lambda)\defeq (1-\lambda)^{\frac{1}{p}}\lambda$. The derivative
\begin{equation}\label{lambdaOptimize}
    f'(\lambda)= -\frac{1}{p}(1-\lambda)^{\frac{1}{p}-1}\lambda+ (1-\lambda)^{\frac{1}{p}}= (1-\lambda)^{\frac{1}{p}-1} \left( -\frac{\lambda}{p}+ 1-\lambda\right), 
 \end{equation}
 is positive for $\lambda\in(0,\frac{p}{p+1})$ and non-positive
 for  $\lambda\in(\frac{p}{p+1},1)$, so plugging $\lambda=\frac{p}{p+1}$ in (\ref{mKmpKeq}) proves the claim.

    Finally, note
    \begin{equation*}
        \lim_{p\to \infty} \frac{p}{(1+p)^{1+\frac1p}}= 
        \frac{1}{(1+p)^{\frac1p}}- \frac{1}{(1+p)^{1+\frac1p}}
        =
        1,
    \end{equation*}
    thus $\lim_{p\to \infty}\M_p(K)= \M(K)$.
\end{proof}

\begin{remark}\label{isomorphicRemark}
For convex $K\subset \R^n$ with $b(K)=0$, and  any $\lambda\in (0,1)$, 
\begin{equation*}
    \begin{aligned}
        h_{K}(y)&= \left( \frac{1}{(n-1)!}\int_0^\infty r^{n-1} e^{-h_{K}(ry)}\dif r\right)^{-\frac1n} \\
        &\leq \left( \frac{1}{(n-1)!} \int_0^\infty r^{n-1} e^{-h_{p,K}(\frac{ry}{\lambda})+ \frac{n}{p}\log(1-\lambda)} \dif r\right)^{-\frac1n} 
        = \frac{\|y\|_{K^{\circ,p}}}{(1-\lambda)^{\frac1p}\lambda}, 
    \end{aligned}
\end{equation*}
where we used Lemma \ref{hpKineq} and Claim \ref{1homogeneous}.
So,
\begin{equation*}    K^{\circ}\subset K^{\circ,p}\subset \frac{1}{(1-\lambda)^{1/p}\lambda} K^{\circ}\subset \frac{(1+p)^{1+\frac1p}}{p}K^\circ,
\end{equation*}
(optimizing over $\lambda$ as in the proof 
of Lemma \ref{mKmpKineq}).
This yields inclusions independent of $K$ or the dimension. 
Thus for convex bodies with $b(K)=0$, all the $L^p$-polars $K^{\circ,p}$ are `isomorphic' to (each other and to) the classical polar body $K^\circ$. 
They are also `isomorphic' to the sub-level sets of $h_{1,K}$, cf. \cite[p. 16]{klartag-Emilman}, \cite[Lemma 2.2]{klartag-milman}.
Furthermore, the latter (at
least in the symmetric case) are `isomorphic'
to the Lutwak--Zhang centroid bodies from Remark \ref{LZRem} \cite[Lemma 2.3]{klartag-Emilman}.
Nonetheless, `isomorphic' in this context means that inclusions in both directions exist by dilations independent of dimension. Consequently, such equivalences are not typically helpful when one is concerned with sharp lower bounds as in the Mahler Conjectures. Given Lemma \ref{mKmpKineq} and the remarks in the Introduction, we believe that our $L^p$-polars could be helpful in the pursuit of sharp bounds, e.g.,
as in the Mahler Conjectures.
\end{remark}

\subsection{The cube}\label{S23}
The next lemma computes the $L^p$-Mahler volume of the cube. 
\begin{lemma}\label{cubeMp}
    For $p\in(0,\infty)$, 
    \begin{equation*}
        \M_p([-1,1]^n)= 4^n \left(\frac1p\int_0^\infty\left( \frac{y}{\sinh(y)}\right)^{1/p}\dif y\right)^{n}.
    \end{equation*}
\end{lemma}

Note that 
\begin{equation*}    \M_p([-1,1]^n)= \left( \M_p([-1,1])\right)^n, 
\end{equation*}
in agreement with Remark \ref{tensor}.

\begin{proof}
    By Lemma \ref{cube_hp}, 
    \begin{equation*}
        \begin{aligned}
            \M_p([-1,1]^n)&= |[-1,1]^n| \int_{\R^n} e^{-h_{p,[-1,1]^n}(y)}\dif y= 2^n \int_{\R^n} \prod_{i=1}^n \left(\frac{py_i}{\sinh(py_i)}\right)^{1/p} \dif y \\
            &=  2^n\prod_{i=1}^n\int_{\R}\left(\frac{p y_i}{\sinh(p y_i)}\right)^{1/p}\dif y= 2^n \left(\int_{\R}\left( \frac{p y}{\sinh(py)}\right)^{1/p}\dif y\right)^{n}.
        \end{aligned}
    \end{equation*}
    The claim follows from the evenness of $\frac{py}{\sinh(py)}$ and the change of variables $z=py$.
\end{proof}

In the notation of \S\ref{BergmanSubSec}, B\l{}ocki obtained 
$|[-1,1]^n|^2\mathcal{K}_{T_{[-1,1]^n}}(0,0)= (\pi/4)^n$  \cite[(7)]{blocki}.
This agrees with our next corollary as $\M_1(K)= (4\pi)^n |K|^2\mathcal{K}_{T_K}(0,0)$ by \eqref{diagonal_berg}.

\begin{corollary}\label{cubeM1}
    $\M_1([-1,1]^n)= \pi^{2n}$.
\end{corollary}
\begin{proof}
    Setting $p=1$ in Lemma \ref{cubeMp}, $\M_1([-1,1]^n)=\left( 2\int_{\R}\frac{y}{\sinh(y)}\dif y\right)^n= \left(4\int_0^\infty \frac{y}{\sinh(y)}\dif y\right)^n$, because $y/\sinh(y)$ is even. Using $(1-x)^{-1}= \sum_{k=0}^\infty x^k$ for $0<x<1$, expand the integrand 
    \begin{equation*}
        \frac{y}{\sinh(y)}= \frac{2y}{e^{y}-e^{-y}}= \frac{2ye^{-y}}{1-e^{-2y}}= 2ye^{-y}\sum_{k=0}^\infty e^{-2ky}= \sum_{y=0} 2y e^{-(2k+1)y}. 
    \end{equation*}
    Therefore, by integration by parts
    \begin{equation*}
    \begin{aligned}
        \int_{0}^\infty \frac{y}{\sinh(y)}\dif y&= \sum_{k=0}^\infty \int_0^\infty 2y e^{-(2k+1)y}\dif y= \sum_{k=0}^\infty \frac{2}{2k+1} \int_0^\infty e^{-(2k+1)y}\dif y\\
        &= 2\sum_{k=0}\frac{1}{(2k+1)^2}= 2\left( \sum_{k=1}^\infty \frac{1}{k^2}-\sum_{k=1}^\infty \frac{1}{(2k)^2}\right)\\
        &= 2\left(\sum_{k=1}^\infty \frac{1}{k^2}- \frac14\sum_{k=1}^\infty\frac{1}{k^2}\right)= \frac{3}{2}\sum_{k=0}^\infty \frac{1}{k^2}= \frac32 \frac{\pi^2}{6}= \frac{\pi^2}{4}, 
    \end{aligned}
    \end{equation*}
    and hence $\M_1([-1,1]^n)= \left( 4\int_0^\infty \frac{y}{\sinh(y)}\dif y\right)^n= \pi^{2n}$.
\end{proof}

A numerical approximation of $\M_p([-1,1])$ gives following graph.
\begin{figure}[H]
    \centering
    \includegraphics[scale=.6]{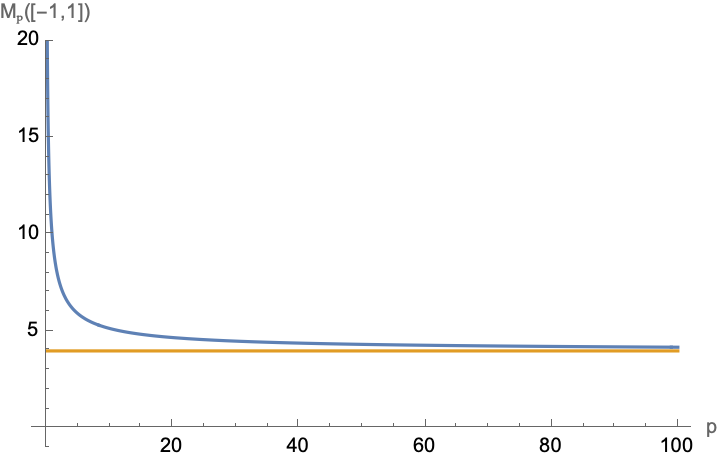}
    \caption{\small $\M_p([-1,1])$ for $p\in (0,100)$ compared to $\M([-1,1])=4$.}
\end{figure}

\subsection{Cube, diamond, and uniqueness of minimizers}
\label{hpDiamondSection}

Let $B_\infty^n=[-1,1]^n$ and $B_1^n=(B_\infty^n)^\circ$ be the cube and diamond (recall \eqref{BqnEq}).
The $L^p$-support function of the cube was computed in Lemma \ref{cube_hp}
and its $L^p$-Mahler volume is given by Lemma \ref{cubeMp}.
Lemma \ref{diamondpSupport} below is the considerably harder computation
of the $L^p$-support function of the diamond.

Lemmas \ref{cubeMp} and \ref{diamondpSupport} allow to 
compare the $L^p$-Mahler volumes of the cube and the diamond.
We carried this out numerically for $n=3$ and those computations lead
to Figure \ref{FigureDiamondCube} from the Introduction.
As discussed in \S\ref{LpMahlerConjSubSec}, this provides evidence 
that the cube is the unique minimizer for Conjecture \ref{pMahlerSym}. 

\begin{lemma}\label{diamondpSupport}
For $p\in(0,\infty)$,
    \begin{equation}
    \lb{diamondhpEq}
        h_{p, B_1^n}(y)= \frac1p\log\left( \frac{n!}{p^n}\sum_{j=1}^n \frac{y_j^{n-2} (e^{py_j} + (-1)^n e^{-py_j}) }{(y_j^2- y_1^2)\cdots (y_j^2-y_{j-1}^2)(y_j^2-y_{j+1}^2)\cdots (y_j^2-y_n^2)}\right).
    \end{equation}
\end{lemma}

The special case $p=1$ 
of \eqref{diamondhpEq} was stated by B\l{}ocki in terms of Bergman kernels without proof \cite[pp. 96--97]{blocki2}. 

For the proof of Lemma \ref{diamondpSupport} we require the following claim. 

\begin{claim}\label{peculiarClaim}
For $n\geq 2$, and distinct $y_1, \ldots, y_n \in \R$,
\begin{equation*}
    \sum_{j=1}^n \frac{y_j^k}{(y_j-y_1)\cdots (y_j-y_{j-1}) (y_j-y_{j+1})\cdots (y_j-y_n)}= \begin{cases}
        0, & 0\leq k<n-1, \\
        1, & k=n-1.
    \end{cases}
\end{equation*}
\end{claim}
\begin{proof}
Consider the rational function
\begin{equation*}
    f: \C\ni z\mapsto \frac{z^k}{(z-y_2)\cdots(z-y_n)}+ \sum_{j=2}^n \frac{y_j^k}{(y_j-z)(y_j-y_2)\cdots (y_j-y_n)}\in\C\cup\{\infty\}, 
\end{equation*}
i.e., think of $y_1$ as a complex variable. 

The claim is that $f$ is a polynomial. It is enough to show that its poles at $y_2, \ldots, y_n$ are removable singularities. By symmetry, it is enough to do it for $y_2$. There are only two terms involving $(z-y_2)$ in the denominator. Write their sum 
as,\begin{equation*}    \begin{gathered}
        \frac{z^k}{(z-y_2)\ldots(z-y_n)}+ \frac{y_2^k}{(y_2-z)\cdots (y_2-y_n)}
        = \frac{\left( \frac{z^k}{(z-y_3)\cdots (z-y_n)}-\frac{y_2^k}{(y_2-y_3)\cdots(y_2-y_n)}\right)}{z-y_2} 
   \end{gathered}
\end{equation*}
   We claim the numerator can be written in the form $(z-y_2)p(z)$
   for some polynomial $p$.
    Indeed,
\begin{equation*}
\begin{aligned}
&\frac{z^k}{(z-y_3)\cdots (z-y_n)}-\frac{y_2^k}{(y_2-y_3)\cdots(y_2-y_n)} \\
        &= \frac{z^k}{(z-y_3)\cdots(z-y_n)}-\frac{y_2^k}{(z-y_3)\cdots (z-y_n)}
        +\frac{y_2^k}{(z-y_3)\cdots (z-y_n)}-\frac{y_2^k}{(y_2-y_3)\cdots (y_2-y_n)} \\
        &=\frac{(z-y_2)(z^{k-1}+ \cdots+ y_2^{k-1})}{(z-y_3)\cdots(z-y_n)} - y_2^k\frac{(z-y_3)\cdots(z-y_n)- (y_2-y_3)\cdots (y_2-y_n)}{(z-y_3)\cdots (z-y_n)(y_2-y_3)\cdots (y_2-y_n)}\\
         &=\frac{(z-y_2)(z^{k-1}+ \cdots+ y_2^{k-1})}{(z-y_3)\cdots(z-y_n)} - y_2^k\frac{(z-y_2)p(z)}{(z-y_3)\cdots (z-y_n)(y_2-y_3)\cdots (y_2-y_n)} \\
         &=\frac{z^{k-1}+ \cdots+ y_2^{k-1}}{(z-y_3)\cdots(z-y_n)} - y_2^k\frac{p(z)}{(z-y_3)\cdots (z-y_n)(y_2-y_3)\cdots (y_2-y_n)}, 
    \end{aligned}
\end{equation*}
where $p(z)$ is a polynomial such that
\begin{equation*}
    (z-y_3)\cdots(z-y_n)- (y_2-y_3)\cdots (y_2-y_n)= (z-y_2) p(z), 
\end{equation*}
since the left-hand side is a polynomial that vanishes at $y_2$. 
In sum, $f$ is a polynomial. In addition, 
\begin{equation*}
    \lim_{z\to\infty} f(z)=\begin{cases}
        0, & k< n-1, \\
        1, & k =n-1,
    \end{cases}
\end{equation*}
proving, by Liouville's theorem \cite[p. 122]{ahlfors}, that $f$ is constant (as a bounded, entire function) and equal to $0$ when $0<k<n-1$, or $1$ when $k=n-1$. 
\end{proof}

\begin{proof}[Proof of Lemma \ref{diamondpSupport}]
    Since $B_1^n$ is the union of $2^n$ simplices of volume $1/n!$, $|B_1^n|= 2^n/n!$. In addition, by splitting the integral into $2^n$ integrals over the simplex,
    \begin{equation}\label{DiamondInt1}
    \begin{aligned}
        e^{ph_{p, B_1^n}(y/p)}=\int_{B_1^n} e^{\langle x,y\rangle}\frac{\dif x}{|B_1^n|}
        &= n! \int_{\Delta_n} \cosh(x_1 y_1)\cdots \cosh(x_n y_n)\dif x.
    \end{aligned}
    \end{equation}
    The rest of the proof is by induction on $n$. 
    
    For $n=2$, by \eqref{DiamondInt1}, 
    \begin{equation*}
        \begin{aligned}
            e^{p h_{p, B_1^2}(y/p)}&= 2\int_{\Delta_2} \cosh(x_1 y_1)\cosh(x_2y_2)\dif x_1\dif x_2 \\
            & =2\int_0^{1} \cosh(x_1 y_1)\int_{0}^{1-x_1} \cosh(x_2 y_2)\dif x_2\dif x_1\\
            &= \int_0^1 \cosh(x_1y_1) \frac{\sinh((1-x_1)y_2)}{y_2}\dif x_1\\
            &= \frac{1}{y_2}\left[\frac{y_1\sinh(y_1x_1)\sinh((1-x_1)y_2)+ y_2\cosh(x_1 y_1)\cosh((1-x_1)y_2)}{y_1^2-y_2^2} 
            \right]_{x_1=0}^1\\
            &= \frac{\cosh(y_1)- \cosh(y_2)}{y_1^2-y_2^2}= \frac{\cosh(y_1)}{y_1^2-y_2^2}+ \frac{\cosh(y_2)}{y_2^2-y_1^2},
        \end{aligned}
    \end{equation*}
where,
\begin{equation*}
    \int \cosh(ax+c) \sinh(bx+d)\dif x= \frac{a\sinh(ax+c)\sinh(bx+d)- b\cosh(ax+c)\cosh(bx+d)}{a^2-b^2}+ C,
\end{equation*}
was used.

For $n\geq 2$, by \eqref{DiamondInt1}, 
\begin{equation}\label{DiamondInt2}
    \begin{aligned}
        \frac{e^{p h_{p, B_1^{n+1}}(y/p)}}{(n+1)!}
        &= \int_{\Delta_{n+1}} \cosh(x_1 y_1)\cdots \cosh(x_{n+1}y_{n+1}) \dif x \\
        &= \int_{x_{n+1}=0}^1 \cosh(x_{n+1}y_{n+1}) \int_{(1-x_{n+1})\Delta_n} \cosh(x_1y_1)\cdots \cosh(x_n y_n)\dif x\\
        &= \int_{x_{n+1}=0}^1 \cosh(x_{n+1} y_{n+1}) \frac{e^{ph_{p, B_1^n}(\frac{(1-x_{n+1})y}{p})}}{n!} (1-x_{n+1})^n\dif x_{n+1}.
    \end{aligned}
\end{equation}
because by \eqref{DiamondInt1} and changing variables, 
\begin{equation*}
    \begin{aligned}
     &\int_{(1-x_{n+1})\Delta_n} \cosh(x_1 y_1)\cdots \cosh(x_n y_n)\dif x\\
        &\qq\q= \int_{\Delta_n} \cos((1-x_{n+1})z_1 y_1)\cdots \cosh((1-x_{n+1})z_n y_n) (1-x_{n+1})^n \dif z \\
        &\qq\q= \frac{e^{ph_{p,B_1^n}(\frac{(1-x_{n+1})y}{p})} }{n!} (1-x_{n+1})^n. 
    \end{aligned}
\end{equation*}
By induction, 
\begin{equation}\label{DiamondInt3}
    \begin{aligned}
        &e^{ph_{p, B_1^n}(\frac{(1-x_{n+1})y}{p})}\\
        &= \frac{n!}{p^n}\sum_{j=1}^n   \frac{(\frac{(1-x_{n+1})y_j}{p})^{n-2} (e^{(1-x_{n+1})y_j}+ (-1)^n e^{-(1-x_{n+1})y_j})}{(\frac{1-x_{n+1}}{p})^{2(n-1)} (y_j^2- y_1^2) \cdots (y_j^2-y_{j-1}^2)(y_j^2- y_{j+1}^2)\cdots (y_j^2- y_n^2)} \\
        &=\frac{n!}{(1-x_{n+1})^n} \sum_{j=1}^n  \frac{y_j^{n-2} (e^{(1-x_{n+1})y_j}+ (-1)^n e^{-(1-x_{n+1})y_j})}{(y_j^2- y_1^2) \cdots (y_j^2-y_{j-1}^2)(y_j^2- y_{j+1}^2)\cdots (y_j^2- y_n^2)}.
    \end{aligned}
\end{equation}
Therefore, by \eqref{DiamondInt2} and \eqref{DiamondInt3},
\begin{equation}\label{DiamondInduction1}
    \begin{gathered}
        \frac{e^{p h_{p, B_1^{n+1}}(y/p)}}{(n+1)!} =\sum_{j=1}^n \frac{y_j^{n-2} \int_0^1 \cosh(x_{n+1}y_{n+1}) (e^{(1-x_{n+1})y_j}+(-1)^n e^{-(1-x_{n+1})y_j})\dif x_{n+1}}{(y_j^2- y_1^2)\cdots (y_j^2- y_{j-1}^2)(y_j^2- y_{j+1}^2)\cdots (y_j^2- y_n^2)}. 
    \end{gathered}
\end{equation}
To complete the proof, compute
\begin{equation}\label{DiamondInt4}
    \begin{aligned}
    \int_{0}^1 \cosh(x_{n+1}y_{n+1}) e^{(1-x_{n+1})y_{j}}\dif x_{n+1}
    &= e^{y_j} \int_0^1 \cosh(x_{n+1}y_{n+1}) e^{-x_{n+1}y_j}\dif x_{n+1} \\
    &= \frac12 e^{y_j}\int_0^1 e^{x_{n+1}(y_{n+1}-y_j)}+ e^{-x_{n+1}(y_{n+1}+y_j)}\dif x_{n+1} \\
    &= \frac12 e^{y_j} \left( \frac{e^{y_{n+1}-y_j}-1}{y_{n+1}-y_j}- \frac{e^{-(y_{n+1}+y_j)}-1}{y_{n+1}+ y_j}\right) \\
    &= \frac12 \left( \frac{e^{y_{n+1}}-e^{y_j}}{y_{n+1}-y_j}- \frac{e^{-y_{n+1}}-e^{y_j}}{y_{n+1}+ y_j}\right) \\
    &= \frac{y_j e^{y_j}- y_j \cosh(y_{n+1})- y_{n+1}\sinh(y_{n+1})}{y_j^2-y_{n+1}^2}, 
    \end{aligned}
\end{equation}
and hence, replacing $y_j$ by $-y_j$ in \eqref{DiamondInt4},
\begin{equation}\label{DiamondInt5}
    \int_{0}^1 \cosh(x_{n+1}y_{n+1}) e^{-(1-x_{n+1})y_j}\dif x_{n+1}= \frac{-y_je^{-y_j}+ y_j \cosh(y_{n+1})-y_{n+1}\sinh(y_{n+1})}{y_j^2-y_{n+1}^2}. 
\end{equation}
Therefore, by \eqref{DiamondInt4} and \eqref{DiamondInt5}, 
\begin{equation}\label{DiamondInduction2} 
    \begin{aligned}
   &\int_0^1 \cosh(x_{n+1})\Big( e^{(1-x_{n+1}y_j)}+ (-1)^n e^{-(1-x_{n+1})y_j}\Big)\dif x_{n+1}\\
        &= \frac{y_j e^{y_j}+ (-1)^{n+1} y_j e^{y_j}}{y_{j}^2-y_{n+1}^2} 
        -\frac{(1-(-1)^n) y_j \cosh(y_{n+1})}{y_j^2-y_{n+1}^2} 
        -\frac{(1+(-1)^n)y_{n+1}\sinh(y_{n+1})}{y_j^2-y_{n+1}^2}.
    \end{aligned}
\end{equation}
By \eqref{DiamondInduction1}, \eqref{DiamondInduction2} and Claim \ref{peculiarClaim}, 
{\allowdisplaybreaks
\begin{align*}
    &\frac{1}{(n+1)!} e^{ph_{p, B_1^{n+1}}(y/p)}\\ 
    &\qq=\sum_{j=1}^n \frac{y_j^{n-2} \int_0^1 \cosh(x_{n+1}y_{n+1}) (e^{(1-x_{n+1})y_j}+(-1)^n e^{-(1-x_{n+1})y_j})\dif x_{n+1}}{(y_j^2- y_1^2)\cdots (y_j^2- y_{j-1}^2)(y_j^2- y_{j+1}^2)\cdots (y_j^2- y_n^2)} \\
    &\qq= \sum_{j=1}^n \frac{y_j^{n-1} (e^{y_j}+ (-1)^{n+1} e^{-y_j})}{(y_j^2- y_1^2)\cdots (y_j^2- y_{j-1}^2)(y_j^2- y_{j+1}^2)\cdots (y_j^2- y_n^2)(y_j^2-y_{n+1}^2)} \\
    &\qq\q-\Big(1-(-1)^n\Big)\sum_{j=1}^n \frac{y_j^{n-1}\cosh(y_{n+1})}{(y_j^2- y_1^2)\cdots (y_j^2- y_{j-1}^2)(y_j^2- y_{j+1}^2)\cdots (y_j^2- y_n^2)(y_j^2-y_{n+1}^2)} \\
    &\qq\q-\Big(1+(-1)^n\Big) \sum_{j=1}^n 
    \frac{y_j^{n-2}y_{n+1}\sinh(y_{n+1})}{(y_j^2- y_1^2)\cdots (y_j^2- y_{j-1}^2)(y_j^2- y_{j+1}^2)\cdots (y_j^2- y_n^2)(y_j^2-y_{n+1}^2)} \\
    &\qq= \sum_{j=1}^n \frac{y_j^{n-1} (e^{y_j}+ (-1)^{n+1} e^{-y_j})}{(y_j^2- y_1^2)\cdots (y_j^2- y_{j-1}^2)(y_j^2- y_{j+1}^2)\cdots (y_j^2- y_n^2)(y_j^2-y_{n+1}^2)} \\
    &\qq\q+\Big(1-(-1)^n\Big)\cosh(y_{n+1})\frac{y_{n+1}^{n-1}}{(y_{n+1}^2-y_1^2)\cdots (y_{n+1}^2-y_n^2)} \\
    &\qq\q+\Big(1+(-1)^n\Big)\sinh(y_{n+1})\frac{y_{n+1}^{n-1}}{(y_{n+1}^2-y_1^2)\cdots (y_{n+1}^2-y_n^2)} \\
     &\qq= \sum_{j=1}^n \frac{y_j^{n-1} (e^{y_j}+ (-1)^{n+1} e^{-y_j})}{(y_j^2- y_1^2)\cdots (y_j^2- y_{j-1}^2)(y_j^2- y_{j+1}^2)\cdots (y_j^2- y_n^2)(y_j^2-y_{n+1}^2)} \\
    &\qq\q+\Big(e^y+ (-1)^{n+1} e^{-y}\Big)\frac{y_{n+1}^{n-1}}{(y_{n+1}^2-y_1^2)\cdots (y_{n+1}^2-y_n^2)}, 
\end{align*}
}
as desired. 
\end{proof}
Therefore, in dimension $n=3$, for distinct values of $x,y$ and $z$,
\begin{equation*}
    h_{p, B_1^3}(x,y,z)= \frac1p \log\left[ \frac{6}{p^3}\left( \frac{x\sinh(px)}{(x^2-y^2)(x^2-z^2)}+ \frac{y\sinh(py)}{(y^2-x^2)(y^2-z^2)}+ \frac{z\sinh(pz)}{(z^2-x^2)(z^2-y^2)}\right)\right], 
\end{equation*}
which smoothly extends to $\R^3$. In particular, 
\begin{equation*}
    h_{p,B_1^3}(x,y,z)= \begin{cases}
    \frac1p\log\left[\frac{6}{p^3}\left(\frac{p\cosh(px)}{2(x^2-z^2)}-\frac{x^2+z^2}{(x^2-z^2)^2}\frac{\sinh(px)}{2x}+ \frac{z\sinh(pz)}{(x^2-z^2)^2}\right)\right], & x=y\neq z, \\
     \frac1p\log\left[\frac{6}{p^3}\left(\frac{p\cosh(px)}{2(x^2-y^2)}-\frac{x^2+y^2}{(x^2-y^2)^2}\frac{\sinh(px)}{2x}+\frac{y\sinh(py)}{(x^2-y^2)^2}\right)\right], & x=z\neq y,\\
    \frac1p\log\left[\frac{6}{p^3}\left(\frac{p\cosh(py)}{2(y^2-x^2)}-\frac{y^2+x^2}{(y^2-x^2)^2}\frac{\sinh(py)}{2y}+\frac{x\sinh(px)}{(y^2-x^2)^2}\right)\right], & y=z\neq x, \\
    \frac1p\log\left[ \frac{6}{p^3}\left(\frac{xp\cosh(px)- \sinh(px)+ x^2p^2\sinh(px)}{8x^3}\right)\right], & x=y=z\neq 0, \\
    0, & x=y=z=0. 
    \end{cases}
\end{equation*}

\section{\texorpdfstring{The $L^p$-Santal\'o point}{The Lᵖ-Santaló point}}\label{S3}
In this section, we prove Proposition \ref{santalo_point_thm}. 

First, let us elucidate the similarities
and differences from the case $p=\infty$.
    The Santal\'o point of $K$ is the unique point $x_{\infty, K}\in\mathrm{int}\,K$ for which $b((K-x_{\infty, K})^\circ)=0$ \cite[(2.3)]{santalo}. This is equivalent to $b(h_{K-x_{\infty, K}})=0$ since 
    \begin{equation}\label{bhK}
    b(h_K)= (n+1) b(K^\circ). 
    \end{equation}
    However, since $h_{p,K}$ are not 1-homogeneous for $p<\infty$, it is not in general true that $b(K^{\circ,p})$ vanishes when $b(h_{p,K})$ does.
    To verify \eqref{bhK}, first compute $V(h_K)$. Since $h_K$ is 1-homogeneous and $h_K= \|\cdot\|_{K^{\circ}}$, by Claim \ref{1homogeneous} and \eqref{normVolume},
    \begin{equation}\label{KcircVolume}
        \begin{aligned}
           V(h_K)&= \int_{\R^n} e^{-h_K(y)}\dif y= \int_{\partial B_2^n}\int_0^\infty r^{n-1} e^{-h_K(ru)}\dif r\dif u= (n-1)! \int_{\partial B_2^n}\frac{\dif u}{h_K(u)^{n}}= n! |K^{\circ}|.
        \end{aligned}
    \end{equation}
    Another way to see \eqref{KcircVolume} is to start with \eqref{Mpdef} and \eqref{MpKeq}, i.e.,
    \begin{equation}\label{VhpEq}
    V(h_{p,K})= n! |K^{\circ,p}|,
    \end{equation}
    and take $p\to\infty$.

    For the barycenters, compute in polar coordinates, 
    \begin{equation}\label{bKcircp}
    \begin{aligned}
        b(K^{\circ,p})&= \frac{1}{|K^{\circ,p}|}\int_{\{\|y\|_{K^{\circ,p}}\leq 1\}} y\dif y\\
        &=\frac{1}{|K^{\circ,p}|}\int_{\{(r,u)\in (0,\infty)\times \partial B_2^n: \|ru\|_{K^{\circ,p}}\leq 1\}} ru r^{n-1}\dif r\dif u \\
        &= \frac{1}{|K^{\circ,p}|}\int_{\partial B_2^n} \int_{r=0}^{1/\|u\|_{K^{\circ,p}}} r^{n}u \dif r\dif u\\
        &= \frac{1}{|K^{\circ,p}|}\frac{1}{n+1} \int_{\partial B_2^n} \frac{u}{\|u\|_{K^{\circ,p}}^{n+1}}\dif u\\
        &= \frac{1}{n+1}\frac{1}{|K^{\circ,p}|}\int_{\partial B_2^n}u \left(\frac{1}{(n-1)!}\int_0^\infty r^{n-1} e^{-h_{p,K}(ru)}\dif r \right)^{\frac{n+1}{n}}\dif u.
    \end{aligned}
    \end{equation}
    In addition, by \eqref{VhpEq}, 
    \begin{equation}\label{bhpK}
    \begin{aligned}
       b(h_{p,K})&= \frac{1}{V(h_{p,K})}\int_{\R^n}y e^{-h_{p,K}(y)}\dif y= \frac{1}{|K^{\circ,p}|} \frac{1}{n!}\int_{\partial B_2^n}u \int_0^\infty r^n e^{-h_{p,K}(ru)}\dif r\dif u. 
    \end{aligned}
    \end{equation}
    For $p=\infty$, since $h_{\infty, K}= h_K$ is homogeneous. Claim \ref{1homogeneous} gives, 
    \begin{equation}\label{eq1101}
        \begin{aligned}
            \left(\frac{1}{(n-1)!}\int_0^\infty r^{n-1} e^{-h_{K}(ru)}\dif r \right)^{\frac{n+1}{n}}= \left( \frac{1}{h_K(u)^n}\right)^{\frac{n+1}{n}}= \frac{1}{h_K(u)^{n+1}}= \frac{1}{n!}\int_{0}^\infty r^n e^{-h_{p,K}(ru)}\dif r,
        \end{aligned}
    \end{equation}
    so \eqref{bhK} follows from \eqref{bKcircp}--\eqref{eq1101},
    but  without homogeneity such a relation
    does not hold.
   
    \begin{remark}
    While \eqref{bhK} does not hold for all $p$, one can show a weaker inequality of the form 
    \begin{equation*}
        \left(\frac{1}{(n-1)!} \int_0^\infty r^{n-1} e^{-h_{p,K}(ru)}\dif r\right)^{\frac{n+1}{n}} \leq  (n+1)\frac{\|e^{-h_{p,K}}\|_\infty^{\frac1n}}{(n!)^{\frac1n}}\,\frac{1}{n!}\int_0^\infty r^n e^{-h_{p,K}(ru)}\dif r,
    \end{equation*}
    by using \cite[Lemma 2.2.4]{brazitikos_etal}.
\end{remark}

The proof of Proposition \ref{santalo_point_thm} is based on two key lemmas, 
proved in
\S\ref{Lemma4.2pf} and \S\ref{Lemma4.3pf}.

\begin{lemma}\label{finiteness_bodies}
Let $p\in(0,\infty]$. For a convex body $K\subset \R^n$, $\M_p(K-x)<\infty$ if and only if $x\in \mathrm{int}\,K$.
\end{lemma}

\begin{lemma}\label{strict_convexity}
    Let $p\in(0,\infty]$. For a convex body $K\subset \R^n$, $x\mapsto \M_p(K-x), x\in\mathrm{int}\,K$, is twice differentiable and strictly convex with $\nabla_x\M_p(K-x)= \M_p(K-x) b(h_{p, K-x})$.
\end{lemma}

\begin{proof}[Proof of Proposition \ref{santalo_point_thm}]
    Since, by Lemmas \ref{finiteness_bodies} and \ref{strict_convexity}, $x\mapsto \M_p(K-x)$ is strictly convex in $\mathrm{int}\,K$ and blows up on $\R^n\setminus \mathrm{int}\,K$, it must have a unique minimum at some $x_{p,K}\in \mathrm{int}\,K$. This is a critical point and therefore, by Lemma \ref{strict_convexity}, 
    \begin{equation*}
        0= \nabla_x\M_p(K-x_{p,K})= \M_p(K-x_{p,K}) b(h_{p, K-x_{p,K}}), 
    \end{equation*}
    thus $b(h_{p, K-x_{p,K}})=0$. 
\end{proof}

We call $x_{p,K}$ the $L^p$-Santal\'o point of $K$. For future reference we
record its characterization:
\begin{corollary}\label{santalo_point_corollary}
    Let $p\in(0,\infty]$. For a convex body $K\subset \R^n$, there exists a unique $x_{p,K}\in \R^n$ such that $b(h_{p,K-x_{p,K}})=0$.
\end{corollary}

It is not clear to us how to directly prove Corollary \ref{santalo_point_corollary} if not by Proposition \ref{santalo_point_thm}. In general, for a convex function $\phi:\R^n\to \R\cup\{\infty\}$ with $b(\phi)\in \R^n$, it is not hard to see that there is an $x\in \R^n$ such that under the translation 
\begin{equation*}
    T_x: \R^n \ni y\mapsto y-x \in \R^n,
\end{equation*}
the pull-back of $\phi$
\begin{equation*}
    T_x^*\phi(y)\defeq \phi(y-x)
\end{equation*}
has its barycenter at the origin, $b(T^*_x\phi)=0$. This is because, 
\begin{equation*}
    b(T^*_x\phi)= \int_{\R^n} y e^{-T_x^*\phi(y)}\frac{\dif y}{V(\phi)}= \int_{\R^n}y e^{-\phi(y-x)}\frac{\dif y}{V(\phi)}= \int_{\R^n}(y+x) e^{-\phi(y)}\frac{\dif y}{V(\phi)}= b(\phi)+ x, 
\end{equation*}
so it is enough to choose $x= -b(\phi)$. 
However, functional translation of $h_{p,K}$ does not correspond to the translation of the body. That is, in general, $T^*_x h_{p,K}\neq h_{p, K-x}$. In fact, by Lemma \ref{list} (ii), $h_{p,K-x}(y)= h_{p,K}(y)-\langle y,x\rangle$, and hence
\begin{equation*}
    b(h_{p,K-x})= \int_{\R^n}y e^{-h_{p,K-x}(y)}\frac{\dif y}{V(h_{p,K-x})}= \int_{\R^n} ye^{-h_{p,K}(y)} e^{\langle y,x\rangle}\frac{\dif y}{V(h_{p, K-x})}, 
\end{equation*}
from which is not clear what $x$ should be so that $b(h_{p,K-x})=0$.

\begin{remark}
While we discuss lack of translation-invariance of some quantities, it will
be helpful to note how $\M_p$ transforms 
under the $GL(n,\R)$-action.
For $p>0$, a convex body $K\subset \R^n$, and $A\in GL(n,\R)$, by Lemma \ref{list} (iii),
\begin{equation*}
    \|x\|_{(AK)^{\circ,p}}\defeq \left( \int_0^\infty r^{n-1}e^{-h_{p,AK}(rx)}\dif r\right)^{-\frac{1}{n}}= \left( \int_0^\infty r^{n-1}e^{-h_{p,K}(rA^T x)}\dif r\right)^{-\frac{1}{n}}= \|A^T x\|_{\LpK}, 
\end{equation*}
hence
\begin{equation}
\label{LpK_affine_transform}
(AK)^{\circ,p}= (A^{-1})^T K^{\circ,p}.
\end{equation}
In sum:

\begin{lemma}\label{Mp_aff_inv}
    Let $p\in(0,\infty]$. For a compact body $K\subset \R^n$ and $A\in GL(n,\R)$, $\M_p(AK)= \M_p(K)$.
\end{lemma}

This $GL(n,\R)$-invariance will be useful in several places, e.g., in the
proof of Claim \ref{MpFiniteClaim} below and in proving Theorem \ref{santalo_symmetric} when we deal with Steiner symmetrization. 
\end{remark}

\subsection{\texorpdfstring{Finiteness of $\M_p$}{Finiteness of Mₚ}}
\lb{Lemma4.2pf}
Lemma \ref{finiteness_bodies} follows from the following two claims.
\begin{claim}
\lb{MpFiniteClaim}
    Let $p\in(0,\infty]$. For a convex body $K\subset \R^n$ with $0\in\mathrm{int}\,K$, and $r>0$ such that $[-r,r]^n\subset K$, 
    \begin{equation*}
        \M_p(K)\leq \frac{|K|^{1+1/p}}{(2r)^{n+n/p}}\M_p([-1,1]^n).
    \end{equation*}
    In particular, $\M_p(K)<\infty$. 
\end{claim}
\begin{proof}
    Since $0\in\mathrm{int}\,K$, there is $r>0$ such that $[-r,r]^n\subset \mathrm{int}\,K$. By \eqref{hpKrEq}, 
    \begin{equation*}
    \begin{aligned}
        \M_p(K)&\defeq |K|\int_{\R^n}e^{-h_{p,K}(y)}\dif y\\
        &\leq |K|\int_{\R^n} e^{-h_{p, [-r,r]^n}(y)} \frac{|K|^{1/p}}{(2r)^{n/p}} \dif y
        = \frac{|K|^{1+1/p}}{(2r)^{n+ n/p}} (2r)^n\int_{\R^n}e^{-h_{p,[-r,r]^n}(y)}\dif y\\
        &= \frac{|K|^{1+1/p}}{(2r)^{n+n/p}} \M_p([-r,r]^n)
        = \frac{|K|^{1+1/p}}{(2r)^{n+n/p}} \M_p([-1,1]^n), 
    \end{aligned}
    \end{equation*}
    where we used Lemma \ref{Mp_aff_inv}.
    By Lemma \ref{mKmpKineq}, since $b([-1,1]^n)=0$, 
    $$
    \M_p([-1,1]^n)\leq \left( \frac{(1+p)^{1+\frac1p}}{p}\right)^n \M([-1,1]^n)= \left( \frac{(1+p)^{1+\frac1p}}{p}\right)^n 4^n,
    $$ 
   concluding the proof.
\end{proof}

\begin{claim}
    Let $p\in(0,\infty]$. For a convex body $K\subset \R^n$ with $0\notin \mathrm{int}\,K$, $\M_p(K)=\infty$. 
\end{claim}
\begin{proof}
    By convexity of $K$, since $0\notin \mathrm{int}\,K$, there is a hyperplane through the origin 
    \begin{equation*}
        u^\perp\defeq \{x\in \R^n: \langle x,u\rangle=0\}, 
    \end{equation*}
    such that $K\subset \{x\in \R^n: \langle x,u\rangle\geq 0\}$. In particular, $\langle x,-u\rangle\leq 0$ for all $x\in K$, and hence
    \begin{equation*}
        c\defeq \int_K e^{p\langle x, -u\rangle}\frac{\dif x}{|K|}< 1.
    \end{equation*}
    If it was exactly equal to 1, then $\langle x,u\rangle=0$ for all $x\in K$, that is $K\subset u^\perp$, which is a contradiction because $K$ has non-empty interior.
    Let $U\subset \partial B_2^n$ be an open neighborhood of $-u$ such that
    \begin{equation*}
        \int_K  e^{p\langle x, v\rangle}\frac{\dif x}{|K|}\leq \frac{1+c}{2}<1, \quad \text{ for all } v\in U.
    \end{equation*}
    For $r\geq 1$ and $v\in U, x\in K$, since $p\langle x,v\rangle<0$, $rp\langle x,v\rangle\leq p\langle x,v\rangle$, thus 
    \begin{equation}\label{4.1eq1}
        \int_K e^{rp\langle x,v\rangle}\frac{\dif x}{|K|}\leq \int_K e^{p\langle x,v\rangle}\frac{\dif x}{|K|}\leq \frac{1+c}{2}<1, \quad v\in U, r\geq 1.
    \end{equation}
    In polar coordinates, by (\ref{4.1eq1}),
    \begin{equation*}
        \begin{aligned}
            \M_p(K)&= |K|\int_{\R^n} e^{-h_{p,K}(y)}\dif y= |K| \int_{\R^n} \frac{\dif y}{\left( \int_K e^{p\langle x,y\rangle} \frac{\dif x}{|K|}\right)^{1/p}}\\
            &= |K|\int_{\partial B_2^n}\int_0^\infty \frac{r^{n-1}\dif r\dif v}{\left(\int_K e^{rp\langle x,v\rangle}\frac{\dif x}{|K|}\right)^{1/p}}\geq 
             |K|\int_{U}\int_1^\infty \frac{r^{n-1}\dif r\dif v}{\left(\int_K e^{rp\langle x,v\rangle}\frac{\dif x}{|K|}\right)^{1/p}}\\
             &\geq |K| \left(\frac{1+c}{2}\right)^{-1/p}\int_U \int_{1}^\infty r^{n-1}\dif r= \infty.
        \end{aligned}
    \end{equation*}
\end{proof}

\subsection{\texorpdfstring{Smoothness and convexity of $\M_p$}{Smoothness and convexity of Mₚ}}
\lb{Lemma4.3pf}
\begin{proof}[Proof of Lemma \ref{strict_convexity}.]   
    Denote by $e_1, \ldots, e_n$ the standard basis of $\R^n$.
    For $x\in\mathrm{int}\,K$ there is $r>0$ such that $x+ 2r B_2^n\subset \mathrm{int}\,K$. 
    Using Lemma \ref{list} (ii), and 
    letting $0<\e<r$, 
    \begin{equation}\label{finite_eq}
        \begin{aligned}
            \left| \frac{\M_p(K-x-\e e_i)- \M_p(K-x)}{\e}\right|&= \left| \frac1\e \int_{\R^n} e^{-h_{p,K-x-\e e_i}(y)}- e^{-h_{p,K-x}(y)} \dif y\right|\\
            &= \left| \frac1\e \int_{\R^n} e^{-h_{p,K}(y)} e^{\langle x+\e e_i, y\rangle}- e^{-h_{p,K}(y)} e^{\langle x, y\rangle}\dif y\right|\\
            &\leq \int_{\R^n} \left|\frac{e^{\langle x+\e e_i, y\rangle}- e^{\langle x,y\rangle}}{\e}\right| e^{-h_{p,K}(y)}\dif y\\
            &= \int_{\R^n} \left|\frac{e^{\e y_i}- 1}{\e}\right| e^{\langle x,y\rangle-h_{p,K}(y)}\dif y.
        \end{aligned}
    \end{equation}
    Note
    \begin{equation*}
        \left|\frac{e^{\e y_i}-1}{\e}\right|\leq \sum_{m=1}^\infty \frac{\e^{m-1}|y_i|^m}{m!}\leq \sum_{m=1}^\infty \frac{r^{m-1}|y_i|^m}{m!}= \frac1r \sum_{m=1}^\infty \frac{r^m |y_i|^m}{m!}\leq \frac1r e^{r|y_i|}= \frac{1}{r} e^{r\mathrm{sgn}(y_i)y_i},
    \end{equation*}
    so, by \eqref{finite_eq},
    \begin{equation}\label{finite_eq2}
        \begin{aligned}
        \left| \frac{\M_p(K-x-\e e_i)- \M_p(K-x)}{\e}\right|
        &\leq\int_{\R^n} \frac1r e^{r\mathrm{sgn}(y_i) y_i} e^{-h_{p,K-x}(y)}\dif y\\
            &= \frac1r \int_{\R^{n-1}} \left(\int_0^\infty  e^{ry_i} +\int_{-\infty}^0 e^{-r y_i} \right) e^{-h_{p,K-z}(y)}\dif y \\
            &\leq \frac1r \int_{\R^n} e^{ry_i} e^{-h_{p,K-x}(y)}\dif y+ \frac1r \int_{\R^n} e^{-ry_i} e^{-h_{p,K-x}(y)}\dif y \\
            &= \M_p(K- x-re_i)+ \M_p(K-x+ re_i).
        \end{aligned}
    \end{equation}
    Since $x+ re_i, x- re_i\in \mathrm{int}\,K$, by Lemma \ref{finiteness_bodies}, the right-hand side of (\ref{finite_eq2}) is finite, and hence by dominated convergence we may differentiate under the integral, 
    \begin{equation*}
    \begin{aligned}
        \nabla_x\M_p(K-x)&= \nabla_x \left( |K|\int_{\R^n} e^{-h_{p,K-x}(y)}\dif y\right)= \nabla_x \left( |K|\int_{\R^n} e^{-h_{p,K}(y)} e^{\langle y,x\rangle}\dif y\right) \\
        &= |K|\int_{\R^n}\nabla_x\left(e^{-h_{p,K}(y)} e^{\langle y,x\rangle}\right)\dif y= |K|\int_{\R^n}y e^{-h_{p,K}(y)} e^{\langle y,x\rangle}\dif y\\
        &= |K|\int_{\R^n}y e^{-h_{p,K-x}(y)}\dif y\\
        &= |K|\int_{\R^n}e^{-h_{p,K-x}(x)}\dif x\int_{\R^n} ye^{-h_{p,K-x}(y)}\frac{\dif y}{\int_{\R^n}e^{-h_{p,K-x}(x)}\dif x}\\
        &= \M_p(K-x) b(h_{p, K-x}).
    \end{aligned}
    \end{equation*}
    Similarly, one can show the second-order derivatives exist and are continuous.  Differentiating under the integral sign,
    \begin{equation*}
        \begin{aligned}
            \frac{\partial^2}{\partial x_i\partial x_j}\M_p(K-x)= |K|\int_{\R^n} y_i y_j e^{-h_{p,{K-x}}(y)}\dif y. 
        \end{aligned}
    \end{equation*}
    Therefore, for $v\in \R^n$, 
    \begin{equation*}
        v^T \nabla_{x}^2\M_p(K-x) v= |K|\sum_{i,j=1}^n \int_{\R^n} v_i v_j y_i y_j e^{-h_{p,K-x}(y)}\dif y= |K|\int_{\R^n} \langle v,y\rangle^2 e^{-h_{p,K-x}(y)}\dif y\geq 0,
    \end{equation*}
    with equality if and only if $\langle y,v\rangle=0$ for almost all $y$, or equivalently $v=0$, proving strict convexity.
\end{proof}

\section{\texorpdfstring{The upper bound on $\M_p$}{The upper bound on Mₚ}}\label{S4}

This section is dedicated to proving the $L^p$-Santal\'o theorem, Theorem \ref{santalo_symmetric}. 
As expected, we use symmetrization. However, there are a number of
intricate details that need to be carefully dealt with, since $L^p$-polarity
is a highly non-local operation compared to classical polarity. 
On the surface of it though,
as in the case $p=\infty$, the key estimate we need to prove is the monotonicity
of volume under Steiner symmetrization:
\begin{proposition}\label{steiner_ppolar}
    Let $p\in(0,\infty]$. For a symmetric convex body $K\subset \R^n$ and $u\in \partial B_2^n$, 
    let $\sigma_uK$ be the Steiner symmetral of $K$ (Definition \ref{steiner_def}).
    Then, 
    $|(\sigma_uK)^{\circ,p}|\geq |\LpK|$.
\end{proposition}

\subsection{Outline of the proof of Proposition \ref{steiner_ppolar}}

Proposition \ref{steiner_ppolar} is proved in \S\ref{ppolarSantaloProof}.
For $n=1$, $\sigma_u K=K$ if $K=-K$. 
Thus, take $n>1$ for the rest of the section. 
We follow a classical proof for the case $p=\infty$ \cite[Proposition 9.2]{gruber} \cite[Proposition 1.1.15]{artstein-giannopoulos-milman}
and make the appropriate modifications to $p\in(0,\infty)$. 
This involves comparing the volume of the `slices' of the polar body perpendicular to the vector used for Steiner symmetrization. For a convex body $K\subset \R^n$, and $x_n\in\R$, denote by
\begin{equation}\label{sliceDef}
    K(x_n)\defeq \{\xi\in \R^{n-1}: (\xi,x_n)\in K\},
\end{equation}
the slice of $K$ at height $x_n$. By Tonelli's theorem \cite[\S 2.37]{folland}, the volume of a convex body may be expressed as an integral of the volume of its slices,
\begin{equation}\label{sliceVolumeIntegral}
    |K|= \int_{\{(\xi,x_n)\in \R^{n-1}\times \R: \,\, \xi\in K(x_n)\}} \dif \xi\dif x_n= \int_{-\infty}^\infty |K(x_n)|\dif x_n. 
\end{equation}
In view of \eqref{sliceVolumeIntegral}, Proposition \ref{steiner_ppolar} follows from the next lemma. Denote by $e_1, \ldots, e_n$ the standard basis of $\R^n$.
\begin{lemma}\label{sliceVolume}
    Let $p\in(0,\infty]$. For a symmetric convex body $K\subset \R^n$, $|(\sigma_{e_n}K)^{\circ,p}(x_n)|\geq |K^{\circ,p}(x_n)|$ for all $x_n\in \R$.
\end{lemma}

Lemma \ref{sliceVolume}, in turn, follows from the Brunn--Minkowski inequality and the following monotonicity property of the average of antipodal slices under Steiner symmetrization.
\begin{lemma}\label{steiner_cor}
    Let $p\in(0,\infty]$. For a convex body $K\subset \R^n$,
    \begin{equation}\label{averageEq}
        \frac{K^{\circ,p}(x_n)+ K^{\circ,p}(-x_n)}{2}\subset \frac{(\sigma_{e_{n}}K)^{\circ,p}(x_n)+ (\sigma_{e_{n}}K)^{\circ,p}(-x_n)}{2}= (\sigma_{e_{n}}K)^{\circ,p}(x_n). 
    \end{equation}
\end{lemma}
The equality on the right-hand side holds because $\sigma_{e_n} K$, and hence $(\sigma_{e_n}K)^{\circ,p}$ (Lemma \ref{KcircpSymPlane}), are by construction symmetric with respect to $e_n^\perp$. 
Nonetheless, note that no symmetry on $K$ is assumed for Lemma \ref{steiner_cor}, in contrast to Lemma \ref{sliceVolume}.
Applying the Brunn--Minkwoski inequality on Lemma \ref{steiner_cor} gives
\begin{equation*}
    |(\sigma_{e_n}K)^{\circ,p}(x_n)|^{\frac{1}{n-1}}\geq \frac12 |K^{\circ,p}(x_n)|^{\frac{1}{n-1}}+ \frac12 |K^{\circ,p}(-x_n)|^{\frac{1}{n-1}}. 
\end{equation*}
Without any symmetry assumption on $K$, $|K^{\circ,p}(x_n)|$ and $|K^{\circ,p}(-x_n)|$ may be unrelated. 
For symmetric convex bodies, $K^{\circ,p}(-x_n)= -K^{\circ,p}(x_n)$ (Claim \ref{symSlices}) and hence $|K^{\circ,p}(-x_n)|= |K^{\circ,p}(x_n)|$, justifying the symmetry assumption in Lemma \ref{sliceVolume}. 

\begin{figure}[H]
    \centering
    \includegraphics[scale=.25]{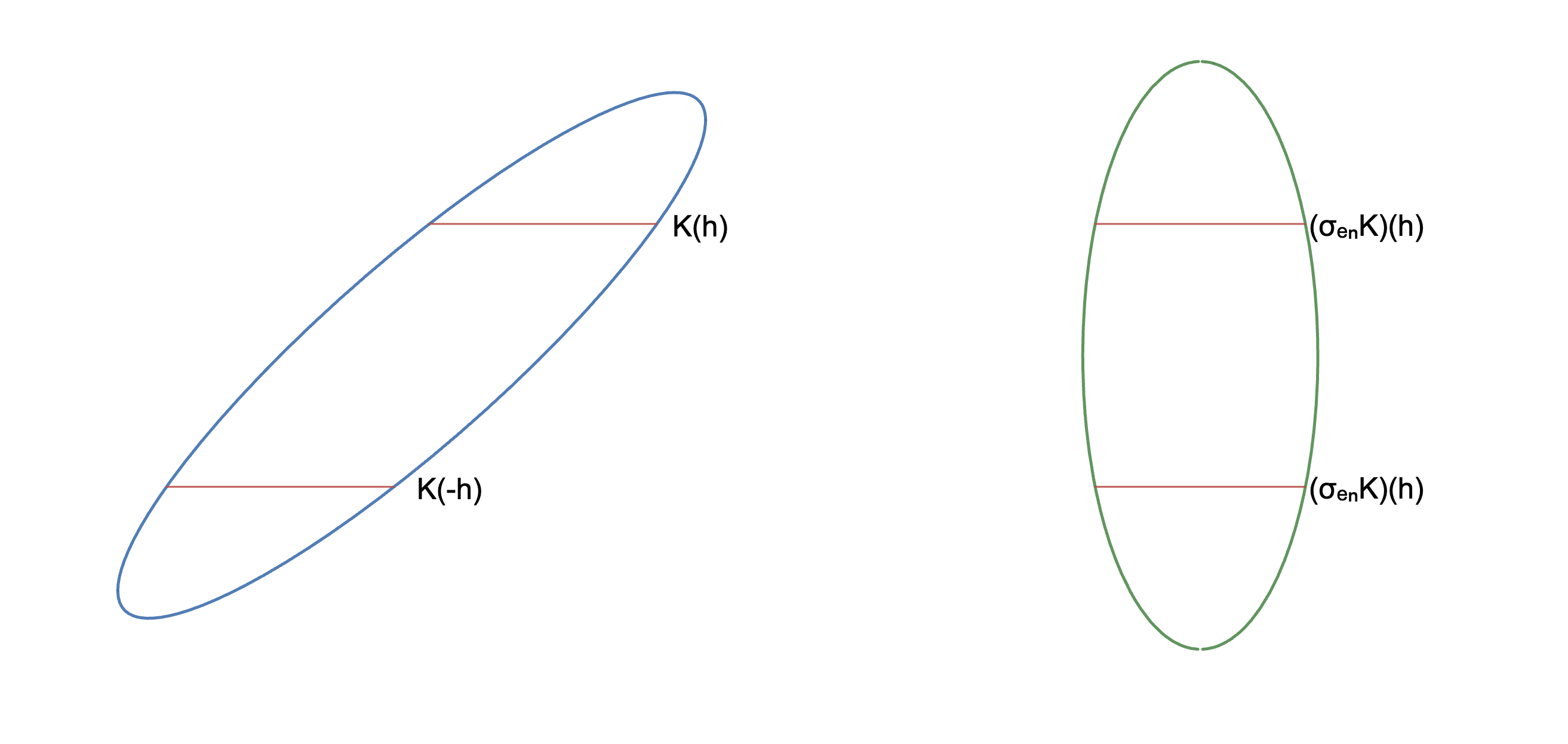}
    \caption{\small Comparing the slices}
\end{figure}

In order to obtain the inclusion of Lemma \ref{steiner_cor}, we first obtain an inequality relating the norms before and after symmetrization:
\begin{equation}\label{steiner_symm_lem1}
        \left\|\left(\frac{\xi+\xi'}{2}, x_n\right)\right\|_{(\sigma_{e_{n}}K)^{\circ,p}}\leq \frac{\|(\xi, x_n)\|_{\LpK}+ \|(\xi', -x_n)\|_{\LpK}}{2}. 
\end{equation}
For $p=\infty$, by \eqref{PolarSupportNorm}, \eqref{steiner_symm_lem1} reads
\begin{equation}\label{hKSteiner}
    h_{\sigma_{e_n}K}\left( \frac{\xi+\xi'}{2},x_n\right)\leq \frac{h_K(\xi,x_n)+ h_K(\xi',-x_n)}{2},
\end{equation}
which is classical and simple to prove: any element of $\sigma_{e_n}K$ is of the form $(z, \frac{t-s}{2})$, for $(z,t), (z,s)\in K$, so
\begin{equation*}
    \begin{aligned}
        \left\langle \Big(\frac{\xi+\xi'}{2}, x_n\Big), \Big( z,\frac{t-s}{2}\Big)\right\rangle &= \left\langle \frac{\xi+\xi'}{2}, z\right\rangle+ x_n\frac{t-s}{2} \\
        &= \frac{\langle \xi,z\rangle+ x_n t}{2}+ \frac{\langle \xi',z\rangle- x_n s}{2} \\
        &= \frac{\langle (\xi, x_n), (z,t)\rangle}{2}+ \frac{\langle (\xi',-x_n), (z,-s)\rangle}{2} \\
        &\leq \frac{h_K(\xi,x_n)+ h_K(\xi',- x_n)}{2},
    \end{aligned}
\end{equation*}
and \eqref{hKSteiner} follows.
One of our key estimates in this section is a
3-parameter ($p,s,t$) family generalization of \eqref{hKSteiner}:
\begin{lemma}\label{steiner_claim2}
    Let $p\in(0,\infty]$, and $K\subset \R^n$ a convex body. For $\xi,\xi'\in \R^{n-1}, x_n\in \R$ and $r,t,s>0$ with $\frac{2}{r}= \frac{1}{t}+ \frac{1}{s}$,
    \begin{equation*}
        h_{p,\sigma_{e_{n}}K}\left(r\frac{\xi+\xi'}{2}, r x_n\right)\leq \frac{s}{t+s} h_{p,K}(t\xi, t x_n)+ \frac{t}{t+s} h_{p, K}(s\xi', -s x_n).
    \end{equation*}
\end{lemma}
For $p=\infty$, Lemma \ref{steiner_claim2} is equivalent to \eqref{steiner_symm_lem1}.
Lacking homogeneity, for $p\in (0,\infty)$ this is no longer the case. 
Notwithstanding, Lemma \ref{steiner_claim2} is exactly the condition necessary to apply Ball's Brunn--Minkowski inequality for harmonic means (Theorem \ref{KBall_ineq}, proven in Appendix \ref{appendixA}) from which we deduce \eqref{steiner_symm_lem1}.
The next step in the proof of Proposition \ref{steiner_ppolar} is to use \eqref{steiner_symm_lem1} to obtain Lemma \ref{steiner_cor}. 
Finally, Lemma \ref{steiner_cor} and a symmetry property for antipodal slices of symmetric bodies (Corollary \ref{KcircpSliceSym}) give Lemma \ref{sliceVolume} from which Proposition \ref{steiner_ppolar} follows by \eqref{sliceVolumeIntegral}. 

The proof of Proposition \ref{steiner_ppolar} is organized
as follows. Sections \ref{MpSteiner} and \ref{MpHausdorff} are preparatory. In \S \ref{MpSteiner} we recall a few basics of Steiner symmetrization. In \S\ref{MpHausdorff}, Lemma \ref{Mp_Hausdorff} establishes the continuity of $\M_p$ in the Hausdorff topology (Definition \ref{HausdorffDef}). \S \ref{SliceSymSection} establishes several symmetries between antipodal slices for symmetric convex bodies.
Section \ref{SantaloTechnicalSection} is dedicated to proving Lemma \ref{steiner_claim2}, and Section \ref{KcircpSteiner}  to proving Lemmas \ref{steiner_cor} and \ref{sliceVolume}. In \S \ref{ppolarSantaloProof}, we complete the proofs of Proposition \ref{steiner_ppolar} and Theorem \ref{santalo_symmetric}.

\subsection{Steiner symmetrization}\label{MpSteiner}

For a vector $u\in\partial B_2^n$ denote by 
\begin{equation*}
    u^\perp\defeq \{x\in \R^n: \langle x,u\rangle=0\}, 
\end{equation*}
the hyperplane through the origin that is normal to $u$. Let, also,
\begin{equation*}
    \pi_{u^\perp}: \R^n \ni x\mapsto x- \langle x,u\rangle u \in u^\perp
\end{equation*}
the projection onto $u^\perp$.
Given $u\in \partial B_2^n$, one may  
foliate any convex body $K$ by a family of straight line segments parametrized by a hyperplane $u^\perp$. The Steiner symmetral $\sigma_uK$
is the unique such foliation for which the line segments have their midpoints in $u^\perp$
\cite[pp. 286--287]{steiner}
(see also \cite[\S 9]{gruber}
\cite[Definition 1.1.13]{artstein-giannopoulos-milman}):
\begin{definition}\label{steiner_def}
    For $K\subset \R^n$ a convex body and $u\in \partial B_2^n$, the Steiner symmetral in the $u$ direction is given by
    \begin{equation*}
        \sigma_u(K)\defeq \{x+ tu: x\in \pi_{u^\perp}(K)
        \text{ and } |t|\leq \frac12  |K\cap (x+\R u)|\}. 
    \end{equation*}
\end{definition}

Steiner symmetrization produces a convex body that is symmetric with respect to $u^\perp$. 
\begin{definition}
A convex body $K\subset \R^n$ is symmetric with respect to a hyperplane $u^\perp$ if for all $x\in K$
\begin{equation*}
    x- 2\langle x,u\rangle u \in K.
\end{equation*}
\end{definition}
Equivalently, $K$ remains invariant under reflection with respect to $u^\perp$.
Steiner symmetrization also preserves volume and convexity \cite[Proposition 9.1]{gruber}: 
\begin{lemma}\label{SteinerProperties}
    For a convex body $K\subset \R^n$ and $u\in\partial B_2^n$, $\sigma_u(K)$ is a convex body, symmetric with respect to $u^\perp$, with $|\sigma_u(K)|= |K|$.
\end{lemma}

Orthogonal transformations preserve volume and, by \eqref{LpK_affine_transform}, commute with $L^p$-polarity. The following lemma then justifies working with $u=e_n$ throughout.   
\begin{lemma}\label{steiner_orthogonal}
For a convex body $K\subset \R^n$, $u\in \partial B_2^n$, and $A\in O(n)$, 
\begin{equation*}
    \sigma_{u}(K)= A^{-1}\sigma_{Au}(AK). 
\end{equation*}
In particular, $|\sigma_u(K)|= |\sigma_{Au}(AK)|$.
\end{lemma}
\begin{proof}
Since $A\in O(n)$ is invertible, it is enough to show $A^{-1}\sigma_{Au}(AK)\subset \sigma_u(K)$. Let $x+ t Au\in \sigma_{Au}(AK)$ with 
$$
x\in \pi_{(Au)^\perp}(AK), \quad
\hbox{\ and \ } |t|\leq \frac12 |(AK)\cap (x+ \R Au)|.
$$

First, 
\begin{equation}\label{piAuEq}
    \pi_{(Au)^\perp}(AK)= A\pi_{u^\perp}(K).
\end{equation}
Indeed, for $z\in\R^n$,
$$
    \pi_{(Au)^\perp}(Az)= Az-\langle Az, Au\rangle Au= Az- \langle z, u\rangle Au= A(z- \langle z,u\rangle u)= A\pi_{u^\perp}(z),
$$
because, since $A\in O(n)$, $\langle Az, au\rangle= \langle z, A^TAu\rangle= \langle z,u\rangle$.

Second, 
\beq
\lb{AuperpEq2}
(AK)\cap (x+ \R Au)= A\left( K\cap (A^{-1}x+ \R u)\right).
\eeq
That is because, $y\in (AK)\cap (x+ \R Au)$, if and only if $y\in AK$ and $y= x+ s Au, x\in K, s\in \R$. Equivalently, $A^{-1}y\in K$ and $A^{-1}y= A^{-1}x+ su\in A^{-1}x+ \R u$, i.e., $A^{-1}y\in K\cap (A^{-1} x+ \R u)$. 

Using \eqref{AuperpEq2} and as $A\in O(n)$ preserves volume, $|K\cap (A^{-1}x+\R u)|= |(AK)\cap (x+\R Au)|$. 
Thus $A^{-1}(x+ t Au)= A^{-1}x+ tu$, is such that 
$A^{-1}x\in A^{-1} \pi_{(Au)^\perp}(AK) = A^{-1}(A \pi_{u^\perp}(K))= \pi_{u^\perp}(K)$ (using \eqref{piAuEq}), and $|t|\leq \frac12 |(AK)\cap (x+ \R Au)|= \frac12|K\cap (A^{-1}x+ \R u)|$, that is $A^{-1}(x+ t Au)= A^{-1}x+ tu\in \sigma_u(K)$.
\end{proof}

Recall the definition of the Hausdorff metric.
\begin{definition}\label{HausdorffDef}
    For $K, L\subset \R^n$ two compact bodies, let
    \begin{equation*}
        d_H(K,L)\defeq \inf\{\e>0: K\subset L+ \e B_2^n \text{ and } L\subset K+\e B_2^n\}, 
    \end{equation*}
    be the Hausdorff distance between $K$ and $L$. 
\end{definition}

Repeated Steiner symmetrizations Hausdorff
converge  to a 2-ball
\cite{gross}  \cite[Theorem 9.1]{gruber}.

\begin{lemma}\label{steiner_limit}
    For a convex body $K\subset \R^n$, there is $\rho>0$ and a sequence of vectors $u_j\in\partial B_2^n$ such that if $K_j\defeq \sigma_{u_j}(K_{j-1})$, where $K_0\defeq K$, then $K_j\to \rho B_2^n$ in the Hausdorff metric.
\end{lemma}

\subsection{\texorpdfstring{Hausdorff continuity of $\M_p$
}{Hausdorff continuity of Mₚ}}\label{MpHausdorff}

The aim of this subsection is to verify that $\M_p$ is continuous under Hausdorff convergence (Lemma \ref{Mp_Hausdorff}).

By Lemma \ref{steiner_limit}, iterated applications of Steiner symmetrization $d_H$-converge to a 2-ball. Therefore, in order to obtain Theorem \ref{santalo_symmetric}, it is necessary to show that $\M_p$ is $d_H$-continuous. 
\begin{lemma}\label{Mp_Hausdorff}
    Let $p\in(0,\infty]$, and $\{K_j\}_{j\geq 1}\subset \R^n$ a sequence of convex bodies $d_H$-converging to a convex body 
    $K\subset \R^n$ with $|K^{\circ,p}|<\infty$.        
    Then, $\M_p(K_j)\to \M_p(K)$. 
\end{lemma}

Lemma \ref{Mp_Hausdorff} follows from the next two claims. First, the volume of convex bodies is continuous under the Hausdorff metric. Note this is not true without the convexity assumption, e.g. for space-filling curves. Denote by
\begin{equation*}
    \bm{1}_K(x)\defeq \begin{cases}
    1, & x\in K, \\
    0, & x\notin K, 
    \end{cases}
\end{equation*}
the indicator function of $K$.
\begin{claim}\label{HausdorffClaim1}
Let $\{K_j\}_{j\geq 1}\subset \R^n$ be a sequence of convex bodies $d_H$-converging to $K\subset \R^n$. Then, $|K_j|\to |K|$. 
\end{claim}
\begin{proof}
Since $d_H(K_j, K)\to 0$, there are $\e_j>0$ such that $K_j\subset K+ \e_j B_2^n$ and $K+\e_j B_2^n\subset K_j$, with $\e_j\to 0$. In particular, $\{\e_j\}_{j\geq 1}$ is bounded. For simplicity, take $\e_j\leq 1$. In particular, $K_j\subset K+\e_j B_2^n\subset K+B_2^n$, thus $\bm{1}_{K_j}\leq \bm{1}_{K+B_2^n}$ for all $j$. 
This allows for the use of dominated convergence. 
It is therefore enough to show
\begin{equation}\label{KjPC}
    \lim_{j\to\infty} \bm{1}_{K_j}(x)= \bm{1}_K(x), \quad x\in (\mathrm{int}\,K)\cup (\R^n\setminus K). 
\end{equation}
Then, by dominated convergence, 
\begin{equation*}
    \lim_{j\to\infty} |K_j|= \lim_{j\to\infty} \int_{\R^n}\bm{1}_{K_j}= \int_{\R^n}\lim_{j\to\infty} \bm{1}_{K_j}= \int_{\R^n}\bm{1}_K= |K|.
\end{equation*}
For \eqref{KjPC}, let $x\in \mathrm{int}\,K$. There is $\e>0$ such that $x+ \e B_2^n\subset K$. Since $\e_j\to 0$, there is $j_0\geq 1$ such that $\e_j<\e$ for all $j\geq j_0$. Therefore, $x+\e B_2^n\subset K\subset K_j+\e_j B_2^n\subset K_j+ \e B_2^n$. By the cancellation law for the Minkowski sum of convex bodies \cite[Theorem 6.1 (i)]{gruber}, $\{x\}\subset K_j$, i.e., $x\in K_j$. Therefore, $\bm{1}_{K_j}(x)=1= \bm{1}_K(x)$ for all $j\geq j_0$.

For $x\in \R^n\setminus K$, since $K$ is closed, $\R^n\setminus K$ is open, thus there is $\e>0$ such that $x+ 2\e B_2^n\subset \R^n\setminus K$, i.e., $(x+2\e B_2^n)\cap K=\emptyset$. Let $j_0\geq 1$ with $\e_j<\e$ for all $j\geq j_0$. Then, $K_j\subset K+ \e B_2^n$ and hence
\begin{equation*}
    (x+ \e B_2^n)\cap K_j\subset (x+\e B_2^n)\cap (K+\e B_2^n) =\emptyset,
\end{equation*}
because for $y\in (x+\e B_2^n)\cap(K+\e B_2^n)$, $y= x+ \e u= z+ \e v,$ for $u,v\in B_2^n$ and $z\in K$. That is, $z= x+ \e(u-v)\in x+2\e B_2^n$, thus $z\in K\cap (x+2\e B_2^n)=\emptyset$, a contradiction.  
Therefore, $x\notin K_j$ for all $j\geq j_0$, i.e., $\bm{1}_{K_j}(x)=0= \bm{1}_K(x)$ for all $j\geq j_0$, proving \eqref{KjPC}. 
\end{proof}

Second, the volume of the $L^p$-polars is also continuous under Hausdorff convergence given that the limit is a convex body with finite $\M_p$ volume. 
\begin{claim}\label{HausdorffClaim2}
    Let $p\in(0,\infty]$, and $\{K_j\}_{j\geq 1}\subset \R^n$ a sequence of convex bodies $d_H$-converging to a convex body $K$ with $|K^{\circ,p}|<\infty$. Then, $|K_j^{\circ,p}|\to |K^{\circ,p}|$. 
\end{claim}
\begin{proof}
Since $d_H(K_j, K)\to 0$, 
there are $\e_j>0$ such that $K_j\subset K+\e_j B_2^n$ and $K\subset K_j+ \e_j B_2^n$ with $\e_j\to 0$. In particular, $\{\e_j\}_{j\geq 1}$ is bounded. For simplicity, take $\e_j\leq 1/2$. In particular, 
\begin{equation*}
    K_j\subset K+\e_j B_2^n\subset K+ B_2^n, 
\end{equation*}
so $K_j$ are uniformly bounded. Let $M>0$ such that $|x|\leq M$ for all $x\in K_j$ and all $j$.
For $y\in \R^n$,
    \begin{equation}\label{hausdorff_eq1}
    \begin{aligned}
        \left||K_j|e^{ph_{p,K_j}(y)}- |K|e^{ph_{p,K}(y)}\right|&=\left| \int_{K_j} e^{p\langle x,y\rangle}\dif x- \int_K e^{p\langle x,y\rangle}\dif x\right|\\
        &\leq \int_{(K_j\setminus K)\cup (K\setminus K_j)} e^{p\langle x,y\rangle}\dif x\\
        &\leq |(K_j\setminus K)\cup (K\setminus K_j)| e^{pM|y|}.
    \end{aligned}
    \end{equation}
    Note that 
    \begin{equation*}
        \bm{1}_{(K_j\setminus K)\cup (K\setminus K)}(y)= |\bm{1}_{K_j}(y)-\bm{1}_{K}(y)|
    \end{equation*}
    which converges to 0 almost everywhere by \eqref{KjPC}. By dominated convergence, $|(K_j\setminus K)\cup (K\setminus K_j)|\to 0$.
    Taking $j\to \infty$ in \eqref{hausdorff_eq1}, $|K_j| e^{ph_{p,K_j}(y)}\to |K| e^{ph_{p,K}(y)}$. By Claim \ref{HausdorffClaim1}, $|K_j|\to |K|$, thus
    \begin{equation}\label{hausdorffEq2}
        \lim_{j\to\infty} h_{p,K_j}(y)= h_{p,K}(y), \quad y\in \R^n, 
    \end{equation}
    establishing the pointwise convergence. 
    
    The aim is to use dominated convergence on $e^{-h_{p,K_j}}$, for which a uniform (independent of $j$) and integrable upper bound is necessary. By assumption $|K^{\circ,p}|<\infty$, or equivalently, by Lemma \ref{finiteness_bodies}, $0\in\mathrm{int}\,K$. That is, there is $r>0$ such that $[-2r,2r]^n\subset K$. Therefore, for large enough $j_0>0$, $[-r,r]^n\subset K_j$ for all $j\geq j_0$. In addition, by Claim \ref{HausdorffClaim1}, $|K_j|\to |K|>0$, thus there is $M'>0$ with $|K_j|\leq M'$ for all $j$. As a result, 
    \begin{equation*}
    \begin{aligned}
        h_{p,K_j}(y)&= \frac1p\log\int_{K_j} e^{p\langle x,y\rangle}\frac{\dif x}{|K_j|} \geq \frac1p\log\int_{[-r,r]^n} e^{p\langle x,y\rangle}\frac{\dif x}{M'}= h_{p,[-r,r]^n}(y)+\log\frac{(2r)^n}{M'},
    \end{aligned}
    \end{equation*}
    and hence
    \begin{equation*}
        e^{-h_{p,K_j}(y)}\leq \frac{M'}{(2r)^n} e^{-h_{p,[-r,r]^n}(y)}.
    \end{equation*}
    The right-hand side is integrable since by \eqref{LpK_affine_transform}, 
    \begin{equation*}
    \int_{\R^n}e^{-h_{p,[-r,r]^n}(y)}\dif y= \frac{\M_p([-r,r]^n)}{|[-r,r]^n|}= \frac{1}{(2r)^n} \M_p([-1,1]^n),
    \end{equation*}   
    which is finite by Lemma \ref{mKmpKineq}.
    The claim now follows from \eqref{hausdorffEq2} and the dominated convergence theorem. 
\end{proof}

\begin{proof}[Proof of Lemma \ref{Mp_Hausdorff}]
    By Claims \ref{HausdorffClaim1}--\ref{HausdorffClaim2}, $|K_j|\to |K|$ and $|K_j^{\circ,p}|\to |K^{\circ,p}|$, thus by \eqref{MpKeq},
        $\lim_{j\to\infty}\M_p(K_j)= \lim_{j\to \infty} n! |K_j||K_j^{\circ,p}|= n! |K| |K^{\circ,p}|= \M_p(K). 
        $
\end{proof}

\subsection{Slice analysis of symmetric convex bodies}\label{SliceSymSection}
\subsubsection{Symmetry with respect to a hyperplane}

Antipodal slices are related when $-K= K$:
      $\xi \in K(-x_n)$ if and only if $(\xi,-x_n)\in K$ or $-(\xi,-x_n)= (-\xi,x_n)\in K$, i.e., if and only if $-\xi\in K(x_n)$. In sum:
\begin{claim}\label{symSlices}
    For a symmetric convex body $K\subset \R^n$, $K(-x_n)= -K(x_n)$ for all $x_n\in \R$. 
\end{claim}

If, instead, one assumes $K$ to be symmetric with respect to the hyperplane $e_n^\perp$, then antipodal slices are exactly equal: 
note that $\xi\in K(x_n)$ if and only if $(\xi,x_n)\in K$, which by the symmetry of $K$ with respect to $e_n^\perp$ is equivalent to $(\xi,-x_n)\in K$ or $\xi\in K(-x_n)$. Thus:

\begin{claim}\label{planeSymmetric}
    For a convex body $K\subset \R^n$ symmetric with respect to $e_n^\perp$, $K(-x_n)= K(x_n)$ for all $x_n\in \R$. 
\end{claim}

\subsubsection{\texorpdfstring{$L^p$-polarity preserves symmetries}{Lᵖ-polarity preserves symmetries}}

\begin{corollary}\label{KcircpSliceSym}
    Let $p\in(0,\infty]$. For a symmetric convex body $K$, $K^{\circ,p}(-x_n)= -K^{\circ,p}(x_n)$ for all $x_n\in \R$. 
\end{corollary}
\begin{proof}
    By Theorem \ref{LpKwelldefined}, $K^{\circ,p}$ is symmetric. Thus, by Claim \ref{symSlices}, $K^{\circ,p}(-x_n)= -K^{\circ,p}(x_n)$.
\end{proof}

In addition, $K^{\circ,p}$ inherits symmetries with respect to hyperplanes from $K$.
\begin{lemma}\label{KcircpSymPlane}
    Let $p\in (0,\infty], u\in\partial B_2^n$ and $K$ a convex body symmetric with respect to $u^\perp$. Then, $K^{\circ,p}$ is symmetric with respect to $u^\perp$. 
\end{lemma}
\begin{proof}
    By symmetry with respect to $u^\perp$, $\pi_{u^\perp}(K)= K\cap u^\perp$. There is concave $f:K\cap u^\perp\to [0,\infty)$ such that
    \begin{equation*}
        K= \{x+ tu: x\in K\cap u^\perp \text{ and } |t|\leq f(x)\}. 
    \end{equation*}
    For $y\in K\cap u^\perp, s\in \R$, 
    \begin{equation*}
        \begin{aligned}
            h_{p,K}(y+ s u)&= \frac1p \log\left(\int_K e^{p\langle z,y+ su\rangle}\frac{\dif z}{|K|} \right) \\
            &= \frac1p \log\left( \int_{x\in K\cap u^\perp}\int_{t=-f(x)}^{f(x)} e^{p\langle x+ tu, y+ su\rangle} \frac{\dif t\dif x}{|K|}\right)\\ 
            &= \frac1p \log\left( \int_{x\in K\cap u^\perp}e^{p\langle x, y\rangle}\int_{t=-f(x)}^{f(x)} e^{ p ts} \frac{\dif t\dif x}{|K|}\right) \\
            &= \frac1p \log\left( \int_{x\in K\cap u^\perp}e^{p\langle x, y\rangle}\int_{\tau=-f(x)}^{f(x)} e^{ -p \tau s} \frac{\dif \tau\dif x}{|K|}\right) \\
            &= h_{p,K}(y-su),
        \end{aligned}
    \end{equation*}
    by the change of variables $\tau=-t$.
    As a result, $\|y+su\|_{K^{\circ,p}}= \|y-su\|_{K^{\circ,p}}$, and hence $y+su\in K^{\circ,p}$ if and only if $y-su\in K^{\circ,p}$ as desired. 
\end{proof}

By Lemma \ref{SteinerProperties}, $\sigma_{e_n}K$ is symmetric with respect to $e_n^\perp$, thus, by Lemma \ref{KcircpSymPlane}, $(\sigma_{e_n}K)^{\circ,p}$ also is. Therefore, by Claim \ref{planeSymmetric} its antipodal slices are equal.

\begin{corollary}\label{SteinerKcircpCor}
    Let $p\in(0,\infty]$. For a convex body $K\subset \R^n$, $(\sigma_{e_n}K)^{\circ,p}(x_n)= (\sigma_{e_n}K)^{\circ,p}(-x_n)$ for all $x_n\in\R$. 
\end{corollary}

\subsection{Proof of Lemma \ref{steiner_claim2}}
\label{SantaloTechnicalSection}

The only two ingredients required for the proof of Lemma \ref{steiner_claim2} are H\"older's inequality and the log-convexity of $\sinh(t)/t$ (Claim \ref{steiner_claim3} below).
\begin{proof}[Proof of Lemma \ref{steiner_claim2}.]
Let $f,g: \pi_{e_n^\perp}(K)\to \R, g\leq f$, so that 
\begin{equation*}
    K= \left\{(\xi, x_n)\in \pi_{e_n^\perp}(K)\times \R: g(\xi) \leq x_n\leq f(\xi)\right\}. 
\end{equation*}
Then, 
\begin{equation*}
    \sigma_{e_n}K= \left\{(\xi, x_n)\in \pi_{e_n^\perp}(K)\times \R: |x_n|\leq \frac{f(\xi)- g(\xi)}{2}\right\}. 
\end{equation*}
In the integrals below it will be convenient to use slice-coordinates
$$
(\eta,y_n)\in \sigma_{e_{n}K}, 
\q
\hbox{ with \ } 
\eta\in (\sigma_{e_{n}}K)\cap e_n^\perp,
\q
y_n\in\R.
$$
    Since $|\sigma_{e_{n}}K|=|K|$ and $(\sigma_{e_{n}}K)\cap e_{n}^\perp = \pi_{e_n}^\perp (K)$, 
    \begin{equation}\label{eq101}
    \begin{aligned}
        h_{p, \sigma_{e_{n}K}}\left(r\frac{\xi+\xi'}{2}, r x_n\right)&= \frac1p\log\left( \int_{\sigma_{e_{n}}K} e^{p\langle r\frac{\xi+\xi'}{2}, \eta\rangle} e^{pr x_n y_n}\frac{\dif\eta\dif y_n}{|\sigma_{e_{n}}K|}\right)\\
        &= \frac1p \log\left( \int_{\eta\in (\sigma_{e_{n}}K)\cap e_n^\perp} \int_{y_n=-\frac{f(\eta)-g(\eta)}{2}}^{\frac{f(\eta)- g(\eta)}{2}}e^{pr\langle \frac{\xi+\xi'}{2}, \eta\rangle} e^{pr x_n y_n}\frac{\dif y_n\dif \eta}{|K|}\right)\\
        &= \frac1p \log\left( \int_{\pi_{e_n^\perp}(K)} e^{pr\langle \frac{\xi+\xi'}{2}, \eta\rangle} \frac{e^{pr x_n\frac{f(\eta)-g(\eta)}{2}}- e^{-prx_n\frac{f(\eta)- g(\eta)}{2}}}{pr x_n}\frac{\dif \eta}{|K|}\right)\\
        &= \frac1p \log\left( \int_{\pi_{e_n^\perp}(K)} e^{pr\langle \frac{\xi+\xi'}{2}, \eta\rangle} \frac{2}{prx_n}\sinh\left(pr x_n\frac{f(\eta)- g(\eta)}{2}\right)\frac{\dif \eta}{|K|}\right).
    \end{aligned}
    \end{equation}
    Also,
    \begin{equation}\label{eq102}
        \begin{aligned}
            h_{p, K}(t\xi, t x_n)&= \frac1p \log\left( \int_K e^{p\langle t\xi, \eta\rangle}e^{pt x_n y_n}\frac{\dif \eta\dif y_n}{|K|}\right)\\
            &= \frac1p \log\left(\int_{\pi_{e_n^\perp}(K)}\int_{y_n= g(\eta)}^{f(\eta)} e^{pt\langle \xi,\eta\rangle}e^{p t x_n y_n}\frac{\dif y_n\dif \eta}{|K|} \right)\\
            &= \frac1p \log\left( \int_{\pi_{e_n^\perp}(K)}e^{pt\langle \xi,\eta\rangle}\frac{1}{px_n t}\Big(e^{p x_n t f (\eta)}- e^{p x_n t g(\eta)}\Big)\frac{\dif \eta}{|K|}\right)\\
            &=\frac1p\log\left( \int_{\pi_{e_n^\perp}(K)}e^{pt\langle \xi,\eta\rangle}\frac{2}{px_n t}e^{px_n t\frac{f(\eta)+g(\eta)}{2}}\sinh\left(p x_n t\frac{f(\eta)-g(\eta)}{2}\right)\frac{\dif \eta}{|K|}\right), 
        \end{aligned}
    \end{equation}
    because 
    \begin{equation*}
    \begin{aligned}
    e^{px_n tf(\eta)}- e^{px_n tg(\eta)}&= e^{px_n t\frac{f(\eta)+g(\eta)}{2}}\Big(e^{px_n t\frac{f(\eta)-g(\eta)}{2}}- e^{-px_n t\frac{f(\eta)-g(\eta)}{2}}\Big)\\
    &=2e^{px_n t\frac{f(\eta)+g(\eta)}{2}}\sinh\left(p x_n t\frac{f(\eta)-g(\eta)}{2}\right).
    \end{aligned}
    \end{equation*}
    Similarly, 
    \begin{equation}\label{eq103}
    \begin{aligned}
        &h_{p,K}(s\xi' , -sx_n)
        \\
        &= \frac1p\log\left( \int_{\pi_{e_n^\perp}(K)} e^{p\langle s\xi', \eta\rangle} \frac{2}{p(-sx_n )}e^{-px_n s\frac{f(\eta)+g(\eta)}{2}} \sinh\left(p(-s x_n)\frac{f(\eta)-g(\eta)}{2}\right)\frac{\dif \eta}{|K|}\right) \\
        &= \frac1p\log\left( \int_{\pi_{e_n^\perp}(K)} e^{ps\langle \xi',\eta\rangle} \frac{2}{px_n s}e^{-px_n s\frac{f(\eta)+g(\eta)}{2}} \sinh\left(px_n s\frac{f(\eta)-g(\eta)}{2}\right)\frac{\dif \eta}{|K|}\right). 
    \end{aligned}
    \end{equation}
By \eqref{eq102}--\eqref{eq103} and H\"older's inequality,
\begin{equation}\label{eq104}
\begin{aligned}
    &\frac{s}{t+s}h_{p,K}(t\xi, t x_n)+ \frac{t}{t+s} h_{p,K}(s\xi', -s x_n)\\
    &= \frac1p \log\left[\left(\int_{\pi_{e_n^\perp}(K)} e^{pt\langle \xi,\eta\rangle} \frac{2}{px_n t}e^{px_n t\frac{f(\eta)+g(\eta)}{2}}\sinh\left(px_n t\frac{f(\eta)-g(\eta)}{2}\right) \frac{\dif \eta}{|K|}\right)^{\frac{s}{t+s}}\right.\\
    &\hspace{1.6cm} \left.\left(\int_{\pi_{e_n^\perp}(K)}e^{ps\langle \xi', \eta\rangle}\frac{2}{px_n s} e^{-px_n s\frac{f(\eta)+g(\eta)}{2}}\sinh\left(px_n s\frac{f(\eta)-g(\eta)}{2}\right)\frac{\dif \eta}{|K|}\right)^{\frac{t}{t+s}}\right]\\
    &\geq \frac1p\log\left( \int_{\pi_{e_n^\perp}(K)} e^{p\frac{ts}{t+s}\langle \xi,\eta\rangle} \left(\frac{2}{px_n t}\right)^{\frac{s}{t+s}} e^{px_n \frac{ts}{t+s}\frac{f(\eta)+g(\eta)}{2}}\left(\sinh\left(px_n t\frac{f(\eta)-g(\eta)}{2}\right)^{\frac{s}{t+s}}\right)\right.\\
    &\hspace{2.9cm}\left. e^{p\frac{ts}{t+s}\langle \xi',\eta\rangle} \left(\frac{2}{px_n s}\right)^{\frac{t}{t+s}} e^{-px_n \frac{ts}{t+s}\frac{f(\eta)+g(\eta)}{2}} \left(\sinh\left(px_n s\frac{f(\eta)-g(\eta)}{2}\right)^{\frac{t}{t+s}}\right)\frac{\dif \eta}{|K|}\right) \\
    &=\frac1p\log\left(\int_{\pi_{e_n^\perp}(K)} e^{p\frac{ts}{t+s}\langle \xi+ \xi',\eta\rangle} J(\eta,t)^{\frac{s}{t+s}} J(\eta,s)^{\frac{t}{t+s}}\frac{\dif \eta}{|K|}\right)\\
    &= \frac1p\log\left(\int_{\pi_{e_n^\perp}(K)} e^{p r\langle \frac{\xi+\xi'}{2},\eta\rangle} J(\eta,t)^{\frac{s}{t+s}} J(\eta,s)^{\frac{t}{t+s}}\frac{\dif \eta}{|K|}\right), 
\end{aligned}
\end{equation}
where
\begin{equation*}
    J(\eta,t)\defeq \frac{2}{px_n t}\sinh\left(px_n t\frac{f(\eta)- g(\eta)}{2}\right).
\end{equation*}
By Claim \ref{steiner_claim3} below, $\log J$ is convex in $t$, and therefore,
\begin{equation}\label{eq105}
    J(\eta,t)^{\frac{s}{t+s}} J(\eta,s)^{\frac{t}{t+s}}\geq J\left(\eta, \frac{s}{t+s}t+ \frac{t}{t+s}s\right)= J\left(\eta, \frac{2ts}{t+s}\right)= J(\eta, r),
\end{equation}
because $\frac{2ts}{t+s}= r$. Therefore, by \eqref{eq101}, \eqref{eq104} and \eqref{eq105},
\begin{equation*}
    \begin{aligned}
        &\frac{s}{t+s}h_{p,K}(t\xi, t x_n)+ \frac{t}{t+s} h_{p,K}(s\xi', -s x_n)
        \\
        &\geq \frac1p\log\left(\int_{K\cap e_{n}^\perp} e^{p r\langle \frac{\xi+\xi'}{2},\eta\rangle}\frac{2}{px_n r}\sinh\left(px_n r\frac{f(\eta)- g(\eta)}{2}\right)\frac{\dif \eta}{|K|}\right)
        \\
        &= h_{p, \sigma_{e_{n}}K}\Big(r\frac{\xi+\xi'}{2}, rx_n\Big),
    \end{aligned}
\end{equation*}
as desired. 
\end{proof}

\begin{claim}\label{steiner_claim3}
    For any $x>0$, $t\mapsto \log\left(\frac1t \sinh(tx)\right), t>0,$ is convex.
\end{claim}
\begin{proof}
    Write 
    $$f(t)\defeq \log\left(\frac1t \sinh(tx)\right)= \log(\sinh(tx))-\log t.
    $$ 
    Compute the derivatives, $f'(t)= x\frac{\cosh(tx)}{\sinh(tx)}-\frac1t$ and 
    \begin{equation*}
    \begin{aligned}
        f''(t)&= x^2\frac{\sinh(tx)}{\sinh(tx)}- x^2\frac{(\cosh(tx))^2}{(\sinh(tx))^2}+\frac{1}{t^2}= x^2\left(1- \frac{(\cosh(tx))^2}{(\sinh(tx))^2}+\frac{1}{(tx)^2}\right)\\
        &= x^2\left(1- \frac{1+(\sinh(tx))^2}{(\sinh(tx))^2}+ \frac{1}{(tx)^2}\right)= x^2\left(\frac{1}{(tx)^2}- \frac{1}{(\sinh(tx))^2}\right)\geq 0,
    \end{aligned}
    \end{equation*}
    because $\sinh(y)\geq y$, for all $y\geq 0$.
\end{proof}

\subsection{\texorpdfstring{
Slice analysis of 
$\LpK$ under Steiner symmetrization}{Slice analysis of Kᵒᵖ under Steiner symmetrization}}\label{KcircpSteiner}

\subsubsection{A monotonicity property for the average of antipodal slices}

For the proof of Lemma \ref{steiner_cor}, we first prove   \eqref{steiner_symm_lem1}. The aim is to apply the following theorem due to Ball \cite[Theorem 4.10]{ball} for $F,G,H$ 
appropriate exponentials of the $L^p$-support functions. 
\begin{theorem}\label{KBall_ineq}
Let $F, G, H: (0,\infty)\to [0,\infty)$ be measurable functions, not almost everywhere 0, with
\begin{equation}\label{fghCondition}
    H(r)\geq F(t)^{\frac{s}{t+s}} G(s)^{\frac{t}{t+s}}, \quad \text{ for all} \quad \frac2r=\frac1t+\frac1s.  
\end{equation}
Then, for $q\geq 1$, 
\begin{equation*}
    2\left(\int_0^\infty r^{q-1} H(r)\dif r \right)^{-1/q}\leq \left( \int_0^\infty t^{q-1} F(t)\dif t\right)^{-1/q}+ \left(\int_0^\infty s^{q-1}G(s)\dif s\right)^{-1/q}.
\end{equation*}
\end{theorem}

For the reader's convenience, we give a proof in Appendix \ref{appendixA}.  
Applying Theorem \ref{KBall_ineq} to prove Lemma \ref{steiner_symm_lem1} becomes possible by Lemma \ref{steiner_claim2}. 

\begin{proof}[Proof of Lemma \ref{steiner_cor}]
Set 
\begin{equation*}
\begin{gathered}
F(t)\defeq e^{-h_{p,K}(t\xi, t x_n)}, \\ 
G(s)\defeq e^{-h_{p, K}(s\xi', -sx_n)}, \\
H(r)\defeq e^{-h_{p,\sigma_{e_{n}}K}(r\frac{\xi+\xi'}{2}, rx_n)}. 
\end{gathered}
\end{equation*}
By Lemma \ref{steiner_claim2}, for any $t,s>0$ with $\frac{2}{r}=\frac1t+ \frac1s$, $H(r)\geq F(t)^{\frac{s}{t+s}}G(s)^{\frac{t}{t+s}}$, thus, by Theorem \ref{KBall_ineq} for $q= n$,
\begin{equation*}
    \begin{aligned}
        \left\|\left(\frac{\xi+\xi'}{2}, x_n\right)\right\|_{(\sigma_{e_{n}}K)^{\circ,p}}&= \left(\frac{1}{(n-1)!}\int_{0}^\infty r^{n-1} e^{-h_{p,\sigma_{e_{n}}K}(r\frac{\xi+\xi'}{2}, rx_n)}\dif r\right)^{-\frac{1}{n}}\\
        &\leq \frac12\left(\frac{1}{(n-1)!}\int_{0}^\infty t^{n-1} e^{-h_{p,K}(t\xi, t x_n)}\dif t\right)^{-\frac{1}{n}}\\
        &\hspace{.5cm}+\frac12\left(\frac{1}{(n-1)!}\int_0^\infty s^{n-1} e^{-h_{p,K}(s\xi', -s x_n)}\dif s\right)^{-\frac{1}{n}}\\
        &= \frac12 \|(\xi,x_n)\|_{\LpK}+ \frac12 \|(\xi',-x_n)\|_{\LpK}. 
    \end{aligned}
\end{equation*}
verifying \eqref{steiner_symm_lem1}. 

    For $\xi\in (\LpK)(x_n)$ and $\xi'\in (\LpK)(-x_n)$, by definition \eqref{sliceDef}, $(\xi,x_n)\in\LpK$ and $(\xi',-x_n)\in\LpK$, i.e., $\|(\xi,x_n)\|_{\LpK}\leq 1$ and $\|(\xi',-x_n)\|_{\LpK}\leq 1$. By \eqref{steiner_symm_lem1}, 
    $$
    \left\|\Big(\frac{\xi+\xi'}{2}, x_n\Big)\right\|_{\big(\sigma_{e_{n}}K)^{\circ,p}}\leq \frac{\|(\xi,x_n)\|_{\LpK}+ \|(\xi',-x_n)\|_{\LpK}}{2}\leq 1,
    $$
    i.e., $(\frac{\xi+\xi'}{2}, x_n)\in(\sigma_{e_{n}}K)^{\circ,p}$, or $\frac{\xi+\xi'}{2}\in (\sigma_{e_{n}}K)^{\circ,p}(x_n)$.
    Finally, by Corollary \ref{SteinerKcircpCor}, $(\sigma_{e_n^\perp}K)^{\circ,p}(x_n)= (\sigma_{e_n^\perp}K)^{\circ,p}(-x_n)$ hence the equality in the right-hand side of \eqref{averageEq}.
\end{proof}

\subsubsection{Monotonicity of the volume of slices under Steiner symmetrization}

\begin{proof}[Proof of Lemma \ref{sliceVolume}.]
By the Brunn--Minkowski inequality and Lemma \ref{steiner_cor}, 
\begin{equation*}
    \begin{aligned}
        |(\sigma_{e_n}K)^{\circ,p}(x_n)|^{\frac{1}{n-1}}&\geq \frac{|K^{\circ,p}(x_n)+ K^{\circ,p}(-x_n)|^{\frac{1}{n-1}}}{2} \\
        &\geq \frac{|K^{\circ,p}(x_n)|^{\frac{1}{n-1}}+ |K^{\circ,p}(-x_n)|^{\frac{1}{n-1}}}{2} \\
        &= |K^{\circ,p}(x_n)|^{\frac{1}{n-1}},
    \end{aligned}
\end{equation*}
because $K$ is symmetric thus, by Corollary \ref{KcircpSliceSym}, $K^{\circ,p}(-x_n)=-K^{\circ,p}(x_n)$, and hence their volumes are equal $|K^{\circ,p}(-x_n)|= |K^{\circ,p}(x_n)|$. 
\end{proof}

\subsection{Proof of 
Theorem \ref{santalo_symmetric}}
We now complete the proofs of Proposition \ref{steiner_ppolar} and Theorem \ref{santalo_symmetric}.

\begin{proof}[Proof of Proposition \ref{steiner_ppolar}.]\label{ppolarSantaloProof}
Take for a moment $u= e_{n}$. 
By \eqref{sliceVolumeIntegral} and Lemma \ref{sliceVolume},
\begin{equation}\label{steiner_eq2}
    |(\sigma_{e_{n}}K)^{\circ,p}|= \int_{-\infty}^\infty |(\sigma_{e_{n}}K)^{\circ,p}(x_n)|\dif x_n\geq \int_{-\infty}^\infty |(\LpK)(x_n)|\dif x_n= |\LpK|.
\end{equation}
In general, for $u\in \partial B_2^n$, there is $A\in O(n)$ such that $Au= e_{n}$. By Lemma \ref{steiner_orthogonal}, $\sigma_uK= A^{-1}(\sigma_{Au}(AK))= A^{-1}(\sigma_{e_{n}}(AK))$. By (\ref{LpK_affine_transform}), $(\sigma_uK)^{\circ,p}= (A^{-1}\sigma_{e_n}(AK))^{\circ,p}= A^T (\sigma_{e_n}(AK))^{\circ,p}$, thus by (\ref{steiner_eq2}),
\begin{equation*}
    |(\sigma_uK)^{\circ,p}|= |\det A^T| |(\sigma_{e_{n}}(AK))^{\circ,p}|\geq |\det A^T| |(AK)^{\circ,p}|= |A^T (AK)^{\circ,p}|= |\LpK|, 
\end{equation*}
because, again by (\ref{LpK_affine_transform}), $A^T (AK)^{\circ,p}= (A^{-1}A K)^{\circ,p}= \LpK$.
\end{proof}

Theorem \ref{santalo_symmetric} follows from Proposition \ref{steiner_ppolar} and the fact that repeated Steiner symmetrizations converge to a dilated $2$-ball (Lemma \ref{steiner_limit}).
\begin{proof}[Proof of Theorem \ref{santalo_symmetric}.]
There is $\rho>0$ and a sequence $\{u_j\}_{j\geq 1}\subset \partial B_2^n$ such that for 
\begin{equation*}
    K_0\defeq K, \quad K_j\defeq \sigma_{u_j}K_{j-1},
\end{equation*}
$K_j\to \rho B_2^n$ in the Hausdorff metric \cite[Theorem 1.1.16]{artstein-giannopoulos-milman}. By Proposition \ref{steiner_ppolar}, 
\begin{equation*}
    \M_p(K_j)= n! |K_j||\LpK_j|= n!|K| |\LpK_j|\leq n! |K||(\sigma_{u_{j+1}}K_j)^{\circ, p}|= n! |K| |\LpK_{j+1}|= \M_p(K_{j+1}). 
\end{equation*}
In particular, $\M_p(K)\leq \M_p(K_j)$ for all $j$. Sending $j\to \infty$, $K_j\to \rho B_2^n$ in the Hausdorff metric, and hence, by Lemma \ref{Mp_aff_inv} and Lemma \ref{Mp_Hausdorff}, $\M_p(K_j)\to \M_p(\rho B_2^n)= \M_p(B_2^n)$, thus $\M_p(K)\leq \M_p(B_2^n)$.
\end{proof}

\section{A connection to Bourgain's slicing problem}\label{S5}

In this section we explore the relationship between the $L^p$ support functions
$h_{p,K}$
\eqref{hpdef} and the slicing problem (Conjecture \ref{SlicingConjecture}). 
The aim is to prove Theorem \ref{bern_iso_prop} and then illustrate how it implies a
sub-optimal upper bound on the isotropic constant (Corollary \ref{suboptimal_bound})
originally due to Milman--Pajor. We also explain some interesting connections
to and motivations from complex geometry.

In \S \ref{S51} we recall the definitions of the covariance matrix and the isotropic constant, and relate these to $h_{p,K}$ 
(Lemma \ref{MA_isotropic}). In \S \ref{S52} we recall the definition of the \MA measure and its basic properties. Theorem \ref{bern_iso_prop} is proved in \S \ref{S53}. The proof consists of two parts: using Jensen's inequality to bound $\int\log\det\nabla^2 h_{1,K}$ (Lemma \ref{berniso_lem2}), and then bounding $\int_{\R^n} h_{1,K}\dif\nu_{p,K}$ (Lemma \ref{berniso_lem1}). In \S \ref{S54}, we show $\log\det\nabla^2 h_{p,K}+ p(n+1)h_{p,K}$ is convex, proving Theorem \ref{bern_iso_theorem}. From Theorems \ref{bern_iso_prop} and \ref{bern_iso_theorem} we then obtain an upper bound on the isotropic constant of order $O(\sqrt{n})$ (Corollary \ref{suboptimal_bound}). In \S \ref{LpSupportMeasureS}, we define the $L^p$ support functions of compactly supported probability measures and show that Theorem \ref{bern_iso_theorem} cannot be improved in that setting (Example \ref{simplexExample}). 
Finally, in \S\ref{complexgeom_subsec} we explain some novel connections of our work
to complex geometry, in particular to Ricci curvature, Fubini--Study metrics, 
Bergman metrics, Kobayashi's Theorem, and holomorphic line bundles.

\subsection{Preliminaries}
\subsubsection{Affine-invariance of \texorpdfstring{$\calC$}{Affine-invariance of C}}
\lb{AffineC}

The isotropic constant is an affine invariant
(e.g., \cite[p. 77]{brazitikos_etal}), hence so is $\calC$. 
As we could not find precisely the following lemma in the literature, we include its proof for completeness.

\begin{lemma}\label{cov_lem0.1}
  For $K\subset \R^n$ and $A\in GL(n,\R), b\in \R^n$,
    \begin{equation*}
    \mathrm{Cov}(AK+ b)= A\mathrm{Cov}(K)A^T,
    \end{equation*}
    where $\mathrm{Cov}(K)$ is defined in \eqref{CovEq}.
\end{lemma}
\begin{proof}
    Write $A= [A_i^j]_{i,j=1}^n$, $b= (b_1,\ldots, b_n)$ and $T(x)= Ax+b$. The Einstein summation convention of summing over repeated indices is used. Changing variables $y= T^{-1}x= A^{-1}x- A^{-1}b$, $\dif y= |\det A^{-1}|\dif x= |\det A|^{-1}\dif x$, 
{\allowdisplaybreaks
\begin{align*}
  \Cov_{ij}(AK+b)&= \int_{T(K)}x_ix_j\frac{\dif x}{|T(K)|}- \int_{T(K)} x_i\frac{\dif x}{|T(K)|}\int_{T(K)} x_j\frac{\dif x}{|T(K)|}\\
  &= \int_{K} (Ay+b)_i (Ay+b)_j\frac{|\det A|\dif y}{|AK+b|}\\
  &\hspace{.5cm}- \int_{K} (Ay+b)_i\frac{|\det A|\dif y}{|AK+b|}\int_{K} (Ay+b)_j\frac{|\det A|\dif y}{|AK+b|} \\
  &= \int_{K}(A_i^k y_k+ b_i)(A_j^l y_l+ b_j)\frac{\dif y}{|K|}- \int_{K} (A_i^k y_k+b_i)\frac{\dif y}{|K|}
  \int_{K}(A_j^l y_l+ b_j)\frac{\dif y}{|K|} \\
  &= A_{i}^k A_j^l\int_{K}y_k y_l \frac{\dif y}{|K|}+ b_j A_i^k\int_{K} y_k\frac{\dif y}{|K|}+ b_i A_j^l\int_{K}y_j\frac{\dif y}{|K|}+ b_ib_j \\
  &\hspace{.5cm}-A_{i}^k A_j^l\int_{K}y_ky_l\frac{\dif y}{|K|} - b_j A_i^k\int_{K}y_k \frac{\dif y}{|K|}- b_i A_j^l\int_{K}y_l\frac{\dif y}{|K|}- b_ib_j\\
  &= A_i^k A_j^l \left( \int_{K}y_ky_l\frac{\dif y}{|K|}-\int_{K}y_k \frac{\dif y}{|K|}\int_{K}y_l\frac{\dif y}{|K|}\right)\\
  &= A_i^k A_j^l \Cov_{kl}(K), 
\end{align*}
}
proving the claim.
\end{proof}

    Let $A\in GL(n,\R)$ and $b\in \R^n$. By Lemma \ref{cov_lem0.1},
    \begin{equation*}
        \calC(AK+b)= \frac{|AK+b|^2}{\det\mathrm{Cov}(AK+b)}= \frac{(\det A)^2 |K|^2}{\det(A\mathrm{Cov}(K)A^T)}= \frac{(\det A)^2 |K|^2}{(\det A)^2 \det\mathrm{Cov}(K)}= \calC(K),
    \end{equation*}
proving:

\begin{corollary}\label{calCaffineInv}
    $\calC$ is an affine invariant.
\end{corollary}

\subsubsection{\texorpdfstring{$L^p$-support functions and the isotropic constant}{Lᵖ-support functions and the isotropic constant}}\label{S51}

Next, we relate
the functional $\calC$ \eqref{CDefEq}  to $h_{p,K}$ \eqref{hpdef}
(for $p=1$ see Klartag \cite[Lemma 3.1]{klartag}).

\begin{lemma}\label{MA_isotropic}
    Let $p>0$. For a convex body $K\subset \R^n$, $\nabla^2 h_{p,K}(0)= p\mathrm{Cov}(K)$ and 
    $$
        \calC(K)= \frac{p^n |K|^2}{\det\nabla^2 h_{p,K}(0)}. 
    $$
\end{lemma}
\begin{proof}
By direct calculation, 
\begin{equation*}
    \frac{\partial}{\partial y_i}h_{p,K}(y)= \frac{\int_K x_i e^{p\langle x,y\rangle}\dif x}{\int_K e^{p\langle x,y\rangle}\dif x}, 
\end{equation*}
and 
\begin{equation*}
    \begin{aligned}
        \frac{\partial^2}{\partial y_i\partial y_j}h_{p,K}(y)= \frac{p\int_K x_ix_j e^{p\langle x,y\rangle}\dif x \int_K e^{p\langle x,y\rangle}\dif x- p\int_K x_ie^{p\langle x,y\rangle}\dif x \int_K x_j e^{p\langle x,y\rangle}\dif x}{\left(\int_K e^{p\langle x,y\rangle}\dif x\right)^2}. 
    \end{aligned}
\end{equation*}
Since for $y=0$, $\int_K e^{p\langle x,0\rangle}\dif x= |K|$,
\begin{equation*}
\frac{\partial^2 h_{p,K}}{\partial y_i\partial y_j}(0)= p\int_{K} x_ix_j \frac{\dif x}{|K|}- p\int_K x_i\frac{\dif x}{|K|}\int_K x_j\frac{\dif x}{|K|}= p\mathrm{Cov}_{i,j}(K),
\end{equation*}
and
\begin{equation*}
    \det\nabla^2h_{p,K}(0)= \det(p \mathrm{Cov}(K))= p^n \det\mathrm{Cov}(K)= p^n\frac{|K|^2}{\calC(K)},
\end{equation*}
as claimed.
\end{proof}

\subsubsection{The \MA measure}\label{S52}
We review some basic details concerning
the Monge--Amp\`ere measure, following Rauch--Taylor \cite{rauch-taylor}. 
Legendre duality is defined by $f^*(y):=\sup_{x\in\R^n}[\langle y,x\rangle-f(x)]$.
\begin{definition} {\rm\cite[p. 215]{rockafellar}}
    For a convex function $\phi: \R^n\to \R\cup\{\infty\}$ and $x\in \R^n$, 
    the subdifferential of $\phi$ at $x$ is 
    \begin{equation*}
        \partial\phi(x)\defeq \{y\in \R^n: \phi(z)\geq \phi(x)+ \langle y, z-x\rangle, \text{ for all } z\in\R^n\}. 
    \end{equation*}
\end{definition}

\begin{lemma}\label{subgradient_lem1}
{\rm \cite[Theorem 23.5]{rockafellar}}
    For $\phi: \R^n\to \R$ convex, $\partial\phi(\R^n)\subset \{\phi^*<\infty\}$. 
\end{lemma}
\begin{proof}
    By definition of the subgradient, for $y\in\partial \phi(x)$, $\phi(z)\geq \phi(x)+ \langle y, z-x\rangle$ for all $z\in\R^n$, i.e., $\langle y,x\rangle-\phi(x)\geq \langle y,z\rangle-\phi(z)$. Taking supremum over all $z\in \R^n$, 
    \begin{equation*}
        \phi^*(y)\leq \langle y,x\rangle-\phi(x)<\infty, 
    \end{equation*}
    as claimed. 
\end{proof}

\begin{corollary}\label{subgradient_cor1}
    $\partial h_K(\R^n)\subset K$ and $\partial h_{p,K}(\R^n)\subset K$.
\end{corollary}
\begin{proof}
    Since $h_K^*= \mathbf{1}_K^\infty$, by Lemma \ref{subgradient_lem1}, $\partial h_K(\R^n)\subset \{\mathbf{1}_K^\infty<\infty\}=K$. Similarly, since $h_{p,K}\leq h_K$, the Legendre transform $\mathbf{1}_K^\infty= h_K^*\leq h_{p,K}^*$ thus, by Lemma \ref{subgradient_lem1}, 
    \begin{equation*}
        \partial h_{p,K}(\R^n)\subset \{h_{p,K}^*<\infty\}\subset \{\mathbf{1}_K^\infty<\infty\}= K. 
    \end{equation*}
\end{proof}

\begin{definition}\label{MAdef}
{\rm \cite[Definition 2.6]{rauch-taylor}}
    For a convex function $\phi$, let
    \begin{equation*}
        (\mathrm{MA}\phi)(U)\defeq |\partial\phi(U)|,
    \end{equation*}
    where the right hand side denotes the Lesbegue measure
    of $\partial\phi(U)$ in $\R^n$.
\end{definition}

\begin{lemma}\label{MAhpK}
$
        \mathrm{MA}h_{p,K}(\R^n)\leq |K|.
$
\end{lemma}
\begin{proof}
By definition, $\mathrm{MA}h_{p,K}(\R^n)= |\partial h_{p,K}(\R^n)|\leq |K|$ because $\partial h_{p,K}(\R^n)\subset K$
by Corollary \ref{subgradient_cor1}.
\end{proof}

\begin{remark}
\lb{MADensRem}
In fact, equality holds in Lemma \ref{MAhpK}. In particular, $h_{p,K}$ is a smooth, strictly convex function with $\nabla h_{p,K}(\R^n)= \mathrm{int}\,K$ (see \cite[Lemma 3.1]{klartag} for the case $p=1$). By smoothness of $h_{p,K}$
we also know the density of $\mathrm{MA}h_{p,K}$ equals $\det\nabla^2h_{p,K}$.
\end{remark}

\subsection{Conditional lower bounds on the isotropic constant}
\label{S53}

The proof of Theorem \ref{bern_iso_prop} relies on the following observation. Assume that $K$ satisfies \eqref{B} for some $B>0$, i.e.,
    \begin{equation*}
        u_{B,K}(y)\defeq \log\det\nabla^2h_{1,K}(y)+ B h_{1,K}(y)
    \end{equation*}
    is convex.
    Note that $h_{1,K}(0)= 0$ thus $u_{B,K}(0)= \log\det\nabla^2 h_{1,K}(0)$. By Lemma \ref{MA_isotropic},
\begin{equation}\label{berniso_e1}
    \calC(K)= \frac{|K|^2}{\det\mathrm{Cov}(K)}= \frac{|K|^2}{\det\nabla^2h_{1,K}(0)}= |K|^2 e^{-u_{B,K}(0)}.
\end{equation}
Since $u_{B,K}$ is convex by assumption, for a probability measure $\mu$ with $b(\mu)=0$, by Jensen's inequality, 
\begin{equation}\label{berniso_e2}
\begin{aligned}
    u_{B,K}(0)&= u_{B,K}\left( \int_{\R^n} y\dif\mu(y)\right)\\
    &\leq \int_{\R^n} u(y)\dif\mu(y)\\
    &= \int_{\R^n} \log\det\nabla^2h_{1,K}(y)\dif\mu(y)+ B\int_{\R^n} h_{1,K}(y)\dif\mu(y).
\end{aligned}
\end{equation}
By (\ref{berniso_e1}) and (\ref{berniso_e2}), in order to get bounds on $\calC(K)$ it is enough to bound $\int \log\mathrm{MA}h_{1,K}\dif \mu$ and $\int_{\R^n} h_{1,K}\dif \mu$, for a suitable probability measure $\mu$.

Here, we consider the probability measures
\eqref{nuMeasure}
for which we obtain the desired bounds (Lemmas \ref{berniso_lem2} and \ref{berniso_lem1}). 
By Corollary \ref{santalo_point_corollary}, we may translate $K$ to a suitable position in order to obtain estimates on $\int_{\R^n}\log\mathrm{MA} h_{1,K}(y)\dif \nu_{p,K}(y)$ (Lemma \ref{berniso_lem2} (ii) and (iii)).

\subsubsection{A bound on \texorpdfstring{$\int\log\det\nabla^2h_{1,K}$}{A bound on the integral of the logarithm of the determinant of the Hessian of h₁,ₖ in terms of Lᵖ-Mahler volumes}
in terms of 
\texorpdfstring{$L^p$}{TEXT}-Mahler volumes}
\begin{lemma}\label{berniso_lem2}
    Let $p>0$. For a convex body $K\subset \R^n$, and $\nu_{p,K}$ as in \eqref{nuMeasure}, 

\noindent 
(i) \begin{equation*}
    \int_{\R^n}\log\det\nabla^2 h_{1,K}(y)\dif\nu_{p,K}(y)\leq\log\left( |K|^2 \frac{\M_{1/(2p)}(K)}{\M_{1/p}(K)^2} \frac{p^n}{2^n}\right). 
\end{equation*}

\noindent
(ii) If $b(\nu_{p,K})=0$, then
\begin{equation*}
    \int_{\R^n}\log\det\nabla^2 h_{1,K}(y)\dif\nu_{p,K}(y)\leq \log\left(\frac{|K| e^n}{ \int_{\R^n} e^{-p h_{1,K}(y)}\dif y} \right).
\end{equation*}

\noindent
(iii) If $b(K)=0$, then
\begin{equation*}
    \int_{\R^n}\log\det\nabla^2 h_{1,K}(y)\dif\nu_{p,K}(y)\leq \log\left(\frac{|K|}{ \int_{\R^n} e^{-p h_{1,K}(y)}\dif y} \right).
\end{equation*}
\end{lemma}

For the proof of Lemma \ref{berniso_lem2} we need the following. 
\begin{claim}\label{bern_iso_claim2}
    Let $p>0$. For a convex body $K\subset \R^n$,
    \begin{equation*}
        \int_{\R^n}\log\det\nabla^2 h_{1,K}(y)\dif\nu_{p,K}(y)\leq \log\left(|K|\frac{\int_{\R^n} e^{-2ph_{1,K}(y)}\dif y}{\left( \int_{\R^n} e^{-p h_{1,K}(y)}\dif y\right)^2} \right). 
    \end{equation*}
\end{claim}
\begin{proof}
    By Jensen's inequality and Cauchy--Schwarz, 
    \begin{equation*}
        \begin{aligned}
            \int_{\R^n} \log\det\nabla^2 h_{1,K}(y)\dif \nu_{p,K}(y)&= 2\int_{\R^n} \log(\det\nabla^2 h_{1,K}(y))^{1/2}\dif \nu_{p, K}(y)\\
            &\leq 2\log\int_{\R^n} (\det\nabla^2 h_{1,K}(y))^{1/2} \dif\nu_{p,K}(y)\\
            &= 2\log\int_{\R^n} (\det\nabla^2 h_{1,K}(y))^{1/2} \frac{e^{-p h_{1,K}(y)}}{\int_{\R^n} e^{-ph_{1,K(y)}}\dif y}\dif y\\
            &\leq 2 \log\left( \int_{\R^n} \det\nabla^2 h_{1,K}(y)\dif y \frac{\int_{\R^n} e^{-2ph_{1,K}(y)}\dif y}{\left( \int_{\R^n} e^{-p h_{1,K}(y)}\dif y \right)^2}\right)^{1/2}\\
            &\leq \log\left( |K| \frac{\int_{\R^n} e^{-2ph_{1,K}(y)}\dif y}{\left( \int_{\R^n} e^{-p h_{1,K}(y)}\dif y \right)^2}\right), 
        \end{aligned}
    \end{equation*}
    because by Lemma \ref{MAhpK} and Remark \ref{MADensRem}, $\int_{\R^n}\det\nabla^2 h_{1,K}(y)\dif y= \mathrm{MA}h_{1,K}(\R^n)\leq |K|$. 
\end{proof}

\begin{proof}[Proof of Lemma \ref{berniso_lem2}]
(i) In view of Claim \ref{bern_iso_claim2}, it is enough to compute the following two integrals,
\begin{equation*}
\begin{aligned}
    \int_{\R^n} e^{-2p h_{1,K}(y)}\dif y&= \frac{1}{(2p)^n}\int_{\R^n} e^{-2p h_{1,K}(y/(2p))}\dif y\\
    &= \frac{1}{(2p)^n} \int_{\R^n} e^{-h_{1/(2p), K}(y)}\dif y\\
    &= \frac{1}{(2p)^n}\frac{\M_{1/(2p)}(K)}{|K|^{1+ 2p}}, 
\end{aligned}
\end{equation*}
because by Lemma \ref{list} (i), $2p h_{1,K}(y/(2p))= h_{1/(2p),K}(y)$.  Similarly, 
\begin{equation*}
    \int_{\R^n} e^{-p h_{1,K}(y)}\dif y= \frac{1}{p^n} \frac{\M_{1/p}(K)}{|K|^{1+p}}. 
\end{equation*}
Therefore, 
\begin{equation}\label{berniso_e3}
    |K|\frac{\int_{\R^n} e^{-2p h_{1,K}(y)}\dif y}{\left( \int_{\R^n} e^{-p h_{1,K}(y)}\dif y\right)^2}= |K| \frac{\M_{1/(2p)}(K)}{(2p)^{n}|K|^{1+2p}} \frac{p^{2n}|K|^{2+2p}}{\M_{1/p}(K)^2}= |K|^2\frac{p^n}{2^n}\frac{\M_{1/(2p)}(K)}{\M_{1/p}(K)^2}.
\end{equation}
The claim follows from (\ref{berniso_e3}) and Claim \ref{bern_iso_claim2}.

\medskip
\noindent 
(ii) Since $b(\nu_{p,K})= b(ph_{1,K})=0$, by Lemma \ref{fradelizi_lemma} below, $ph_{1,K}(y)\geq p h_{1,K}(0)-n= -n$. Therefore, 
\begin{equation}\label{berniso_e4}
    \int_{\R^n} e^{-2p h_{1,K}(y)}\dif y= \int_{\R^n} e^{-ph_{1,K}(y)} e^{-p h_{1,K}(y)}\dif y\leq e^n\int_{\R^n} e^{-p h_{1,K}(y)}\dif y.
\end{equation}
The claim follows directly from (\ref{berniso_e4}) and Claim \ref{bern_iso_claim2}.

\medskip
\noindent 
(iii) Since $b(K)=0$, $\int_K\langle x,y\rangle \dif x=0$ for all $y\in \R^n$. As a result, by Jensen's inequality,
\begin{equation*}
    h_{1,K}(y)= \log\int_{K} e^{\langle x,y\rangle}\frac{\dif x}{|K|}\geq \int_{K} \log e^{\langle x,y\rangle}\frac{\dif x}{|K|}= \int_K \langle x,y\rangle \frac{\dif x}{|K|}=0, 
\end{equation*}
i.e., $h_{1,K}(y)\geq 0$. Therefore, for $p>0$, $2p h_{1,K}(y)\geq p h_{1,K}(y)$ and hence
\begin{equation}\label{berniso_e5}
    \int_{\R^n} e^{-2p h_{1,K}(y)}\dif y\leq \int_{\R^n} e^{-ph_{1,K}(y)}\dif y. 
\end{equation}
The claim follows directly from (\ref{berniso_e5}) and Claim \ref{bern_iso_claim2}. 
\end{proof}

In the previous proof we made use of the following estimate of Fradelizi, stated without proof in \cite[Theorem 4]{fradelizi}. 
We include
a proof for the reader's convenience (see also \cite[Theorem 2.2.2]{brazitikos_etal}).
\begin{lemma}\label{fradelizi_lemma}
For a convex function $\phi: \R^n\to \R\cup\{\infty\}$, 
\begin{equation*}
    \inf_{x\in \R^n}\phi(x)\geq \phi(b(\phi))-n.
\end{equation*}
\end{lemma}
\begin{proof}
To begin with, it is enough to consider $\phi$ to be smooth, strictly convex, and bounded from below by $C|x|^2$ for large $|x|$. That is because
for a smooth, non-negative, compactly supported molifier $\chi: \R^n\to [0,\infty)$ we know that 
\begin{equation*}
    \phi_\e(x)\defeq \frac{1}{\e^n}\int_{\R^n} \phi(x-y)\chi(y/\e)\dif y
\end{equation*}
is smooth, convex and decreases to $\phi$ as $\e\to 0$. Let, 
\begin{equation*}
    \phi_{j,\e}(x)\defeq \phi_{\e}(x)+ \frac{1}{j}\frac{|x|^2}{2},
\end{equation*}
smooth, convex functions that decrease to $\phi$ as $\e\to 0$ and $j\to \infty$. In addition, $\phi_{j,\e}(x)\geq C|x|^2$ for large enough $|x|$, since $\phi_\e$ can be estimated by a linear term due to convexity, that is, $\phi_\e(x)\geq \phi_\e(0)+ \langle \nabla \phi_\e(0), x\rangle$. By monotone convergence, $b(\phi_{j,\e})\to b(\phi)$ as $\e\to 0$ and $j\to \infty$.
By convexity, $\phi(x)\geq \phi(b(\phi))+ \langle \partial\phi(b(\phi)), x- b(\phi)\rangle$ for all $x$, so if the claim holds for $\phi_{j,\e}$, then
\begin{equation*}
\begin{aligned}
    \phi_{j,\e}(y)&\geq \phi_{j,\e}(b(\phi_{j,\e}))-n\\
    &\geq \phi(b(\phi_{j,\e}))-n\\
    &\geq \phi(b(\phi))+ \langle \partial \phi(b(\phi)), b(\phi_{j,\e})-b(\phi)\rangle-n,
\end{aligned}
\end{equation*}
because $\phi_{j,\e}\geq \phi$.
Taking $j\to \infty$ and $\e\to 0$ yields $\phi(y)\geq \phi(b(\phi))-n$.

 For $\phi$ smooth, strictly convex with $\phi(x)\geq C|x|^2,$ for large $|x|$,
    by Jensen's inequality, 
    \begin{equation}\label{frad_eq1}
        \phi(b(\phi))= \phi\left( \int_{\R^n} x e^{-\phi(x)}\frac{\dif x}{V(\phi)}\right)\leq \int_{\R^n} \phi(x) e^{-\phi(x)}\frac{\dif x}{V(\phi)}. 
    \end{equation}
By convexity, for all $x,y\in \R^n$, 
$
    \phi(y)\geq \phi(x)+ \langle \nabla \phi(x), y-x\rangle, 
$
thus, integrating with respect $e^{-\phi(x)}\frac{\dif x}{V(\phi)}$,
\begin{equation}\label{frad_eq2}
\begin{aligned}
    \phi(y)&\geq \int_{\R^n}\phi(x) e^{-\phi(x)}\frac{\dif x}{V(\phi)}+ \int_{\R^n} \langle \nabla\phi(x), y-x\rangle e^{-\phi(x)}\frac{\dif x}{V(\phi)}\\
    &= \int_{\R^n}\phi(x) e^{-\phi(x)}\frac{\dif x}{V(\phi)}+ \sum_{i=1}^n \int_{\R^n} \frac{\partial \phi}{\partial x_i}(x) (y_i- x_i) e^{-\phi(x)}\frac{\dif x}{V(\phi)}\\
    &=\int_{\R^n}\phi(x) e^{-\phi(x)}\frac{\dif x}{V(\phi)}- \sum_{i=1}^n \int_{\R^n} \frac{\partial}{\partial x_i}(e^{-\phi(x)}) (y_i- x_i) \frac{\dif x}{V(\phi)}\\
    &= \int_{\R^n}\phi(x) e^{-\phi(x)}\frac{\dif x}{V(\phi)}- \sum_{i=1}^n \int_{\R^n}  e^{-\phi(x)}\frac{\dif x}{V(\phi)}\\
    &= \int_{\R^n}\phi(x) e^{-\phi(x)}\frac{\dif x}{V(\phi)}- n, 
\end{aligned}
\end{equation}
because, by integration by parts,
\begin{equation*}
    \int_{\R} \frac{\partial}{\partial x_i} (e^{-\phi(x)}) (y_i- x_i)\dif x= 0- \int_{\R} e^{-\phi(x)}\dif x,
\end{equation*}
since 
\begin{equation*}
   \lim_{|x_i|\to \infty} e^{-\phi(x)} |y_i-x_i|\leq \lim_{x_i\to \infty} e^{-C|x|^2} |y-x|=0.
\end{equation*}
By (\ref{frad_eq1}) and (\ref{frad_eq2}), $\phi(b(\phi))\leq \phi(y)+n$, for all $y\in \R^n$, from which the claim follows. 
\end{proof}

\subsubsection{A bound on \texorpdfstring{$\int_{\R^n} h_{1,K}\dif\nu_{p,K}$}{A bound on the integral of h₁,ₖ}}
\begin{lemma}\label{berniso_lem1}
Let $p>0$. For a convex body $K\subset \R^n$,
    \begin{equation*}
        \int_{\R^n} h_{1,K}(y)\dif \nu_{p,K}(y)\leq \frac{n}{p}. 
    \end{equation*}
\end{lemma}

\begin{proof}[Proof of Lemma \ref{berniso_lem1}]
    By Lemma \ref{list} (v),
    $h_{p,K}$ increases to $h_K$ with $p$. Therefore, by Lemma \ref{list} (i),
    \begin{equation*}
    \begin{aligned}
        F(p)&\defeq \log \int_{\R^n} e^{-h_{p,K}(y)}\dif y= \log\int_{\R^n} e^{-\frac1p h_{1,K}(py)}\dif y\\
        &= \log\int_{\R^n} e^{-\frac1p h_{1,K}(y)}\frac{\dif y}{p^n}= \log\int_{\R^n}e^{-\frac1p h_{1,K}(y)}\dif y- n\log p,
    \end{aligned}
    \end{equation*}
    is decreasing with $p$, and hence, its derivative must be non-positive, 
    \begin{equation*}
        \begin{aligned}
            0\geq \frac{\dif F}{\dif p}= \frac{\int_{\R^n} e^{-\frac1p h_{1,K}(y)} h_{1,K}(y) \dif y}{p^2 \int_{\R^n}e^{-\frac1p h_{1,K}(y)}\dif y}- \frac{n}{p}= \frac{1}{p^2}\int_{\R^n}  h_{1,K}(y)\dif \nu_{\frac1p}(y)- \frac{n}{p},  
        \end{aligned}
    \end{equation*}
    and the lemma follows. 
\end{proof}

\subsubsection{Proof of Theorem \ref{bern_iso_prop}}
\begin{claim}\label{bern_iso_claim1}
    Let $p>0$. For a convex body $K\subset \R^n$ with $0\in\mathrm{int}\,K$ and $|K|=1$, 
    \begin{equation*}
        \int_{\R^n} e^{-ph_{1,K}(y)}\dif y\geq \frac{\M(K)}{p^n}.
    \end{equation*}
\end{claim}
\begin{proof}
    Since $h_{p,K}\leq h_K$, by homogeneity of $h_K$, 
    \begin{equation*}
        \begin{aligned}
            \int_{\R^n} e^{-p h_{1,K}(y)}\dif y&\geq \int_{\R^n} e^{-p h_K(y)}\dif y= \int_{\R^n} e^{-h_{K}(py)}\dif y \\
            &= \int_{\R^n} e^{-h_K(v)}\frac{\dif v}{p^n}= \frac{n! |K^\circ|}{p^n}= \frac{\M(K)}{p^n}, 
        \end{aligned}
    \end{equation*}
    because $|K|=1$ thus $\M(K)\defeq n!|K||K^\circ|= n!|K^\circ|$. 
\end{proof}

\begin{proof}[Proof of Theorem \ref{bern_iso_prop}]
    By assumption, \eqref{B} holds.
    Thus (\ref{berniso_e2}) applies for probability measures with barycenter at the origin. 
    
    \medskip
    \noindent 
    (i) In order to apply the estimate \eqref{berniso_e2}, it is necessary to have a measure with barycenter at the origin. By Corollary \ref{santalo_point_corollary}, we may translate $K$ so that $b(\nu_{p,K})= b(h_{{1/p},K})=0$. By Corollary \ref{calCaffineInv}, this does not affect $\calC(K)$.
By (\ref{berniso_e2}) and Lemmas \ref{berniso_lem1} and \ref{berniso_lem2}(i), 
\begin{equation*}
    u_{B,K}(0)\leq \log\left( |K|^2 \frac{\M_{1/(2p)}(K)}{\M_{1/p}(K)^2} \frac{p^n}{2^n}\right)+ \frac{Bn}{p}
\end{equation*}
As a result, by (\ref{berniso_e1}),
    \begin{equation*}
        \begin{aligned}
            \calC(K)\geq \frac{\M_{1/p}(K)^2}{\M_{1/(2p)}(K)}\frac{2^n}{p^n} e^{-\frac{Bn}{p}}.
        \end{aligned}
    \end{equation*}
Choosing $p=B$, 
\begin{equation*}
\calC(K)\geq \frac{\M_{1/B}(K)^2}{\M_{1/(2B)}(K)}\frac{2^n}{B^n} e^{-n}.
\end{equation*}

\medskip    
\noindent 
(ii) Similarly, to apply \eqref{berniso_e2} we need a measure with barycenter at the origin.
By Corollary \ref{santalo_point_corollary}, we may translate $K$ so that $b(\nu_{p, K})= b(h_{1/p,K})=0$. Also, rescale so that $|K|=1$. By Corollary \ref{calCaffineInv} this does not affect $\calC(K)$. 
By (\ref{berniso_e2}), Lemma \ref{berniso_lem1} and Lemma \ref{berniso_lem2} (ii), 
\begin{equation*}
   u_{B,K}(0)\leq \log\left( \frac{e^n}{ \int_{\R^n} e^{-p h_{1,K}(y)}\dif y}\right)+ \frac{Bn}{p}\leq \log\left( \frac{e^n p^n}{\M(K)^n}\right)+ \frac{Bn}{p},
\end{equation*}
where we used Claim \ref{bern_iso_claim1} for the last inequality. 
As a result, by (\ref{berniso_e1}), since $|K|=1$,
\begin{equation}\label{iso_tool_eq2}
    \calC(K)= e^{-u_{B,K}(0)}\geq \frac{\M(K)}{e^np^n} e^{-\frac{Bn}{p}}.
\end{equation}
We can now optimize over all $p$ on the right hand side. Setting
$$
f(p):=
p^n e^{(1+\frac{B}{p})n}
$$
gives
$$
f'(p)=e^{n+nB/p}\cdot\Big[
np^{n-1}
-
p^n\cdot
nB/p^2
\Big]
=
ne^{n+nB/p}p^{n-2}(p
-
B
),
$$
and the second derivative
gives
$$
\begin{aligned}
f''(p)&=
ne^{n+nB/p}
\big[
-nBp^{-2}
p^{n-2}(p-B)
+
(n-2)p^{n-3}(p-B)
+
p^{n-2}
\big]\cr
&=
ne^{n+nB/p}
p^{n-4}
\big[
-nB(p-B)
+
(n-2)p(p-B)
+
p^{2}
\big]
,
\end{aligned}
$$
so $f''(B)=
ne^{2n}
B^{n-2}>0$ as long as $B>0$.
This 
confirms $p=B$
is a minimum. Thus, choosing $p= B$ in (\ref{iso_tool_eq2}), 
\begin{equation*}
    \calC(K)\geq \frac{\M(K)}{e^{2n}B^n}.
\end{equation*}

\bigskip
\noindent
(iii) Since $K$ is symmetric, $b(K)= b(\nu_{p,K})=0$. Rescale $K$ so that $|K|=1$. $\calC(K)$ remains invariant under rescalling by Corollary \ref{calCaffineInv}. By (\ref{berniso_e2}), Lemmas \ref{berniso_lem1} and \ref{berniso_lem2}(iii), and Claim \ref{bern_iso_claim1},
\begin{equation*}
   u_{B,K}(0)\leq \log\left( \frac{1}{ \int_{\R^n} e^{-p h_{1,K}(y)}\dif y}\right)+ \frac{Bn}{p}\leq \log\left( \frac{p^n}{\M(K)}\right)+ \frac{Bn}{p},
\end{equation*}
As a result, by (\ref{berniso_e1}), since $|K|=1$,
\begin{equation*}
    \calC(K)= e^{-u_{B,K}(0)}\geq \frac{\M(K)}{p^n} e^{-\frac{Bn}{p}}.
\end{equation*}
Thus, choosing $p=B$, 
\begin{equation*}
    \calC(K)\geq \frac{\M(K)}{e^n B^n},
\end{equation*}
concluding the proof of Theorem \ref{bern_iso_prop}
\end{proof}

\begin{remark}\label{h1Ksuffices}
    It is enough to formulate Theorem \ref{bern_iso_prop} in terms of $h_{1,K}$: By Lemma \ref{list} (i), $h_{p,K}(y)= \frac1p h_{1,K}(py)$, therefore $\nabla^2 h_{p,K}(y)= p\nabla^2 h_{1,K}(py)$. As a result, 
    \begin{equation*}
    \log\det\nabla^2  h_{p,K}(y)+ pB h_{p,K}(y)= \log\det\nabla^2  h_{1,K}(py)+ B h_{1,K}(py)+ n\log p. 
    \end{equation*}
    Thus, $\log\det\nabla^2 h_{p,K}(y)+ pB h_{p,K}(y)$ is convex if and only if $\log\det\nabla^2  h_{1,K}(y)+ Bh_{1,K}(y)$ is. 
\end{remark}

\subsection{A sub-optimal bound}
\label{S54}

We prove Theorem \ref{bern_iso_theorem}, i.e., we show that 
$\log\det\nabla^2 h_{p,K}+ p(n+1) h_{p,K}$ is convex. Corollary \ref{suboptimal_bound} then follows from Theorem \ref{bern_iso_prop} (ii). 

\begin{proof}[Proof of Corollary \ref{suboptimal_bound}.]
By Theorem \ref{bern_iso_prop} (ii) and Theorem \ref{bern_iso_theorem}, 
\begin{equation*}
    \calC(K)\geq \frac{\M(K)}{e^{2n}(n+1)^n}\geq \left( \frac{\pi}{2e^2}\right)^n \frac{1}{(n+1)^n}.
\end{equation*}
By tensorization, $\calC(K)\geq (\frac{\pi}{2e^2 n})^n$.
\end{proof}

\begin{proof}[Proof of Theorem \ref{bern_iso_theorem}.]
    Recall by the proof of Lemma \ref{MA_isotropic},
    \begin{equation*}
        \nabla^2 h_{p,K}(y)= p\left( \int_K x_i x_j \dif\nu^y(x)- \int_K x_i\dif\nu^y(x)\int_K x_j \dif\nu^y(x)\right)_{i,j}, 
    \end{equation*}
    where
    \begin{equation*}
        \dif\nu^y(x)\defeq \frac{e^{p\langle x,y\rangle}}{\int_K  e^{p\langle x,y\rangle}\frac{\dif x}{|K|}} \frac{ \mathbf{1}^\infty_K(x)\dif x}{|K|}, 
    \end{equation*}
    a probability measure that depends on $y$. Consider the $(n+1)\times (n+1)$ matrix, 
    \begin{equation*}
        M\defeq \begin{bmatrix}
        1 &
\begin{matrix}
    \int_K x_1\dif\nu^y(x) & \ldots & \int_K x_n \dif\nu^y(x)
\end{matrix}\\
\begin{matrix}
    \int_K x_1\dif\nu^y(x) \\ \vdots\\ \int_K x_n \dif\nu^y(x)
\end{matrix} 
& \left[ \int_{K}x_ix_j \dif\nu^y(x)\right]_{i,j=1}^n
        \end{bmatrix}.
    \end{equation*} 
    By row reduction, 
    \begin{equation*}
        \det M= p^{-n}\det\nabla^2 h_{p,K}. 
    \end{equation*}
Note that for $i,j\in\{0, 1,\ldots, n\}$, $M_{ij}= \langle x_i, x_j\rangle_{L^2(\dif\nu^y)}$, where $x_0=1$. 
For 
\begin{equation*}
    \Delta(x^{(0)},\ldots, x^{(n)})\defeq \det\begin{bmatrix}
    1 & x_1^{(0)} & \ldots & x_n^{(0)} \\
    \vdots & \vdots & \ddots & \vdots \\
    1 & x_1^{(n)} & \ldots & x_n^{(n)}
    \end{bmatrix},
\end{equation*}
by Andréief's formula \cite[(1.7)]{forrester}, 
\begin{equation*}
\begin{aligned}
    \det M&= \frac{1}{(n+1)!}\int_{K^{n+1}} \left( \det\begin{bmatrix}
    1 & x_1^{(0)} & \ldots & x_n^{(0)} \\
    \vdots & \vdots & \ddots & \vdots \\
    1 & x_1^{(n)} & \ldots & x_n^{(n)}
    \end{bmatrix}\right)^2 \dif\nu^y(x^{(0)})\ldots\dif\nu^y(x^{(n)}) \\
    &= \frac{1}{(n+1)!} \int_{K^{n+1}} |\Delta|^2 \frac{e^{p\langle x^{(0)}, y\rangle}}{\int_K e^{p\langle x^{(0)}, y\rangle}\dif x^{(0)}}\dif\lambda (x^{(0)})\ldots \frac{e^{p\langle x^{(n)}, y\rangle}}{\int_K e^{p\langle x^{(n)}, y\rangle}\dif x^{(n)}}\dif\lambda (x^{(n)}) \\
    &= \frac{1}{(n+1)!} \frac{1}{\left( \int_K e^{p\langle x,y\rangle} \dif x\right)^{n+1}} \int_{K^{n+1}}|\Delta|^2 e^{p\langle \sum_{j=0}^n x^{(j)},y\rangle}\dif x^{(0)}\ldots\dif x^{(n)}. 
\end{aligned}
\end{equation*}
Therefore, 
\begin{equation*}
\begin{aligned}
    \log\det\nabla^2 h_{p,K}(y)&= n\log p+ \log\det M\\
    &= n\log p- \log(n+1)!- (n+1) \log \int_K e^{p\langle x,y\rangle}\dif x+ \log\phi(y) \\
    &= n\log p- \log(n+1)!- p(n+1) h_{p,K}(y)+ \log\phi(y), 
\end{aligned}
\end{equation*}
where
\begin{equation*}
    \phi(y)\defeq \int_{K^{n+1}}|\Delta|^2 e^{p\langle \sum_{j=0}^n x^{(j)},y\rangle}\dif x^{(0)}\ldots\dif x^{(n)}. 
\end{equation*}
Since $\log\phi$ is convex (Lemma \ref{logconvLemma} below), and
\begin{equation*}
    \log\det\nabla^2 h_{p,K}(y)+ p(n+1)h_{p,K}(y)= n\log p- \log(n+1)!+ \log\phi(y), 
\end{equation*}
the claim follows. 
\end{proof}

\begin{lemma}\label{logconvLemma}
    Let $K\subset \R^n$ be a convex body, $m\in\mathbb{N}$, and $f: \R^{nm}\to [0,+\infty)$ measurable function. Then, 
    \begin{equation*}
        \phi(y)\defeq \log\int_{K^m} f(x_1, \ldots, x_m) e^{\langle x_1+ \ldots+ x_m, y\rangle}\dif x, \quad y\in \R^n, 
    \end{equation*}
    is convex. 
\end{lemma}
\begin{proof}
Write $x=(x_1, \ldots, x_m)\in \R^{nm}$ and let $\lambda\in(0,1), y_1, y_2\in \R^n$. 
Since 
$$f(x) e^{\langle x_1+\ldots+ x_m, (1-\lambda)y_1+\lambda y_2\rangle}= \left(f(x) e^{\langle x_1+\ldots+ x_m, y_1\rangle} \right)^{1-\lambda}\left(f(x) e^{\langle x_1+\ldots+x_m, y_2\rangle}\right)^{\lambda},
$$ by H\"older's inequality for $p= \frac{1}{1-\lambda}$ and $q= \frac{1}{\lambda}$,
\begin{equation*}
\begin{aligned}
    \int_{K^m}f(x) e^{\langle x_1+\ldots+ x_m, (1-\lambda)y_1+\lambda y_2\rangle}\dif x\leq &\left( \int_{K^m}f(x) e^{\langle x_1+\ldots+ x_m, y_1\rangle}\dif x\right)^{1-\lambda} 
    \\&\left( \int_{K^m}f(x) e^{\langle x_1+\ldots+ x_m, y_2\rangle}\dif x\right)^{\lambda}.
\end{aligned}
\end{equation*}
Taking logarithms yields $\phi((1-\lambda)y_1+\lambda y_2)\leq (1-\lambda)\phi(y_1)+ \lambda \phi(y_2)$. 
\end{proof}

\subsection{\texorpdfstring{More general probability measures and sharpness of 
$B=n+1$}{More general probability measures and sharpness of B=n+1}} \label{LpSupportMeasureS}

As just discussed, Theorem \ref{bern_iso_theorem} falls
short of proving the slicing conjecture because the best constant
$B$ we currently obtain is $n+1$. It is interesting to note that
while in the setting of the uniform measure on $K$ this constant
could potentially be improved, many of the results in this section extend to general probability measures and then the constant $n+1$ is in fact {\it optimal}. The purpose
of this subsection is to spell this out.

Throughout this section the only properties of the measure
$$
    \mathbf{1}_{K}^\infty \frac{\dif x}{|K|}
$$
used to obtain the estimates in Lemmas \ref{berniso_lem2} and \ref{berniso_lem1} were that it is a probability measure that is supported on $K$. 
As a result, it may be
replaced by any probability measure
\begin{equation*}
    \mu 
\end{equation*}
that is supported on $K$, i.e., for any measurable $A\subset \R^n\setminus K$, $\mu(A)=0$, so that, in addition, $\mathrm{co}\,\mathrm{supp}(\mu)= K$. For example, \eqref{berniso_e2} was already obtained for any probability measure with barycenter at the origin. 
For a convex body $K\subset \R^n$ and a probability measure $\mu$ whose convex hull of its support is $K$, let
\begin{equation*}
    h_{p,\mu}(y)\defeq \frac1p \log\int_{K} e^{p\langle x,y\rangle}\dif \mu(x). 
\end{equation*}
As in Lemma \ref{MA_isotropic}, 
\begin{equation*}
    \frac1p \nabla^2 h_{p,\mu}(0)= \mathrm{Cov}(\mu)\defeq \left[ \int_{K}x_ix_j\dif\mu(x)- \int_{K}x_i\dif\mu(x)\int_{K}x_j\dif\mu(x)\right]_{i,j=1}^n. 
\end{equation*}
For $p>0$, let
\begin{equation}\label{numu}
    \nu_{p,\mu}\defeq \frac{e^{-ph_{1,\mu}(y)}\dif y}{\int_{\R^n}e^{-ph_{1,\mu}(y)}\dif y}. 
\end{equation}
Then, Claim \ref{bern_iso_claim2}, Lemmas \ref{berniso_lem2} and \ref{berniso_lem1} generalize. 
\begin{lemma}
Let $p>0$. For a convex body $K\subset \R^n$, $\mu$ a probability measure with $\mathrm{co}\,\mathrm{supp}(\mu)= K$, and $\nu_{p,\mu}$ as in \eqref{numu}, 

\noindent 
(i) \begin{equation*}
    \int_{\R^n}\log\det\nabla^2 h_{1,\mu}(y)\dif\nu_{p,\mu}(y)\leq\log\left( |K| \frac{\int_{\R^n}e^{-2p h_{1,\mu}(y)}\dif y}{\left(\int_{\R^n}e^{-ph_{1,\mu}(y)}\dif y\right)^2}\right). 
\end{equation*}

\noindent
(ii) If $b(\nu_{p,\mu})=0$, then
\begin{equation*}
    \int_{\R^n}\log\det\nabla^2 h_{1,\mu}(y)\dif\nu_{p,\mu}(y)\leq \log\left(\frac{|K| e^n}{ \int_{\R^n} e^{-p h_{1,\mu}(y)}\dif y} \right).
\end{equation*}

\noindent
(iii) If $b(\mu)=0$, then
\begin{equation*}
    \int_{\R^n}\log\det\nabla^2 h_{1,\mu}(y)\dif\nu_{p,\mu}(y)\leq \log\left(\frac{|K|}{ \int_{\R^n} e^{-p h_{1,\mu}(y)}\dif y} \right).
\end{equation*}
\end{lemma}
\begin{lemma}
Let $p>0$.
For a convex body $K\subset \R^n$ and a probability measure $\mu$ with $\mathrm{co}\,\mathrm{supp}(\mu)=K$,
    \begin{equation*}
        \int_{\R^n} h_{1,\mu}(y)\dif \nu_{p,\mu}(y)\leq \frac{n}{p}. 
    \end{equation*}
\end{lemma}

Theorem \ref{bern_iso_theorem} also generalizes. 
\begin{theorem}\label{MeasureProp}
    Let $p>0$. For a probability measure $\mu$ on $\R^n$ such that $\overline{supp(\mu)}$ is a convex body, the function
    \begin{equation*}
        \log\det\nabla^2 h_{p,\mu}(y)+ p(n+1) h_{p,\mu}(y)
    \end{equation*}
    is convex. 
\end{theorem}

In fact, Theorem \ref{MeasureProp} is sharp: the
next example shows $B=n+1$ cannot be improved.
\begin{example}
\label{simplexExample}
Consider 
\begin{equation*}
    \mu\defeq \frac{\delta_0+ \delta_{e_1}+ \ldots+ \delta_{e_n}}{n+1}
\end{equation*}
the probability measure on the standard simplex $\Delta_n$ that assigns mass $\frac{1}{n+1}$ to each vertex. Then, 
\begin{equation*}
    \log\det\nabla^2 h_{p,\mu}(y)+ p B h_{p,\mu}(y)
\end{equation*}
is convex if and only if $B\geq n+1$. To see this, compute 
\begin{equation*}
    h_{p,\mu}(y)= \frac1p \log\int_{\Delta_n} e^{p\langle x,y\rangle}\dif \mu(x)= \frac{1}{p} \log\frac{1+ e^{py_1}+ \ldots+ e^{py_n}}{n+1}. 
\end{equation*}
For the gradient, by the chain rule, 
\begin{equation}\label{GradCoordinate}
    \frac{\partial h_{p,\mu}}{\partial y_i}(y)= \frac1p \frac{n+1}{1+ e^{py_1}+ \ldots+ e^{py_n}} \frac{\partial}{\partial y_i}\left(\frac{1+ e^{py_1}+ \ldots+ e^{py_n}}{n+1}\right)= \frac{e^{py_i}}{1+ e^{py_1}+ \ldots+ e^{py_n}},
\end{equation}
thus
\begin{equation*}
    \nabla h_{p,\mu}(y)= \frac{(e^{py_1}, \ldots, e^{p y_n})}{1+ e^{py_1}+\ldots+ e^{py_n}} .
\end{equation*}
For the Hessian, by the quotient rule on \eqref{GradCoordinate},
\begin{equation*}
\begin{aligned}
    \frac{\partial^2 h_{p,\mu}}{\partial y_i\partial y_j}(y)&= \frac{p e^{p y_i}\delta_{ij}(1+e^{py_1}+ \ldots+ e^{py_n})- e^{py_i}p e^{py_j}}{(1+ e^{py_1}+\ldots + e^{py_n})^2}\\
    &= \frac{p}{1+ e^{py_1}+\ldots+ e^{p y_n}} \left( \delta_{ij} e^{py_i}- \frac{e^{p(y_i+y_j)}}{1+ e^{py_1}+\ldots+ e^{py_n}}\right), 
\end{aligned}
\end{equation*}
thus
\begin{equation*}
    \begin{aligned}
        \nabla^2 h_{p,\mu}(y)&= \frac{p}{1+ e^{py_1}+ \ldots+ e^{py_n}}\left[ \delta_{ij} e^{py_i}- \frac{e^{p(y_i+y_j)}}{1+ e^{py_1}+ \ldots + e^{py_n}}\right]_{i,j=1}^n \\
        &= \frac{p}{1+ e^{py_1}+ \ldots+ e^{py_n}} (D- a^Ta), 
    \end{aligned}
\end{equation*}
where $D= \mathrm{diag}(e^{py_1}, \ldots, e^{py_n})$ and 
$
    a= (1+ e^{py_1}+\ldots+ e^{py_n})^{-1/2} (e^{py_1}, \ldots, e^{py_n}).
$
Therefore,
\begin{equation*}
    \begin{aligned}
        \det\nabla^2 h_{p,\mu}(y)&= \frac{p^n}{(1+ e^{y_1}+ \ldots + e^{py_n})^n} \det(D- a^Ta)\\
        &= \frac{p^n}{(1+ e^{y_1}+ \ldots + e^{py_n})^n} \det D (1- \langle D^{-1}a, a\rangle)\\
        &=  \frac{p^n}{(1+ e^{y_1}+ \ldots + e^{py_n})^n} e^{py_1+\ldots+ py_n} \left( 1-\frac{e^{py_1}+ \ldots+ e^{py_n}}{1+e^{py_1}+\ldots+ e^{py_n}}\right)\\
        &=  \frac{p^n}{(1+ e^{y_1}+ \ldots + e^{py_n})^{n+1}} e^{py_1+ \ldots+ py_n}.
    \end{aligned}
\end{equation*}
As a result, 
\begin{equation*}
    \begin{aligned}
        \log\det\nabla^2 h_{p,\mu}(y)+ pB h_{p,\mu}(y)&=n\log(p)+ py_1+ \ldots+ py_n\\
        &\hspace{.5cm}- (n+1)\log(1+ e^{py_1}+\ldots+ e^{py_n}) \\
        &\hspace{.5cm} + B\log\frac{1+e^{py_1}+ \ldots+ e^{py_n}}{n+1} \\
        &= (B-n-1) \log (1+ e^{py_1}+ \ldots+ e^{py_n})\\ 
        &\hspace{.5cm}+ py_1+ \ldots+ py_n+ n\log p- B\log(n+1), 
\end{aligned}
\end{equation*}
which is convex if and only if $B\geq n+1$ (because $\log(1+ e^{py_1}+\ldots+ e^{py_n})$ is convex).

\noindent
When $B=n+1$ and $p=1$ we get
$$
\log\det\nabla^2  h_{1,\mu}+(n+1) h_{1,\mu}=y_1+...y_n -(n+1)\log(n+1).,
$$
so that $h_{1,\mu}$ solves the Monge--Amp\`ere equation
$$
\det\nabla^2 h_{1,\mu}(y)=\frac{1}{(n+1)^{n+1}} e^{-(n+1) h_{1,\mu}}e^{y_1+...y_n}.
$$
 From here we can read off that 
$$
\det\nabla^2 h_{1,\mu}(0)=\frac{1}{(n+1)^{n+1}}.
$$
We next look at a generalized isotropic constant, by defining
\beq
\mathcal{C}(\mu):=\frac{|K|^2}{\det\nabla^2 h_{1,\mu}(0)}.
\eeq
From the previous equation we then get, remembering that the volume of the unit simplex is $1/n!$, that
$$
\mathcal{C}(\mu)=\frac{(n+1)^{n+1}}{(n!)^2}.
$$
The right hand side here is of the order of magnitude $c^n n^{-n}$, so we see that the `suboptimal' bound of Corollary 1.12  {\it  is optimal in this generality.}
\end{example}

Interpreted benevolently, Example \ref{simplexExample} means that our method is optimal in the sense that  the best possible choice of $B$ gives the correct estimate for $\mathcal{C}(\mu)$. The natural question then arises, for which measures $\mu$ the constant $B$ can be taken smaller so that we as a consequence get a better estimate of $\mathcal{C}(\mu)$. One simple case when this is so is when $\mu$ is divisible, in the sense that we can write
$$
\mu=\nu\star\nu\star\cdots\star\nu= \nu^{k\star}
$$
as the $k$-fold convolution of another probability measure $\nu$ with itself. In that case,
$$
h_{1,\mu}=k h_{1,\nu}.
$$
Applying Theorem 6.20 to $h_{1,\nu}$ we then get that
$$
\log \det\nabla^2 h_{1,\nu}+(n+1)h_{1,\nu}
$$
is convex, which implies that
$$
\log \det\nabla^2 h_{1,\mu}+\frac{n+1}{k} h_{1,\mu}
$$
is convex. This leads to the improved estimate
$$
\mathcal{C}(\mu)\geq c^n\Big(\frac{k}{n}\Big)^n.
$$
This is however not so impressive since the same conclusion can be drawn directly from $\mathcal{C}(\nu)\geq c^n/n^n$ if we note that the convex hull of the support of $\nu$ is $K/k$. This way we also see that it is not really necessary that $\mu$ can be written $\nu^{k\star}$; it is enough that $\mu=f\nu^{k\star}$ where $f$ is bounded.

\subsection{A complex geometric approach to Theorems \ref{bern_iso_theorem} and \ref{MeasureProp}}
\lb{complexgeom_subsec}

In this section we outline a different proof of Theorem \ref{bern_iso_theorem} (and
of its generalization, Theorem \ref{MeasureProp}) which is a little more conceptual, but presupposes a bit of complex geometry. It is based on a theorem by S. Kobayashi 
\cite[Theorem 4.4]{Kobayashi}. Kobayashi's theorem deals with $L^2$ spaces of holomorphic $(n,0)$-forms on complex manifolds, but his proof goes through in a much more general setting and applies in particular to the setting we will now describe.

Let $\mu$ be a compactly supported probability measure on $\R^n$. Let
$$
H_\mu:=\Big\{ \tilde f(z):=\int_{\R^n} e^{\langle z, t\rangle/2} f(t) d\mu(t)\,:\, f\in L^2(\mu)\Big\}.
$$
$H_\mu$ is a space of entire functions on $\C^n$ and we give it an inner product 
\begin{equation}\label{HmuInnerProduct}
\langle \tilde{f}, \tilde{g}\rangle\defeq \int f(t)\overline{g(t)}\dif \mu(t),
\end{equation}
making $H_\mu$ a Hilbert space, isomorphic to $L^2(\mu)$. 

We require that $\mu$ is not supported in any proper linear subspace of  $\R^n$. This implies that for any $a\in \R^n$ there is a function $f$ such that
$$
\int fd\mu=0, \,\, \text{and}\,\, \int \langle a, t\rangle f(t) d\mu(t)\neq 0.
$$
Indeed, if this were not the case, any function orthogonal to 1 in $L^2(\mu)$ would also be orthogonal to $\langle a, t\rangle$, which would imply that $\langle a, t\rangle= c$ on the support on $\mu$, contrary to assumption. In terms of functions in $H_\mu$, this says that there is a function $\tilde f$ which vanishes at the origin, with $\sum a_j\partial_j f$ not vanishing there. Then, replacing $f$ by $e^{\langle z_0, t\rangle/2}f(t)$ we see that the same thing goes for any point $z_0$ in $\C^n$. This means that the conditions A.1 and A.2 in Kobayashi's paper \cite[pp. 271--2]{Kobayashi} are satisfied (we will see the relevance of this shortly). Kobayashi's 
condition~A.1 says that for any point in $\C^n$ there is a function in $H_\mu$ that does not vanish there -- this is trivial in our case. Indeed, for $z_0\in \C^n$, since $\mu$ is compactly supported, $e^{-\langle z_0,t\rangle}\in L^2(\mu)$, and 
\begin{equation*}
    \int e^{-\langle z_0,t\rangle} e^{\langle z_0,t\rangle}\dif\mu(t)= \int \dif\mu(t)=1, 
\end{equation*}
because $\mu$ is a probability measure.

The (diagonal) Bergman kernel for $H_\mu$ is defined as
$$
B_\mu(z):=\sup_{\{\tilde f\,:\, \|\tilde f\|=1\}} |\tilde f(z)|^2.
$$
By condition A.1, the Bergman kernel does not vanish anywhere. It follows directly from the definitions that for
$$
\mathcal{K}_\mu(z,w)\defeq\int e^{\langle \frac{z+\overline{w}}{2}, t\rangle} d\mu(t),  
$$
and $\tilde{f}(z)= \int e^{\langle z,t\rangle/2} f(t)\dif\mu(t)\in H_\mu$, by \eqref{HmuInnerProduct}, 
\begin{equation*}
    \langle \tilde{f}, \mathcal{K}_\mu(\cdot, w)\rangle= \int f(t) \overline{e^{\langle \overline{w},t\rangle/2}}\dif\mu(t)= \int f(t) e^{\langle w,t\rangle/2}\dif\mu(t)= \tilde{f}(w), 
\end{equation*}
i.e., $\mathcal{K}_\mu$ enjoys a reproducing property, in addition to being holomorphic in the first variable and anti-holomorphic in the second. These three properties characterize Bergman kernels \cite[\S 3.2]{MR}, thus $\mathcal{K}_\mu$ is the Bergman kernel of $H_\mu$.
Therefore, on the diagonal, if $z=x+iy$,
$$
B_\mu(z)=\mathcal{K}_\mu(z,z)=\int e^{\langle x, t\rangle} d\mu(t),  
$$
i.e., coming back full circle to the ideas in \S\ref{BergmanSubSec},
\beq
\lb{logBmuEq}
\log B_\mu= h_{1,\mu}.
\eeq

The Bergman metric associated to $H_\mu$ is the  K\"ahler metric on $\C^n$ defined by
$$
g_{j\bar k}:=\frac{\partial^2}{\partial z_j \partial \bar z_k}\log B_\mu.
$$
By \eqref{logBmuEq}, $\log B_\mu$ is convex in $x$, hence plurisubharmonic in $z$, and the matrix $g=[g_{j\bar k}]$ is positive semi-definite, and it is a standard fact (that we omit) that the condition A.2 is precisely what is needed to make sure it is strictly positive definite. (Alternatively, condition A.2 can verified by using \eqref{logBmuEq}, the computation of Lemma \ref{MA_isotropic},
and the Cauchy--Schwarz inequality to note that $h_{p,\mu}$ is strongly convex.)
Kobayashi's theorem says that the Ricci curvature $\mathrm{Ric}\, g$ of the Bergman metric is bounded from above by $(n+1)g$.

At this point we need to make use of a standard formula for the Ricci curvature, valid for any K\"ahler metric. Let 
$$
\Delta:= \det [g_{j\bar k}]
$$
be the density of the volume form of the metric $g$. Then the Ricci curvature 
form of $g$ is given by
$$
R_{j \bar k}= -\frac{\partial^2}{\partial z_j \partial \bar z_k}\log\Delta.
$$
Hence, Kobayashi's estimate
$$
[R_{j k}]\leq (n+1)[g_{j k}],
$$
translates to saying that
$$
\log\Delta +(n+1)\log B_\mu
$$
is plurisubharmonic. In our case, $B_\mu$ and $\log\Delta$ depend only on $x=\mathrm{Re}(z)$, so 
$$
\log\Delta +(n+1)\log B_\mu
$$ 
is actually a convex function of $x$. Moreover, $\log B_\mu=h_{1,\mu}$ and $\Delta=4^{-n}\det\nabla^2h_{1,\mu}$
(in the last equality we used the relation between the complex
Hessian and the real one on functions
depending only on the real part). Therefore
$$
\log \det\nabla^2h_{1,\mu}+(n+1) h_{1,\mu}
$$
is convex,  i.e., \eqref{B} holds
with $B=n+1$,
so we have proved Theorem \ref{MeasureProp}, and, in particular, also Theorem  \ref{bern_iso_theorem}.

\appendix
\section{A (near) norm associated to a convex function}\label{appendixA}
In this section we give proofs for Proposition \ref{norm_phi} \cite[Theorem 5]{ball2} and Theorem \ref{KBall_ineq} \cite[Theorem 4.10]{ball} (cf. \cite{busemann}, \cite[p. 90]{milman-pajor}). Let us start by using Theorem \ref{KBall_ineq} to prove Proposition \ref{norm_phi}.

\begin{proposition}\label{norm_phi}
    For a convex function $\phi: \R^n\to \R\cup\{\infty\}$, 
    \begin{equation}
    \lb{BallFnEq}
        x\mapsto \left( \int_{0}^\infty r^{n-1}e^{-\phi(rx)}\dif r\right)^{-\frac1n}
    \end{equation}
    is positively 1-homogeneous and sub-additive (it is also a norm if $\phi$ is, in addition, even), and 
    \begin{equation*}
        \frac1n \int_{\R^n} e^{-\phi(x)}\dif x= |\{x\in \R^n: \|x\|_{\phi}\leq 1\}|.
    \end{equation*}
\end{proposition}
\begin{proof}[Proof of Proposition \ref{norm_phi}]
\noindent
\textit{1-homogeneity.} Let $x\in \R^n$ and $\lambda>0$. By changing variables $\rho= \lambda r$, 
\begin{equation*}
    \begin{aligned}
        \|\lambda x\|_{\phi}&\defeq \left( \int_0^\infty r^{n-1} e^{-\phi(r\lambda x)}\dif r\right)^{-\frac1n} \\
        &= \left( \int_0^\infty \frac{\rho^{n-1}}{\lambda^{n-1}} e^{-\phi(\rho x)}\frac{\dif \rho}{\lambda}\right)^{-\frac1n}\\
        &= \lambda \left( \int_0^\infty \rho^{n-1} e^{-\phi(\rho x)}\dif \rho\right)^{-\frac1n}= \lambda \|x\|_\phi. 
    \end{aligned}
\end{equation*}
Positivity of $\lambda$ is used in the last step.

\smallskip
\noindent 
\textit{Subadditivity.} Let $x,y\in\R^n$ and 
$r,t,s>0$ with 
$\frac1r= \frac12\Big(\frac1t+ \frac1s\Big)$,
or equivalently,
\begin{equation}
\lb{rstEq}
 \frac r{2t}+ \frac r{2s}=1.
\end{equation}
By \eqref{rstEq} and convexity of $\phi$, 
\begin{equation}
\lb{phirxyEq}
    \phi(r(x+y))= \phi\left(\frac{r}{2t}2tx+ \frac{r}{2s}2sy\right)\leq 
    \frac{r}{2t} \phi(2tx)+ \frac{r}{2s}\phi(2sy)=
     \frac{s}{t+s} \phi(2tx)+ \frac{t}{t+s}\phi(2sy). 
\end{equation}
Set,
$$
H(r)\defeq e^{-\phi(r(x+y))},\q F(t)\defeq e^{-\phi(2tx)},\q G(s)\defeq e^{-\phi(2sy)}.
$$
By \eqref{phirxyEq}, $H(r)\geq F(t)^{\frac{s}{t+s}}G(s)^{\frac{t}{t+s}}$,
so by Theorem \ref{KBall_ineq} (with $q=n$), 
\begin{equation*}
    \begin{aligned}
        \|x+y\|_\phi&= \left( \int_0^\infty r^{n-1}e^{-H(r)}\dif r\right)^{-\frac1n}\\
        &\leq \frac12\left(\int_0^\infty r^{n-1}e^{-F(t)}\dif t \right)^{-\frac1n}+ \frac12\left(\int_0^\infty r^{n-1} e^{-G(s)}\dif s \right)^{-\frac1n}\\
        &= \frac12 \|2x\|_{\phi}+ \frac12\|2y\|_{\phi}= \|x\|_\phi+ \|y\|_{\phi}, 
    \end{aligned}
\end{equation*}
using the already established homogeneity of $\|\cdot\|_\phi$. 

\smallskip
\noindent
\textit{Volume equality.} 
By \eqref{normVolume}, 
\begin{equation}\label{phiVol1}
    \begin{aligned}
        |\{x\in\R^n: \|x\|_\phi\leq 1\}|
         &
        = \frac1n \int_{\partial B_2^n}\frac{\dif u}{\|u\|_\phi^n}
        \cr
        &=\frac1n\int_{\partial B_2^n} \int_0^\infty r^{n-1}e^{-\phi(ru)}\dif r\dif u. \end{aligned}
\end{equation}
Using polar coordinates this is
$\frac1n \int_{\R^n}e^{-\phi(x)}\dif x$.

\smallskip
\noindent
\textit{Norm.} Assuming in addition that $\phi$ is even, for $x\in \R^n$, 
\begin{equation*}
    \|-x\|_{\phi}= \left( \int_0^\infty r^{n-1} e^{-\phi(-rx)}\dif r\right)^{-\frac1n}= \left( \int_0^\infty r^{n-1} e^{-\phi(rx)} \dif r\right)^{-\frac1n}= \|x\|_\phi. 
\end{equation*}
Therefore, for $\lambda\in\R$, $\|\lambda x\|_\phi= \||\lambda|x\|_{\phi}= |\lambda| \|x\|_{\phi}$, making $\|\cdot\|_{\phi}$ into a norm. This concludes
the proof of  Proposition \ref{norm_phi}.
\end{proof}

Next, we turn to proving Theorem \ref{KBall_ineq}.
The proof involves three auxiliary lemmas.
To begin with, invert the variables; for $t,s,r >0$, let
\begin{equation*}
    u\defeq \frac1t, \quad v\defeq \frac{1}{\theta s}, \quad \text{ and } \quad w\defeq \frac{1}{r}, 
\end{equation*}
for some $\theta>0$ to be chosen later. 
In the inverted coordinates, the condition $\frac{2}{r}= \frac1t+ \frac1s$ becomes $w= \frac{u+\theta v}{2}$
.
Now, let
\begin{equation}\label{FGdefEq}
    A(u)\defeq F(u^{-1}) u^{-(q+1)}, \qquad \quad B(v)\defeq G(\theta^{-1}v^{-1}) v^{-(q+1)},
\end{equation}
and 
\begin{equation}\label{HdefEq}
    C(w)\defeq \left(\frac{\theta+1}{2}\right)^{q+1} H(w^{-1}) w^{-(q+1)}.
\end{equation}
The reason behind the multiplication by $\left( \frac{\theta+1}{2}\right)^{q+1}$ will become apparent in the next lemma that translates the \eqref{fghCondition}
relation between $F,G$ and $H$ to one between $A,B$ and $C$.
\begin{lemma}\label{BallLemma1}
    Let $F,G,H$ as in Theorem \ref{KBall_ineq}, and $\theta>0$. 
    For $A,B$ and $C$ as in \eqref{FGdefEq}--\eqref{HdefEq}, 
    \begin{equation*}
        C\left( \frac{u+\theta v}{2}\right)\geq A(u)^{\frac{u}{u+\theta v}} B(v)^{\frac{\theta v}{u+\theta v}},
        \q     \hbox{for all $u,v>0$}.
    \end{equation*}
\end{lemma}

A straightforward change of variables expresses the integrals of $F,G$, and $H$ in terms of integrals of $A,B$, and $C$:
\begin{lemma}\label{BallLemma2}
    Let $F,G,H$ as in Theorem \ref{KBall_ineq} and $\theta>0$. For $A, B$ and $C$ as in \eqref{FGdefEq} and \eqref{HdefEq}, 
    \begin{equation*}
        \begin{gathered}
            \int_0^\infty A(u)\dif u= \int_0^\infty t^{q-1} F(t)\dif t, \\
            \int_0^\infty B(v)\dif v= \theta^q \int_0^\infty s^{q-1} G(s)\dif s, \\
            \int_0^\infty C(w)\dif w= \left( \frac{\theta+1}{2}\right)^{q+1}\int_0^\infty r^{q-1} H(r)\dif r.
        \end{gathered}
    \end{equation*}
\end{lemma}

The following is a standard reduction:
\begin{lemma}\label{BallLemma3}
    It is enough to prove Theorem \ref{KBall_ineq} for $F$ and $G$ bounded.
\end{lemma}

Before proving Lemmas \ref{BallLemma1}--\ref{BallLemma3}, let us show how they imply Theorem \ref{KBall_ineq}.  For a function $E:(0,\infty)\to [0,\infty)$, changing the order of integration,
\begin{equation}\label{Eeq}
    \int_0^\infty E(u)\dif u= \int_0^\infty \int_0^{E(u)}\dif z\dif u= \int_0^{\|E\|_\infty} \int_{\{u: E(u)\geq z\}} \dif u\dif z= \int_0^{\|E\|_\infty} |E\geq z|\dif z,
\end{equation}
where $\|E\|_\infty$ could potentially be infinite. 
Ball applies the 1-dimensional Brunn--Minkowski inequality to the sets $\{E\geq z\}$. 
\begin{proof}[Proof of Theorem \ref{KBall_ineq}.]
\noindent
\textit{Step 1: The setup.} Let 
\begin{equation*}
    a\defeq \left(\int_0^\infty t^{q-1}F(t)\dif t \right)^{\frac1q}, \quad b\defeq \left( \int_0^\infty s^{q-1} G(s)\dif s\right)^{\frac1q}, \quad c\defeq \left( \int_0^\infty r^{q-1} H(r)\dif r\right)^{\frac1q}. 
\end{equation*}
The aim is to show $\frac2c\leq \frac1a+ \frac1b$, or equivalently, 
\begin{equation}
    c\geq \frac{2ab}{a+b}. 
\end{equation}
By Lemma \ref{BallLemma2} and \eqref{Eeq},
\begin{gather}
\label{Fint}
        a^q= \int_0^\infty A(u)\dif u= \int_0^{\|A\|_\infty} |A\geq z|\dif z, \\
\label{Gint}
        (\theta b)^q= \int_0^\infty B(v)\dif v= \int_0^{\|B\|_\infty}|B\geq z|\dif z, \\
\label{Hint}
        \left( \frac{\theta+1}{2}\right)^{q+1} c^q= \int_0^\infty C(w)\dif w= \int_0^{\|C\|_\infty} |C\geq z|\dif z. 
\end{gather}

\noindent
\textit{Step 2: Comparing the superlevel sets.} Lemma \ref{BallLemma1} allows to compare the superlevel sets of $A, B$ and $C$, obtaining an inequality between $a, b$ and $c$. In particular,
\begin{equation}\label{SublevelEq}
    \{C\geq z\}\supset \frac12\{A\geq z\}+ \frac\theta2\{B\geq z\}, 
\end{equation}
because for $u\in \{A\geq z\}$ and $v\in \{B\geq z\}$, $C(\frac{u+\theta v}{2})\geq A(u)^{\frac{u}{u+\theta v}} B(v)^{\frac{\theta v}{u+\theta v}}\geq z^{\frac{u}{u+\theta v}} z^{\frac{\theta v}{u+\theta v}}=z$, i.e., $\frac{u+\theta v}{2}\in \{C\geq z\}$. By the 1-dimensional Brunn--Minkowski inequality, 
\begin{equation}\label{SubLevelEq2}
    |C\geq z|\geq \frac12 |A\geq z|+ \frac\theta2|B\geq z|. 
\end{equation}
By \eqref{Eeq} and \eqref{SubLevelEq2}, 
\begin{equation}\label{CqIneq}
    \begin{aligned}
        \left(\frac{\theta+1}{2}\right)^{q+1} c^q &= \int_0^{\|C\|_\infty} |C\geq z|\dif z\geq \frac12\int_0^{\|C\|_\infty} |A\geq z|\dif z+ \frac\theta2\int_0^{\|C\|_\infty} |A\geq z|\dif z. 
    \end{aligned}
\end{equation}

\noindent
\textit{Step 3: Choosing $\theta$.} By Lemma \ref{BallLemma1}, $\|C\|_\infty\geq \min\{\|A\|_\infty, \|B\|_{\infty}\}$. In view of \eqref{Fint}, \eqref{Gint} and \eqref{CqIneq}, we would like $\|C\|_\infty\geq \max\{\|A\|_\infty, \|B\|_\infty\}$. The only way to achieve this is to have $\|A\|_\infty= \|B\|_\infty$. It is here that one needs to take take $F$ and $G$ bounded so that $\|A\|_\infty$ and  $\|B\|_\infty$ are finite. By Lemma \ref{BallLemma3}, there is no loss in making such an assumption. Choosing
\begin{equation*}
    \theta \defeq \left(\frac{\sup_{r>0} F(r)r^{q+1}}{\sup_{r>0} G(r)r^{q+1}}\right)^{\frac{1}{q+1}}, 
\end{equation*}
gives
\begin{equation*}
    \|A\|_\infty= \sup_{r>0} F(r)r^{q+1}= \sup_{r>0} G(r)(\theta r)^{q+1}= \sup_{u>0} G(\theta^{-1}u^{-1}) u^{-(q+1)}=  \|B\|_\infty. 
\end{equation*}

\noindent
\textit{Step 4: Finishing the proof.} By Lemma \ref{BallLemma1} and the choice of $\theta$, $\|C\|_\infty\geq \|A\|_\infty= \|B\|_\infty$. By \eqref{Fint}, \eqref{Gint}, and \eqref{CqIneq}, 
\begin{equation*}
    \left( \frac{\theta+1}{2}\right)^{q+1} c^q\geq \frac12\int_0^{\|A\|_\infty} |A\geq z|\dif z+ \frac{\theta}{2}\int_0^{\|B\|_\infty} |B\geq z|\dif z= \frac{a^q+ \theta^{q+1} b^q}{2}. 
\end{equation*} 
That is, 
\begin{equation*}
    \begin{aligned}
        c^q&\geq \left( \frac{2}{\theta+1}\right)^q \left( \frac{1}{1+\theta} a^q+ \frac{\theta}{1+\theta} (\theta b)^q\right)\geq \left( \frac{2}{\theta+1}\right)^{q} \left( \frac{1}{\theta+1} a+ \frac{\theta}{\theta+1}\theta b\right)^q, 
    \end{aligned}
\end{equation*}
because for $q\geq 1$, $x\mapsto x^q$ is convex and hence $(1-\lambda)x^q+ \lambda y^q\geq \left((1-\lambda)x+ \lambda y\right)^q$ for all $x,y\geq 0$ and $\lambda\in[0,1]$. Finally, 
\begin{equation*}
\begin{aligned}
    c&\geq \frac{2(a+ \theta^2 b)}{(\theta+1)^2}
    = \frac{2(a+b)(a+ \theta^2 b)}{(a+b)(\theta+1)^2}
    = 2\frac{a^2+ \theta^2 ab+ ab+ \theta^2 b^2}{(a+b)(\theta+1)^2}\\
    &= 2\frac{(\theta^2+1)ab+ a^2+ \theta^2 b^2}{(a+b)(\theta+1)^2}
    = 2\frac{(\theta+1)^2 ab - 2\theta ab + a^2+ \theta^2 b^2}{(a+b)(\theta+1)^2}\\
    &= 2\frac{(\theta+1)^2 ab+ (a-\theta b)^2}{(a+b)(\theta+1)^2}
    = \frac{2ab}{a+b}+ \frac{2(a-\theta b)^2}{(a+b)(\theta+1)^2}
    \geq \frac{2ab}{a+b}, 
\end{aligned}
\end{equation*}
as desired.
This concludes the proof of Theorem \ref{KBall_ineq},
modulo the proofs of 
Lemmas \ref{BallLemma1}--\ref{BallLemma3}, which are given below.
\end{proof}

\begin{proof}[Proof of Lemma \ref{BallLemma1}.]
For $t,s,r>0$ with $\frac2r= \frac1t+ \frac1s$, by assumption, 
\begin{equation}\label{hFGineq}
\begin{aligned}
    H(r)\geq F(t)^{\frac{s}{t+ s}} G(s)^{\frac{t}{t+ s}}&= \left( A(t^{-1}) t^{-(q+1)}\right)^{\frac{ s}{t+ s}} \left( B(\theta^{-1}s^{-1}) (\theta s)^{-(q+1)}\right)^{\frac{t}{t+ s}} \\
    &= A(u)^{\frac{s}{t+s}} B(v)^{\frac{t}{t+s}} \left( t^{\frac{s}{t+s}}  (\theta s)^{\frac{t}{t+s}} \right)^{-(q+1)}.
\end{aligned}
\end{equation}
 Since, $\frac{s}{t+s}= \frac{u}{u+\theta v}$ and $\frac{t}{t+s}= \frac{\theta v}{u+\theta v}$, by \eqref{hFGineq} and the weighted AM--GM, 
\begin{equation}
    \begin{aligned}
        A(u)^{\frac{u}{u+\theta v}} B(v)^{\frac{\theta v}{u+\theta v}}= A(u)^{\frac{s}{t+s}} B(v)^{\frac{t}{t+s}}&\leq H(r) \left(t^{\frac{s}{t+s}} (\theta s)^{\frac{t}{t+s}}\right)^{q+1} \\
        &\leq H(r)\left(\frac{ts}{t+s}+ \frac{\theta st}{t+s}\right)^{q+1}\\
        &= \left( \frac{\theta+1}{2}\right)^{q+1} H(r) \left( \frac{2ts}{t+s}\right)^{q+1}\\
        &= \left(\frac{\theta+1}{2}\right)^{q+1} H(r) r^{p+1}
        = C\left(\frac{u+\theta v}{2}\right), 
    \end{aligned}
\end{equation}
because $\frac{1}{r}= \frac{1}{2t}+ \frac{1}{2s}= \frac12 u+ \frac12 (\theta v)= \frac{u+\theta v}{2}$.
\end{proof}

\begin{proof}[Proof of Lemma \ref{BallLemma2}.]
By changing variables, $u=\frac1t$,
\begin{equation*}
    \int_0^\infty A(t)\dif t= \int_0^\infty u^{-(q+1)} F(u^{-1})\dif u= \int_0^\infty t^{q+1} F(t) \frac{\dif t}{t^{2}}= \int_0^\infty t^{q-1} F(t)\dif t. 
\end{equation*}
For $v= \frac{1}{\theta s}$, 
\begin{equation*}
    \int_0^\infty B(v)\dif v= \int_0^\infty v^{-(q+1)} G(\theta^{-1}v^{-1}) \dif v= \int_0^\infty (\theta s)^{q+1} G(s) \frac{\dif s}{\theta s^2}= \theta^q \int_0^\infty s^{q-1} G(s)\dif s. 
\end{equation*}
Finally, for $w= \frac1r$, 
\begin{equation*}
\begin{aligned}
    \int_0^\infty C(w)\dif w= \left( \frac{\theta+1}{2}\right)^{q+1} \int_0^\infty w^{-(q+1)} H(w^{-1})\dif w&= \left( \frac{\theta+1}{2}\right)^{q+1}\int_0^\infty r^{q+1} H(r)\frac{\dif r}{r^2} \\
    &= \left( \frac{\theta+1}{2}\right)^{q+1}\int_0^\infty r^{q-1} H(r)\dif r. 
\end{aligned}
\end{equation*}
\end{proof}

\begin{proof}[Proof of Lemma \ref{BallLemma3}.]
For $m\in\mathbb{N}$, let
\begin{equation*}
    F_m(t)\defeq F(t)\bm{1}_{\{F\leq m\}}(t), \quad \text{ and } \quad G_m(s)\defeq G(s)\bm{1}_{\{G\leq m\}}(s). 
\end{equation*}
Then, $F_m, G_m$ are both bounded by $m$. In addition, $F\geq F_m$ and $G\geq G_m$, thus for $t,s,r>0$ with $\frac2r=\frac1t+ \frac1s$, 
\begin{equation*}
    H(r)\geq F(t)^{\frac{s}{t+s}} G(s)^{\frac{t}{t+s}}\geq F_m(t)^{\frac{s}{t+s}} G_{m}(s)^{\frac{t}{t+s}}. 
\end{equation*}
Under the assumption that Theorem \ref{KBall_ineq} holds for bounded bounded functions,  
\begin{equation*}
    2\left(\int_0^\infty r^{q-1} H(r)\dif r\right)^{-\frac1q}\leq \left(\int_0^\infty t^{q-1} F_m(t)\dif t \right)^{-\frac1q}+ \left(\int_0^\infty s^{q-1}G_m(s)\dif s \right)^{-\frac1q}. 
\end{equation*}
The claim follows from the monotone convergence theorem by taking $m\to\infty$. 
\end{proof}

\bigskip

{\sc Chalmers University  of Technology}

{\tt bob@chalmers.se} 

\medskip
 
{\sc University of Maryland}

{\tt vmastr@umd.edu, yanir@alum.mit.edu}

\end{document}